%% file: SurveyEG_March2023_V7.tex
\documentclass{article}

\usepackage{microtype}
\usepackage{graphicx}
\usepackage{subfigure}
\usepackage{hyperref}

\input{my_preamble}

\title{Sublinear Convergence Rates of Extragradient-Type Methods: A Survey on Classical and Recent Developments}

\author{Quoc Tran-Dinh\vspace{0.25ex}\\
\newline {Department of Statistics and Operations Research}\\
\newline The University of North Carolina at Chapel Hill\\
318 Hanes Hall, UNC-Chapel Hill, NC 27599-3260.\\
\newline \textit{Email:} \url{quoctd@email.unc.edu}.}

\date{Version 1 (\textit{March 30, 2023})}
\begin{document}
\maketitle

\begin{abstract}
%\normalfont
%\normalsize
The extragradient (EG), introduced by G. M. Korpelevich in 1976, is a well-known method to approximate solutions of saddle-point problems and their extensions such as variational inequalities and monotone inclusions.
Over the years, numerous variants of EG have been proposed and studied in the literature. 
Recently, these methods have gained popularity due to new applications in machine learning and robust optimization. 
In this work, we survey the latest developments in the EG method and its variants for approximating solutions of nonlinear equations and inclusions, with a focus on the monotonicity and co-hypomonotonicity settings. 
We provide a unified convergence analysis for different classes of algorithms, with an emphasis on sublinear best-iterate and last-iterate convergence rates. 
We also discuss recent accelerated variants of EG based on both Halpern fixed-point iteration and Nesterov's accelerated techniques.
Our approach uses simple arguments and basic mathematical tools to make the proofs as elementary as possible, while maintaining generality to cover a broad range of problems.
\end{abstract}

%\beforesec
\section{Introduction}\label{sec:intro}
\aftersec
The \myti{generalized equation} (also called the \myti{[non]linear inclusion}) provides a unified template to model various problems in computational mathematics and related fields such as the optimality condition of optimization problems (in both unconstrained and constrained settings), minimax optimization, variational inequality, complementarity, two-person game,  and fixed-point problems, see, e.g., \cite{Bauschke2011,reginaset2008,Facchinei2003,phelps2009convex,Rockafellar2004,Rockafellar1976b,ryu2016primer}.
Theory and numerical methods for this equation and its special cases have been extensively studied for many decades, see, e.g., the following monographs and the references quoted therein \cite{Bauschke2011,Facchinei2003,minty1962monotone,Rockafellar1997}.
At the same time, several applications of this mathematical tool in operations research, economics, uncertainty quantification, and transportations have been investigated \cite{Ben-Tal2009,giannessi1995variational,harker1990finite,Facchinei2003,Konnov2001}.
In the last few years, there has been a surge of research in minimax problems due to new applications in machine learning and robust optimization, especially in generative adversarial networks (GANs), adversarial training, and distributionally robust optimization, see, e.g., \cite{arjovsky2017wasserstein,Ben-Tal2009,goodfellow2014generative,levy2020large,madry2018towards,rahimian2019distributionally} as a few examples. 
Minimax problems have also found new applications in online learning and reinforcement learning, among many others, see, e.g.,   \cite{arjovsky2017wasserstein,azar2017minimax,bhatia2020online,goodfellow2014generative,jabbar2021survey,levy2020large,lin2022distributionally,madry2018towards,rahimian2019distributionally,wei2021last}.
Such prominent applications have motivated the research in minimax optimization and variational inequality problems (VIPs).
On the one hand, classical algorithms such as gradient descent-ascent, extragradient, and primal-dual methods have been revisited, improved, and extended.
On the other hand, new variants such as accelerated extragradient and accelerated operator splitting schemes have also been developed and equipped with rigorous convergence guarantees and practical performance evaluation.
This new development motivates us to write this survey paper, with the focus on sublinear  convergence rate analysis.

%%%% Problem statements.
\vspace{0.5ex}
\noindent\textbf{Problem statements.}
Since there is a vast amount of literature on the generalized equation, we will only present the recent developments on sublinear convergence rates of the \mytb{extragradient (EG)} method and its variants for approximating the solutions of the following \myti{generalized equation} (also known as a [composite] \myti{nonlinear inclusion}) and its specific cases:
\begin{equation}\label{eq:NI}
\textrm{Find $x^{\star}\in\dom{\Phi }$ such that:} \quad 0 \in \Phi x^{\star} \equiv Fx^{\star} + Tx^{\star},
\tag{NI}
\end{equation}
where $F : \R^p\to\R^p$ is a single-valued operator, $T : \R^p\rightrightarrows 2^{\R^p}$ is a set-valued (or multivalued) mapping from $\R^p$ to $2^{\R^p}$ (the set of all subsets of $\R^p$), $\Phi := F + T$, and $\dom{\Phi } := \dom{F}\cap\dom{T}$ is the domain of $\Phi$, which is the intersection of the domains of $F$ and $T$.
In this paper, we focus on the finite-dimensional Euclidean spaces $\R^p$ and $\R^n$ for ease of presentation. 
However, it is worth noting that most of the results presented in this paper can be extended to Hilbert spaces, as demonstrated in the existing literature.

%%%% Special cases.
\vspace{0.5ex}
\noindent\textbf{Special cases.}
If $F = 0$, then \eqref{eq:NI} reduces to a \myti{generalized equation} or a \myti{[non]linear inclusion} $0 \in Tx^{\star}$.
Alternatively, if $T = 0$, then \eqref{eq:NI} reduces to a \myti{[non]linear equation}:
\begin{equation}\label{eq:NE}
\textrm{Find $x^{\star}\in\dom{F}$ such that:} \quad Fx^{\star} = 0.
\tag{NE}
\end{equation}
If $T := \partial{g}$, the subdifferential of a proper, closed, and convex function $g : \R^p \to \Rext$, then \eqref{eq:NI} reduces a \myti{mixed variational inequality problem} (MVIP):
\begin{equation}\label{eq:MVIP}
\textrm{Find $x^{\star}\in\dom{\Phi}$ such that:} \quad \iprods{Fx^{\star}, x - x^{\star}} + g(x) - g(x^{\star}) \geq 0, \ \textrm{for all} \ x \in \dom{\Phi}.
\tag{MVIP}
\end{equation}
In particular, if $T = \Nc_{\Xc}$, the normal cone of a nonempty, closed, and convex set $\Xc$ in $\R^p$ (i.e. $g = \delta_{\Xc}$, the indicator of $\Xc$), then \eqref{eq:MVIP} reduces the classical (Stampacchia) \myti{variational inequality problem} (VIP):
\begin{equation}\label{eq:VIP}
\textrm{Find $x^{\star}\in\Xc$ such that:} \quad \iprods{Fx^{\star}, x - x^{\star}} \geq 0, \ \textrm{for all} \ x \in \Xc.
\tag{VIP}
\end{equation}
While \eqref{eq:VIP} can be viewed as a primal VIP (or a strong VIP), its dual (or weak) form can be written as
\begin{equation}\label{eq:DVIP}
\textrm{Find $x^{\star}\in\Xc$ such that:} \quad \iprods{Fx, x - x^{\star}} \geq 0, \ \textrm{for all} \ x \in \Xc,
\tag{DVIP}
\end{equation}
which is known as Minty's variational inequality problem.
If $F$ is monotone (see the definition in Section~\ref{sec:background}) then both problems \eqref{eq:VIP} and \eqref{eq:MVIP} are equivalent, i.e. their solution sets are identical, \cite{Facchinei2003,Konnov2001}.
One important special case of \eqref{eq:NI} or \eqref{eq:VIP} is the optimality condition of minimax problems of the form:
\begin{equation}\label{eq:minimax_prob}
\min_{u \in\R^m} \max_{v\in \R^n} \Big\{ \Lc(u, v) := \varphi(u) + \Hc(u, v) - \psi(v) \Big\},
\end{equation}
where $\varphi : \R^m\to\Rext$ and $\psi : \R^n\to\Rext$ are often proper, closed, and convex functions, and $\Hc : \R^m\times\R^n\to\R$ is a bifunction, often assumed to be differentiable, but not necessarily convex-concave.
If we denote $x := [u, v]$ as the concatenation of $x$ and $y$, and define $T := [\partial{\varphi}, \partial{\psi}]$ and $F := [\nabla_u{\Hc}(u, v), -\nabla_v{\Hc}(u, v)]$, then the optimality condition of \eqref{eq:minimax_prob} is exactly captured by \eqref{eq:NI}.

\vspace{1ex}
\noindent\textbf{Related work.}
Extensive research has been conducted in the literature to investigate the existence of solutions and theoretical properties of \eqref{eq:NI} and its special cases.
This research has been conducted  under various assumptions of monotonicity and extensions, including quasi-monotone, pseudo-monotone, and weakly monotone notions. 
Relevant literature references on this topic include \cite{Bauschke2011,Facchinei2003,Konnov2001,mordukhovich2006variational,Zeidler1984}.
Moreover, solution methods for \eqref{eq:NI} and its special cases have been well-developed, particularly in the context of monotonicity and related extensions such as quasi-monotone, pseudo-monotone, or star-monotone notion. 
In addition, nonmonotone instances of \eqref{eq:NI} have also received extensive attention in the literature, with many theoretical results and algorithms focusing on local properties. 
Additional information can be found in references such as \cite{bauschke2020generalized,Bonnans1994a,Bonnans2000a,combettes2004proximal,Dontchev1996,pennanen2002local,Robinson1980,Rockafellar1997}.

%The existence of solutions as well as theoretical properties of \eqref{eq:NI} and its special cases have been widely studied in the literature under different assumptions of monotonicity and extensions such as quasi-monotone, pseudo-monotone, and weakly monotone notions, see, e.g., \cite{Bauschke2011,Facchinei2003,Konnov2001,mordukhovich2006variational,Zeidler1984}.
%Alternatively, solution methods for solving \eqref{eq:NI} and its special cases are well developed, especially under the monotone settings and close extensions (e.g., quasi-monotone or pseudo-monotone).
%Nevertheless, nonmonotone instances of \eqref{eq:NI} have also received significant attention in the literature, with many theoretical results and algorithms focusing on local properties, see, e.g., \cite{bauschke2020generalized,Bonnans1994a,Bonnans2000a,combettes2004proximal,Dontchev1996,pennanen2002local,Robinson1980,Rockafellar1997}.

Existing solution methods  for \eqref{eq:NI} and its special cases often rely on a fundamental assumption: \textit{maximal monotonicity} of $F$ and $T$, or of $\Phi$ to guarantee global convergence.
These methods generally generalize  existing optimization algorithms such as gradient, proximal-point, Newton, and interior-point schemes to \eqref{eq:NI} and its special cases \cite{combettes2004proximal,Facchinei2003,Fukushima1992,Martinet1970,Peng1999,Rockafellar1976b,TranDinh2016c,vuong2018weak}, while leveraging the splitting structure of \eqref{eq:NI} to use individual operators defined on $F$ and $T$.
This approach leads to a class of splitting algorithms for solving \eqref{eq:NI} such as forward-backward splitting (FBS) and Douglas-Rachford (DRS) splitting schemes, as seen in \cite{Bauschke2011,Combettes2005,Davis2015,eckstein1992douglas,lin2020near,Lions1979}.
Alternatively, other approaches rely on primal-dual, dual averaging, and mirror descent techniques, with notable works including \cite{Chambolle2011,Nemirovskii2004,Nesterov2007a}.
These methods have also been further studied in many recent works such as \cite{Chen2013a,chen2017accelerated,Cong2012,Davis2014,Esser2010a,He2012b,Nesterov2006d,TranDinh2015b,tran2019non,ZhuLiuTran2020}.

When it comes to convergence analysis for gradient-based/forward methods, there is a fundamental challenge for generalized equation \eqref{eq:NI} because an objective function, which plays a central role in guaranteeing convergence for optimization problems, does not exist. 
This creates a significant challenge, particularly in nonmonotone settings. 
Additionally, unlike convex functions where strong properties such as coerciveness and cyclic monotonicity hold for their [sub]gradients beyond monotonicity, this is not the case for general monotone and Lipschitz continuous operators. 
This lack of a strong property results in gradient-based (or forward) methods being non-convergent, which limits their practicality, see, e.g.,  \cite{Facchinei2003}. 
To address this issue, the extragradient (EG) method was introduced by G. M. Korpelevich in 1976  \cite{korpelevich1976extragradient} and also by A. S. Antipin in \cite{antipin1976}. 
This method performs two sequential gradient steps at each iteration, making it twice as expensive as the standard gradient method, but converges under only the monotonicity  and the Lipschitz continuity of $F$. 
Since then, this method has been extended and modified in different directions to reduce its per-iteration complexity, including in certain nonmonotone settings, see, e.g., \cite{alacaoglu2022beyond,censor2011subgradient,censor2012extensions,Iusem1997,khobotov1987modification,malitsky2015projected,malitsky2020forward,malitsky2014extragradient,Monteiro2010a,Monteiro2011,popov1980modification,solodov1999hybrid,solodov1999new,solodov1996modified,tseng1990further,tseng2000modified}. 
Among these variants of EG, the past-extragradient scheme in  \cite{popov1980modification} and Teng’s forward-backward-forward splitting method in \cite{tseng2000modified} are the most notable ones. 
However, the results discussed here are only applicable to the monotone setting of \eqref{eq:NI} and its special cases. 
Additionally, most of the convergence results discussed are asymptotic, leading to sublinear “best-iterate” convergence rates of the residual norm associated with \eqref{eq:NI}. 
Under stronger assumptions such as ``strong monotonicity'', linear convergence rates can be achieved. 
Such types of convergence guarantees have been widely studied in the literature and are beyond the scope of this paper, see, e.g., \cite{Bauschke2011,Facchinei2003,Konnov2001}.

Motivated by recent applications in machine learning and robust optimization, such as Generative Adversarial Networks (GANs), adversarial training, distributionally robust optimization, reinforcement learning, and online learning, several methods for solving minimax problems have become critically important and attractive. 
This is particularly true in nonconvex-nonconcave, large-scale, and stochastic settings, as evidenced in works such as \cite{arjovsky2017wasserstein,azar2017minimax,Ben-Tal2009,bhatia2020online,goodfellow2014generative,jabbar2021survey,levy2020large,lin2022distributionally,madry2018towards,rahimian2019distributionally}.
Several researchers have proposed and revisited EG and its variants, including \cite{bohm2022solving,daskalakis2018training,diakonikolas2021efficient,pethick2022escaping}. 
A notable work is due to \cite{diakonikolas2021efficient}, where the authors proposed an EG-plus (EG+) variant of EG, capable of handling nonmonotone instances of \eqref{eq:NE}, known as weak-Minty solutions. 
In \cite{pethick2022escaping}, this method was further extended to \eqref{eq:NI}, while \cite{bohm2022solving,luo2022last}  modified EG+ for Popov's methods, as well as optimistic gradient variants.

In contrast to classical methods, there has been a significant focus on developing accelerated methods for solving \eqref{eq:NI} and its special cases under both monotone and co-hypomonotone structures. 
Early works in this area relied on dual averaging and mirror descent techniques such as those proposed in \cite{Cong2012,Nemirovskii2004,Nesterov2007a}, which require the monotonicity or specific assumptions. 
Attouch et al \cite{attouch2019convergence}  proposed accelerated proximal-point methods for solving \eqref{eq:NI} under the maximal monotonicity of $\Phi$.
Since then, numerous works have followed up and explored Nesterov's acceleration-type methods guided by dynamical systems, utilizing momentum and correction terms for solving \eqref{eq:NI} under monotone assumptions, as demonstrated in works such as \cite{attouch2020convergence,bot2022fast,bot2022bfast,kim2021accelerated,mainge2021fast,mainge2021accelerated}. 
Accelerated methods based on Halpern's fixed-point iteration \cite{halpern1967fixed} have also gained popularity. 
Although initially developed to approximate a fixed-point of a nonexpansive operator, this method can be applied to solve \eqref{eq:NE}, \eqref{eq:VIP}, and \eqref{eq:NI} under monotonicity. 
In \cite{lieder2021convergence}, it was shown that Halpern's fixed-point iteration can achieve $\BigOs{1/k}$ last-iterate convergence rates using a specific choice of parameters, where $k$ is the iteration counter.
The authors in  \cite{diakonikolas2020halpern} further exploited this approach to solve monotone VIPs of the form \eqref{eq:VIP}.
Yoon and Ryu extended Halpern's fixed-point iteration idea to EG methods to solve \eqref{eq:NE} without the co-coerciveness assumption on $F$ in their pioneering work  \cite{yoon2021accelerated}. 
Lee and Kim \cite{lee2021fast}  proposed a similar algorithm for solving \eqref{eq:NE} under the co-hypomonotonicity, further advancing  \cite{yoon2021accelerated} without sacrificing the $\BigOs{1/k}$-convergence rates. 
In \cite{tran2021halpern}, the authors proposed a Halpern-type variant for the past-extragradient method in  \cite{popov1980modification} by adopting the technique from \cite{yoon2021accelerated}. 
Recently, \cite{cai2022accelerated,cai2022baccelerated} extended  \cite{yoon2021accelerated} and \cite{tran2021halpern} to \eqref{eq:VIP} and \eqref{eq:NI} under either monotonicity or co-hypomonotonicity assumptions. 
New convergence analysis for these schemes can also be found in \cite{tran2023extragradient}. 
Note that both Halpern's fixed-point iteration and Nesterov's accelerated schemes for solving \eqref{eq:NE} and \eqref{eq:NI} are related to each other, as shown in \cite{tran2022connection}  for different methods, including EG. 
Nesterov's accelerated variants of EG can also be found in \cite{tran2023extragradient}.

\vspace{0.5ex}
\noindent\textbf{What does this paper survey?}
Our main objective is to provide a comprehensive survey of both classical and recent sublinear convergence rate results for EG and its variants for solving \eqref{eq:NI} and its special cases, as summarized in Table \ref{tbl:survey_results}. 
Specifically, we survey the following results.
%
%Our main goal is to survey both classical and recent sublinear convergence rate results for EG and its variants to solve \eqref{eq:NI} as well as its special cases.
%These results are summarized in Table \ref{tbl:survey_results}.
%
\begin{table}[ht!]
\vspace{-1ex}
\newcommand{\cell}[1]{{\!\!}{#1}{\!\!}}
\begin{center}
\caption{Summary of the results surveyed in this paper and the most related references}\label{tbl:survey_results}
\vspace{-1ex}
\begin{small}
\resizebox{\textwidth}{!}{  
\begin{tabular}{|c|c|c|c|c|} \hline
\cell{Methods} & \cell{Assumptions} & \cell{Add. Assumptions} & \cell{Convergence Rates} & \cell{Citations} \\ \hline
\multicolumn{5}{|c|}{ For solving \eqref{eq:NE}} \\ \hline
\mytb{EG}/\mytb{EG+}/\mytb{FBFS} & \mytb{wMs}  & $F$ is \mytb{chm} & $\BigOs{1/\sqrt{k}}$ best and last   & \cite{diakonikolas2021efficient,golowich2020last,gorbunov2022extragradient,luo2022last} \\ \hline
\mytb{PEG}/\mytb{OG}/\mytb{FRBS}/\mytb{RFBS}/\mytb{GR} & \mytb{wMs} & $F$ is \mytb{chm}  & $\BigOs{1/\sqrt{k}}$ best and last & \cite{bohm2022solving,luo2022last} \\ \hline
\mytb{EAG}/\mytb{FEG}/\mytb{AEG} & $F$ is \mytb{chm} & None & $\BigOs{1/k}$ last-iterate &  \cite{yoon2021accelerated,kim2021accelerated,tran2022connection} \\ \hline
\mytb{PEAG}/\mytb{APEG} &  $F$ is \mytb{chm} & None & $\BigOs{1/k}$ last-iterate &  \cite{tran2021halpern,tran2022connection} \\ \hline
%%%%
\multicolumn{5}{|c|}{ For solving \eqref{eq:NI}, \eqref{eq:MVIP}, and \eqref{eq:VIP}} \\ \hline
%%%%
\mytb{EG}/\mytb{EG+} & $\Phi$ is \mytb{wMs} & $F$ is \mytb{mono}, $T$ is \mytb{3-cm} & $\BigOs{1/\sqrt{k}}$ best and last &   \cite{cai2022tight} \\ \hline
\mytb{FBFS}  & $\Phi$ is \mytb{wMs} & None & $\BigOs{1/\sqrt{k}}$ best-iterate & \cite{pethick2022escaping,luo2022last} \\ \hline
\mytb{OG}/\mytb{FRFS} & $\Phi$ is \mytb{wMs} & None & $\BigOs{1/\sqrt{k}}$ best-iterate &   \cite{luo2022last} \\ \hline
\mytb{RFBS} & $F$ is \mytb{mono} & $T$ is \mytb{mono} &   $\BigOs{1/\sqrt{k}}$ best and last &  \cite{cevher2021reflected,malitsky2015projected} \\ \hline
\mytb{GR} & $F$ is \mytb{mono} & $T$ is \mytb{$3$-cm} &   $\BigOs{1/\sqrt{k}}$ best-iterate &  \cite{malitsky2019golden} \\ \hline
\mytb{EAG}/\mytb{FEG}/\mytb{AEG} & $\Phi$ is \mytb{chm} & None & $\BigOs{1/k}$ last-iterate & \cite{cai2022accelerated,tran2023extragradient} \\ \hline
\mytb{PEAG}/\mytb{APEG} & $\Phi$ is \mytb{chm} & None & $\BigOs{1/k}$ last-iterate &   \cite{cai2022baccelerated,tran2023extragradient} \\ \hline
\end{tabular}}
\end{small}
\\\vspace{1ex}
{\footnotesize
\textbf{Abbreviations:} 
\mytb{EG} $=$ extragradient; 
\mytb{PEG} $=$ past extragradient; 
\mytb{FBFS} $=$ forward-backward-forward splitting;
\mytb{OG} $=$ optimistic gradient;
\mytb{FRBS} $=$ forward-reflected-backward splitting;
\mytb{RFBS} $=$ reflected-forward-backward splitting;
\mytb{GR} $=$ golden ratio; 
\mytb{EAG} $=$ extra-anchored gradient;
\mytb{FEG} $=$ fast extragradient;
\mytb{PEAG} $=$ past extra-anchored gradient;
\mytb{AEG} $=$ Nesterov's accelerated extragradient; 
and \mytb{APEG} $=$ Nesterov's accelerated past extragradient.
In addition, \mytb{wMs} $=$ weak-Minty solution;
\mytb{mono} $=$ monotone;
\mytb{chm} $=$ co-hypomonotone;
and \mytb{3-cm} $=$ $3$-cyclically monotone.
}
\end{center}
\vspace{-3ex}
\end{table}
\begin{compactitem}
\item First, we present both the $\BigOs{1/\sqrt{k}}$-best-iterate and last-iterate sublinear convergence rate results of EG and its variants for solving \eqref{eq:NE}. 
The best-iterate rate is classical for the monotone case, but has recently been obtained under a weak Minty solution condition, see \cite{diakonikolas2021efficient,pethick2022escaping} for EG and  \cite{bohm2022solving,luo2022last}  for past-EG in the non-composite case, i.e., for solving \eqref{eq:NE}. 
The last-iterate convergence rates for EG and past-EG have been recently proven in \cite{golowich2020last,gorbunov2022extragradient} for the monotone equation \eqref{eq:NE} and in  \cite{luo2022last} for the co-hypomonotone case (see also \cite{gorbunov2022convergence}). 
In this paper, we provide a new and unified proof that covers the results in \cite{golowich2020last,gorbunov2022extragradient,luo2022last}.
Our results are stated in a single theorem.
%%%
\item Second, we review the $\BigOs{1/\sqrt{k}}$-sublinear best-iterate convergence rates for EG and past-EG (also known as Popov's method) to solve \eqref{eq:NI} under the monotonicity of $\Phi$. 
We unify the proof of both methods in a single theorem and extend it to cover monotone inclusions of the form \eqref{eq:NI} instead of \ref{eq:VIP} or \ref{eq:MVIP} as in the literature. 
We also prove $\BigOs{1/\sqrt{k}}$ last-iterate convergence for the class of EG-type schemes for solving \eqref{eq:NI} under the monotonicity of $F$ and the $3$-cyclical monotonicity of $T$ (in particular, for solving \eqref{eq:MVIP}), which covers the results in  \cite{cai2022tight} as special cases. 
Next, we discuss the $\BigOs{1/\sqrt{k}}$-sublinear best-iterate convergence rates of the FBFS scheme and its variant: optimistic gradient under the weak-Minty solution notion, which was obtained in \cite{luo2022last}.
We again unify the proof in a single theorem, and our analysis is also different from \cite{luo2022last}.

\item  Third, we provide a new convergence analysis for both best-iterate and last-iterate rates of the reflected forward-backward splitting (RFBS) methods for solving \eqref{eq:NI} under the monotonicity of $F$ and $T$. 
RFBS was proposed in \cite{malitsky2015projected} to solve \eqref{eq:VIP} and was extended to solve \eqref{eq:NI} in  \cite{cevher2021reflected}. 
The best-iterate rates were proven in these works, and the last-iterate rate of RFBS for solving \eqref{eq:VIP} has recently been proven in \cite{cai2022baccelerated}. 
Our result here is more general and covers these works as special cases.
In addition, we also review the best-iterate convergence rate of the golden ration method in \cite{malitsky2019golden}, but extend it to the case $T$ is $3$-cyclically monotone, and extend the range of the golden-ratio parameter $\omega$ to $1 < \omega < 1 + \sqrt{3}$ instead of fixing it at $\omega := \frac{1+\sqrt{5}}{2}$ as in \cite{malitsky2019golden}.
\item Fourth, we present a new analysis for the extra-anchored gradient (EAG) method to solve \eqref{eq:NI}, which covers the results in \cite{cai2022accelerated,yoon2021accelerated} as special cases.
Our result extends to $3$-cyclically monotone operator $T$.
\item Fifth, we summarize the convergence results of EAG and past EAG (also called fast extragradient method \cite{kim2021accelerated}) for solving \eqref{eq:NI} under both monotonicity and co-hypomonotonicity of $\Phi$ from \cite{tran2023extragradient}. 
Note that EAG and past-EAG were first proposed in \cite{yoon2021accelerated} and \cite{tran2021halpern}, respectively, to solve monotone \eqref{eq:NE}. 
EAG was extended to the co-hypomonotone case of \eqref{eq:NE} in \cite{kim2021accelerated}. 
Recently, \cite{cai2022accelerated,cai2022baccelerated} extended EAG and past-EAG to solve \eqref{eq:NI} under the co-hypomonotonicity of $\Phi$.
\item Finally, we review two Nesterov’s accelerated extragradient methods presented in \cite{tran2022connection,tran2023extragradient} for solving \eqref{eq:NI} under the co-hypomonotonicity of $\Phi$, which achieve the same last-iterate convergence rates as EAG.
Note that Nesterov's accelerated extragradient methods have recently been studied in \cite{bot2022fast} for solving \eqref{eq:NE} via a dynamical system point of view.
\end{compactitem}

\vspace{0.5ex}
\noindent\textbf{What is not covered in this paper?}
% The literature on EG and its variants is vast, and it is not possible for us to cover it thoughtfully. 
The literature on EG and its variants is extensive, and it is not feasible for us to cover it in detail in this paper. 
 First, there are various classical and recent variants of EG and past-EG, such as those discussed in \cite{censor2011subgradient,censor2012extensions,Iusem1997,khobotov1987modification,malitsky2020forward,malitsky2014extragradient,solodov1999new,solodov1996modified,solodov1999hybrid,vuong2012extragradient}, that are not included in this paper. 
 These methods are essentially rooted from EG with the aim of improving the per-iteration complexity, theoretical aspects, or practical performance. 
 Second, we do not review results from methods such as gradient/forward, forward-backward splitting, proximal-point and its variants, inertial, dual averaging, mirror descent, and projective methods.
 The majority of these methods is not immediately derived from EG, including recent developments such as those in \cite{boct2015hybrid,boct2016inertial,BricenoArias2011,combettes2004proximal,Combettes2011,chen2017accelerated,eckstein2008family,Eckstein2009,lin2018solving}. 
 Third, we do not cover stochastic and randomized methods, including recent works such as [\cite{alacaoglu2021stochastic,davis2022variance,gidel2018variational,hsieh2019convergence,kannan2019optimal,iusem2017extragradient,pethick2023solving,peng2016arock,tran2022accelerated}. 
 Fourth, we do not present adaptive stepsizes/parameters and linesearch variants of EG-type methods.
 Fifth, we do not disuss continuous view of EG-type methods via dynamical systems or ordinary differential equations (ODEs), which is an emerging research topic in recent years. 
 Finally, we also do not cover specific applications to minimax problems and other concrete applications.

\vspace{0.5ex}
\noindent\textbf{Paper outline.}
This paper is organized as follows. 
Section \ref{sec:background}  reviews basic concepts and related results used in this paper. 
Section \ref{sec:EG4NE}  covers the convergence rate results of EG and its variants for solving \eqref{eq:NE}. 
Section \ref{sec:EG4NI} discusses the convergence rate results of EG and past-EG for solving \eqref{eq:NI}. 
Section \ref{sec:FBFS4NI} provides a new convergence rate analysis of FBFS and OG for solving \eqref{eq:NI}. 
Section \ref{sec:other_methods} presents a new analysis for both the reflected-forward-backward splitting and golden ratio methods for solving \eqref{eq:NI}. 
Section \ref{sec:EAG4NI} focuses on the extra-anchored gradient method and its variants for solving \eqref{eq:NI}. 
Finally, Section \ref{sec:AEG4NI} presents Nesterov's accelerated variants of EG for solving \eqref{eq:NI}. 
We conclude this paper with some final remarks.

%%%% 2. Background and Preliminary Results
\beforesec
\section{Background and Preliminary Results}\label{sec:background}
\aftersec
To prepare for our survey, we will briefly review certain basic concepts and properties of monotone operators and their extensions, as well as resolvents and other related mathematical tools.
These concepts and properties are well-known and can be found in several monographs, including \cite{Bauschke2011,reginaset2008,Facchinei2003,phelps2009convex,Rockafellar2004,Rockafellar1970,Rockafellar1976b,ryu2016primer}.

%%%% 2.1. Basic concepts
\beforesubsec
\subsection{Basic concepts, monotonicity, and Lipschitz continuity}\label{subsec:basic_concepts}
\aftersubsec
We work with finite dimensional Euclidean spaces $\R^p$ and $\R^n$ equipped with standard inner product $\iprods{\cdot, \cdot}$ and Euclidean norm $\norms{\cdot}$.
For a set-valued or multivalued mapping $T : \R^p \rightrightarrows 2^{\R^p}$, $\dom{T} = \set{x \in\R^p : Tx \not= \emptyset}$ denotes its domain, $\ran{T} := \bigcup_{x \in \dom{T}}Tx$ is its range, and $\gra{T} = \set{(x, y) \in \R^p\times \R^p : y \in Tx}$ stands for its graph, where $2^{\R^p}$ is the set of all subsets of $\R^p$.
The inverse of $T$ is defined as $T^{-1}y := \sets{x \in \R^p : y \in Tx}$.
For a proper, closed, and convex function $f : \R^p\to\Rext$, $\dom{f} := \sets{x \in \R^p : f(x) < +\infty}$ denotes the domain of $f$, $\partial{f}$ denotes the subdifferential of $f$, and $\nabla{f}$ stands for the gradient of $f$.
For any function $f$, which can be nonconvex, we call $f^{*}(y) := \sup_{x\in\R^p}\sets{\iprods{x, y} - f(x)}$ the Fenchel conjugate of $f$.

%%% a. Monotonicity.
\vspace{0.75ex}
\noindent\textbf{(a)~Monotonicity.}
For a single-valued or multivalued mapping $T : \R^p \rightrightarrows 2^{\R^p}$ and $\mu \in \R$, we say that $T$ is $\mu$-monotone if $\iprods{u - v, x - y} \geq \mu\norms{x - y}^2$ for all $(x, u), (y, v)  \in \gra{T}$.
If $T$ is single-valued, then this condition reduces to $\iprods{Tx - Ty, x - y} \geq \mu\norms{x - y}^2$ for all $x, y\in\dom{T}$.
If $\mu = 0$, then we say that $T$ is monotone.
If $\mu > 0$, then $T$ is  $\mu$-strongly monotone (or sometimes called coercive), where $\mu > 0$ is called a strong monotonicity parameter.
If $\mu < 0$, then we say that $T$ is weakly monotone.
It is also called $-\mu$-hypomonotone, see \cite{bauschke2020generalized}.
If $T = \partial{g}$, the subdifferential of a proper and convex function, then $T$ is also monotone.
If $g$ is $\mu$-strongly convex with $\mu > 0$, then $T = \partial{g}$ is also $\mu$-strongly monotone.

Alternatively, if there exists $\rho \in \R$ such that $\iprods{u - v, x - y} \geq \rho\norms{u - v}^2$ for all $(x, u), (y, v) \in \gra{T}$, then we say that $T$ is $\rho$-comonotone.
If $\rho = 0$, then this condition reduces to the monotonicity of $T$.
If $\rho > 0$, then $T$ is called $\rho$-cocoercive. 
In particular, if $\rho = 1$, then $T$ is firmly nonexpansive.
If $\rho < 0$, then $T$ is called $-\rho$-cohypomonotone, see, e.g., \cite{bauschke2020generalized,combettes2004proximal}.
For a mapping $T$, we say that $T$ is pseudo-monotone if $\iprods{u, y-x} \geq 0$ implies $\iprods{v, y - x} \geq 0$ for all $(x, u),  (y, v) \in \gra{T}$.
Clearly, if $T$ is monotone, then it is also pseudo-monotone, but the conversion is not true in general. 

We say that $T$ is maximally $\mu$-monotone if $\gra{T}$ is not properly contained in the graph of any other $\mu$-monotone operator.
If $\mu = 0$, then we say that $T$ is maximally monotone.
Note that $T$ is maximally monotone, then $\eta T$ is also maximally monotone for any $\eta > 0$, and if $T$ and $U$ are maximally monotone, and $\dom{T}\cap\intx{\dom{U}} \not=\emptyset$, then $T + U$ is maximally monotone.
For a proper, closed, and convex function $f : \R^p\to\Rext$, the subdifferential $\partial{f}$ of $f$ is maximally monotone. 

For a given mapping $T$ such that $\zer{T} := \set{x \in \dom{T} : 0 \in Tx} \neq\emptyset$,  we say that $T$ is star-monotone (respectively, $\mu$-star-monotone or $\rho$-star-comonotone (see \cite{loizou2021stochastic})) if for some $x^{\star} \in\zer{T}$, we have $\iprods{u, x - x^{\star}} \geq 0$ (respectively, $\iprods{u, x - x^{\star}} \geq \mu\norms{x - x^{\star}}^2$ or $\iprods{u, x - x^{\star}} \geq \rho\norms{u}^2$) for all $(x, u) \in \gra{T}$.
Clearly, if $T$ is monotone (respectively, $\mu$-monotone or $\rho$-co-monotone), then it is also star-monotone (respectively, $\mu$-star-monotone or $\rho$-star-comonotone).
However, the reverse statement does not hold in general.

\vspace{0.75ex}
\noindent\textbf{(b)~Cyclic monotonicity.}
We also say that a mapping $T$ is $m$-cyclically monotone ($m\geq 2$) if $\sum_{i=1}^n\iprods{u^i, x^i - x^{i+1}} \geq 0$ for all $(x^i, u^i) \in \gra{T}$ and $x_1 = x_{m+1}$ (see \cite{Bauschke2011}).
We say that $T$ is cyclically monotone if it is $m$-cyclically monotone for every $m \geq 2$.
If $T$ is $m$-cyclically monotone, then it is also $\hat{m}$-cyclically monotone for any $2 \leq \hat{m} \leq m$.
Since a  $2$-cyclically monotone operator $T$ is monotone, any $m$-cyclically monotone operator $T$ is $2$-cyclically monotone, and thus is also monotone.
An $m$-cyclically monotone operator $T$ is called maximally $m$-cyclically monotone if $\gra{T}$ is not properly contained into the graph of any other $m$-cyclically monotone operator. 
As proven in \cite[Theorem 22.18]{Bauschke2011} that $T$ is maximally cyclically monotone iff $T = \partial{f}$, the subdifferential of a proper, closed, and convex function $f$.
However, there exist maximally $m$-cyclically monotone operators (e.g., rotation linear operators) that are not the subdifferential $\partial{f}$ of a proper, closed, and convex function $f$, see, e.g., \cite{Bauschke2011}.
Furthermore, as indicated in  \cite[Example 2.16]{bartz2007fitzpatrick}, there exist maximally $3$-cyclically monotone operators that are not maximal monotone.

%Note that $A_n = \begin{bmatrix}\cos(\theta) & -\sin(\theta) \\ \sin(\theta) & \cos(\theta)\end{bmatrix}$ is $n$-cyclically monotone iff  $\theta \in [0, \frac{\pi}{n}]$.
%For $n=3$ and $\theta = \frac{\pi}{3}$, we have $A_3 = \begin{bmatrix}\frac{1}{2} & -\frac{\sqrt{3}}{2} \\ \frac{\sqrt{3}}{2}& \frac{1}{2}\end{bmatrix}$.
%Then, we have $\iprods{Au, u} = [\frac{1}{2}u_1 - \frac{\sqrt{3}}{2}u_2]u_1 +  [\frac{\sqrt{3}}{2}u_1 + \frac{1}{2}u_2]u_2 = \frac{1}{2}(u_1^2 + u_2^2)$

%%% b. Lipschitz continuity and co-coerciveness.
\vspace{0.75ex}
\noindent\textbf{(c)~Lipschitz continuity and contraction.}
A single-valued or multivalued mapping $T$ is said to be $L$-Lipschitz continuous if $\sup\set{\norms{u - v} : u \in Tx, \ v \in Ty }\leq L\norms{x - y}$ for all $x, y \in\dom{T}$, where $L \geq 0$ is a Lipschitz constant. 
If $T$ is single-valued, then this condition becomes $\norms{Tx - Ty} \leq L\norms{x - y}$ for all $x, y\in\dom{T}$.
If $L = 1$, then we say that $T$ is nonexpansive, while if $L \in [0, 1)$, then we say that $T$ is $L$-contractive, and $L$ is its contraction factor.
If $T$ is $\rho$-co-coercive with $\rho > 0$, then $T$ is also $L$-Lipschitz continuous with the Lipschitz constant $L := \frac{1}{\rho}$.
%We say that $T$ is $\frac{1}{L}$-co-coercive if $\iprods{Tx - Ty, x - y} \geq \frac{1}{L}\norms{Tx - Ty}^2$ for all $x, y\in\dom{T}$.
However, the reverse statement is not true in general.
%If $L = 1$, then we say that $T$ is firmly nonexpansive.
%Note that if $T$ is $\frac{1}{L}$-cocoercive, then it is also monotone and $L$-Lipschitz continuous (by using the Cauchy-Schwarz inequality as $\norms{Tx - Ty}^2 \leq L\iprods{Tx - Ty, x-y} \leq L\norms{Tx - Ty}\norms{x-y}$), but the reverse statement is not true in general.
%If $L < 0$, then we say that $T$ is $\frac{1}{L}$-co-monotone \cite{bauschke2020generalized} (also known as $-\frac{1}{L}$-cohypomonotone).
For a continuously differentiable function $f : \R^p \to \R$, we say that $f$ is $L$-smooth if its gradient $\nabla{f}$ is $L$-Lipschitz continuous on $\dom{f}$.
If $f$ is convex and $L$-smooth, then $\nabla{f}$ is $\frac{1}{L}$-co-coercive and vice versa, see, e.g., \cite{Nesterov2004}.

\vspace{0.75ex}
\noindent\textbf{(d)~Normal cone.}
Given a nonempty, closed, and convex set $\Xc$ in $\R^p$, the normal cone of $\Xc$ is defined as $\Nc_{\Xc}(x) := \sets{w \in \R^p : \iprods{w, x - y} \geq 0, \ \forall y\in\Xc}$ if $x\in\Xc$ and $\Nc_{\Xc}(x) = \emptyset$, otherwise.
The dual cone $\Nc^{*}_{\Xc}(x) := \sets{ w \in \R^p : \iprods{w, u} \geq 0, \ \forall u \in \Nc_{\Xc}(x)}$ of $\Nc_{\Xc}(x)$ at $x$ is $\Tc_{\Xc}(x)$, the tangent cone of $\Xc$ at $x$.
If $f := \delta_{\Xc}$, the indicator of $\Xc$, and $f^{*}$ is its Fenchel conjugate, then $\partial{f} = \Nc_{\Xc}$, and $\partial{f^{*}} = \Tc_{\Xc}$.
%The tangent cone of $\Xc$ is defined as $\Tc_{\Xc}(x) := \sets{d \in\R^p : \iprods{y, d} \leq 0, ~\forall y \in\Nc_{\Xc}(x)} = \mathrm{cl}(\sets{d : x + td \in \Xc, \ t \geq 0})$.
%It is well-known that $\Tc_{\Xc}$ is the dual cone of $\Nc_{\Xc}$, i.e. $\Tc_{\Xc}(x) = (\Nc_{\Xc}(x))^{*}$.

%%%% c. Resolvent and reflection operators.
\vspace{0.75ex}
\noindent\textbf{(e)~Resolvent and proximal operators.}
The operator $J_Tx := \set{y \in \R^p : x \in y + Ty}$ is called the resolvent of $T$, denoted by $J_Tx = (\Id + T)^{-1}x$, where $\Id$ is the identity mapping.
If $T$ is $\rho$-monotone with $\rho > -1$, then evaluating $J_T$ requires solving a strongly monotone inclusion $0 \in y - x + Ty$.
If $T = \partial{f}$, the subdifferential of proper, closed, and convex function $f$, then $J_Tx$ reduces to the proximal operator of $f$, denoted by $\prox_f$, which can be computed as $\prox_f(x) := \mathrm{arg}\min_{y}\sets{f(y) + (1/2)\norms{y-x}^2}$.
In particular, if $T = \Nc_{\Xc}$, the normal cone of a closed and convex set $\Xc$, then $J_T$ is the projection onto $\Xc$, denoted by $\proj_{\Xc}$.
If $T$ is maximally monotone, then $\mathrm{ran}(\Id + T) = \R^p$ (by Minty's theorem) and $T$ is firmly nonexpansive (and thus nonexpansive).

\beforesubsec
\subsection{Best-iterate and last-iterate convergence rates}\label{subsec:types_of_rates}
\aftersubsec
The results presented in this paper are related to two types of  sublinear convergence rates: the best-iterate and the last-iterate convergence rates.
To elaborate on these concepts, we assume that $D$ is a given metric (e.g., $\norms{Fx^k}^2$ or $e(x^k)^2$ defined by \eqref{eq:res_norm} below) defined on an iterate sequence $\sets{x^k}$ generated by the underlying algorithm for solving \eqref{eq:NI} or its special cases.
For any $k \geq 0$ and a given order $\alpha > 0$,  if
\begin{equation*}
\min_{0 \leq l \leq k}D(x^l) \leq \frac{1}{k+1}\sum_{l=0}^k D(x^l) = \BigO{\frac{1}{k^{\alpha}}},
\end{equation*}
then we say that $\set{x^k}$ has a $\BigO{1/k^{\alpha}}$ best-iterate convergence rate.
In this case, we can take $\hat{x}_k :=  x_{k_{\min}}$ with $k_{\min} := \mathrm{arg}\min_{0\leq l \leq k}D(x^l)$ as the ``best'' output of our algorithm. 
If we instead have $D(x^k) =  \BigO{\frac{1}{k^{\alpha}}}$ with $x^k$ being the $k$-th iterate, then we say that $\sets{x^k}$ has a $\BigO{1/k^{\alpha}}$ last-iterate convergence rate.
We emphasize that the convergence on the metric $D$ of $\sets{x^k}$ does not generally imply the convergence of $\sets{x^k}$ itself, especially when characterize the rate of convergence in different metrics.

\beforesubsec
\subsection{Exact solutions and approximate solutions}\label{subsec:exact_and_approx_sols}
\aftersubsec
There are different metrics to characterize exact and approximate solutions of \eqref{eq:NI}.
The most obvious one is the residual norm of $\Phi$, which is defined as
\begin{equation}\label{eq:res_norm}
e(x) := \min_{\xi \in Tx}\norms{Fx + \xi}, \quad x \in \dom{\Phi}.
\end{equation}
Clearly, if $e(x^{\star}) = 0$ for some $x^{\star}\in\dom{\Phi}$, then $x^{\star} \in \zer{\Phi}$, a solution of \eqref{eq:NI}.
If $T = 0$, then $e(x) = \norms{Fx}$.
However, if $e(\hat{x}) \leq \epsilon$ for a given tolerance $\epsilon > 0$, then $\hat{x}$ can be considered as an $\epsilon$-approximate solution of \eqref{eq:NI}.
The algorithms presented in this paper use this metric as one means to characterize approximate solutions.

Other metrics often used for monotone \eqref{eq:VIP}, a special case of \eqref{eq:NI}, are gap functions and restricted gap functions \cite{Facchinei2003,Konnov2001,Nesterov2007a}, which are respectively defined as
\begin{equation}\label{eq:gap_func}
\mcal{G}(x) := \max_{y \in \Xc}\iprods{Fy, y - x} \quad \text{and} \quad \mcal{G}_{\mbb{B}}(x) := \max_{y\in\Xc\cap \mbb{B}}\iprods{Fy, y  - x},
\end{equation}
where $\mbb{B}$ is a given nonempty, closed, and bounded convex set.
Note that $\mcal{G}(x) \geq 0$ for all $x\in\Xc$, and $\mcal{G}(x^{\star}) = 0$ iff $x^{\star}$ is a solution of \eqref{eq:VIP}.
Therefore, to characterize an $\epsilon$-approximate solution $\tilde{x}$ of \eqref{eq:VIP}, we can impose $\mcal{G}(\tilde{x}) \leq \epsilon$.
For the restricted gap function $\mcal{G}_{\mbb{B}}$, if $x^{\star}$ is a solution of \eqref{eq:VIP} and $x^{\star} \in \mbb{B}$, then $\mcal{G}_{\mbb{B}}(x^{\star}) = 0$.
Conversely, if $\mcal{G}_{\mbb{B}}(x^{\star}) = 0$ and $x^{\star} \in \intx{\mbb{B}}$, then $x^{\star}$ is a solution of \eqref{eq:VIP} in $\mbb{B}$ (see \cite[Lemma 1]{Nesterov2007a}).
For \eqref{eq:DVIP}, we can also define similar dual gap functions and restricted dual gap functions \cite{Nesterov2007a}.
Gap functions have been widely used in the literature to characterize approximate solutions generated by many numerical methods for solving \eqref{eq:VIP} or \eqref{eq:DVIP}, see, e.g., \cite{chen2017accelerated,Cong2012,Facchinei2003,Konnov2001,Nemirovskii2004,Nesterov2007a}.

If $J_{\eta T}$ is well-defined and single-valued for some $\eta > 0$, and $F$ is single-valued, then we can use  the following forward-backward splitting residual operator:
\begin{equation}\label{eq:FB_residual}
G_{\eta \Phi }x := \tfrac{1}{\eta}\left(x - J_{\eta T}(x - \eta Fx)\right), 
\end{equation}
to characterize solutions of \eqref{eq:NI},  where $F$ is single-valued and $J_{\eta T}$ is the resolvent of $\eta T$ for any $\eta > 0$.
It is clear that $G_{\eta}x^{\star} = 0$ iff $x^{\star} \in \zer{\Phi}$.
In addition, if $J_{\eta T}$ is firmly nonexpansive, then we also have 
\begin{equation}\label{eq:FBR_bound2}
\norms{G_{\eta\Phi }x} \leq \norms{Fx + \xi}, \quad (x, \xi) \in \gra{T}.
\end{equation}
Hence, for a given tolerance $\epsilon > 0$, if $\norms{G_{\eta\Phi}\tilde{x}} \leq \epsilon$, then we can say that $\tilde{x}$ is an $\epsilon$-approximate solution of \eqref{eq:NI}.
If $T := \Nc_{\Xc}$, i.e., \eqref{eq:NI} reduces to \eqref{eq:VIP}, then, with $\eta = 1$, $G_{\Phi}x$ reduces to the classical natural map $\Pi_{F,\Xc}x = x - \proj_{\Xc}(x - Fx)$ of \eqref{eq:VIP}, and $r_n(x) := \norms{G_{\Phi}x} = \norms{\Pi_{F,\Xc}x}$ is the corresponding natural residual at $x$.
From \eqref{eq:FBR_bound2}, we have $r_n(x) \leq \norms{Fx + \xi}$ for any $\xi \in \Nc_{\Xc}(x)$.

\beforesubsec
\subsection{Gradient/forward-type methods}\label{subsec:GD_FW_methods}
\aftersubsec
Let us briefly recall the gradient/forward scheme for solving \eqref{eq:NE} as follows.
Starting from $x^0 \in\dom{F}$, at each iteration $k \geq 0$, we update
\begin{equation}\label{eq:FW4NE}
x^{k+1} := x^k - \eta Fx^k,
\tag{FW}
\end{equation}
where $\eta > 0$ is a given constant stepsize.
If $F$ is $\rho$-co-coercive and $0 < \eta < \rho$, then $\sets{x^k}$ converges to $x^{\star}\in\zer{F}$ (see, e.g., \cite{Facchinei2003}).
Otherwise, if $F$ is only monotone and $L$-Lipschitz continuous, then there exist examples (e.g., $Fx = [x_2, -x_1]$) showing that \eqref{eq:FW4NE} is divergent for any choice of constant stepsize $\eta$. 

To solve \eqref{eq:NI}, we can instead apply the forward-backward splitting method as follows.
Starting from $x^0 \in\dom{F}$, at each iteration $k \geq 0$, we update
\begin{equation}\label{eq:FBS4NI}
x^{k+1} := J_{\eta T}(x^k - \eta Fx^k),
\tag{FBS}
\end{equation}
where $\eta > 0$ is a given constant stepsize.
Similar to \eqref{eq:FW4NE}, if $F$ is $\rho$-co-coercive and $T$ is maximally monotone, then with $\eta \in (0, \rho)$, $\sets{x^k}$ generated by \eqref{eq:FBS4NI} converges to $x^{\star}\in\zer{\Phi}$.
If $F = \nabla{f}$, the gradient of a convex  and $L$-smooth function $f$, then $F$ is co-coercive.
However, imposing the co-coerciveness for a general mapping $F$ is often restrictive. 
Hence, both \eqref{eq:FW4NE} and \eqref{eq:FBS4NI} are less practical. 

%%%%%%%%%%%%%%%%%%%%%%%%%%%%%%%%%%%%%%%%
%%%% 3. Extragradient-Type Methods For Nonlinear Equations.
%%%%%%%%%%%%%%%%%%%%%%%%%%%%%%%%%%%%%%%%
\beforesec
\section{Extragradient-Type Methods For Nonlinear Equations}\label{sec:EG4NE}
\aftersec
As we have discussed before, the extragradient method was originally proposed by G. M. Korpelevic in 1976 \cite{Korpelevic1976}  and by A. S. Antipin around the same time \cite{antipin1976} to tackle saddle-point problems. 
Since then, this method has been extensively studied in the literature, with numerous variants proposed (see, e.g., \cite{Facchinei2003,he1997class,konnov1997class,Konnov2001,marcotte1991application,solodov1996modified,sun1995new,sun1996class,xiu2001convergence}). 
In recent years, the popularity of this method has increased further due to its effectiveness in solving minimax problems, including those in convex-concave and nonconvex-nonconcave settings, which are common in machine learning and robust optimization. 
In this section, we briefly survey both classical and recent works \cite{diakonikolas2021efficient,golowich2020last,gorbunov2022extragradient,luo2022last} on the extragradient method, as well as some closely related variants with minor modifications. 
We unify the convergence analysis in one single theorem.

%%%% 3.1. The class of extragradient methods for nonlinear equations.
\beforesubsec
\subsection{The class of extragradient methods for nonlinear equations}\label{subsec:EG4NE_scheme}
\aftersubsec
The class of extragradient methods for solving \eqref{eq:NE} we discuss in this section is presented as follows.
Starting from an initial point $x^0 \in\dom{F}$, at each iteration $k \geq 0$, we update 
\begin{equation}\label{eq:EG4NE}
\arraycolsep=0.2em
\left\{\begin{array}{lcl}
y^k &:= & x^k - \frac{\eta}{\beta} u^k, \vspace{1ex}\\
x^{k+1} &:= & x^k - \eta Fy^k,
\end{array}\right.
\tag{EG}
\end{equation}
where $\eta > 0$ is a given constant stepsize, $\beta \in (0, 1]$ is a scaling factor, and $u^k$ has two options as follows.
\begin{itemize}
\itemsep=0.0em
\item\textbf{Option 1.} If we set $u^k := Fx^k$, then we obtain the \mytb{extragradient} scheme \cite{Korpelevic1976} for \eqref{eq:NE}.
\item\textbf{Option 2.} If we choose $u^k := Fy^{k-1}$, the we obtain the \mytb{past-extragradient} method, also called \mytb{Popov's method} \cite{popov1980modification}, to solve \eqref{eq:NE}. 
This scheme is also known as an \mytb{optimistic gradient} method in the literature, see also \cite{daskalakis2018training,mertikopoulos2019optimistic,mokhtari2020convergence}.
\end{itemize}
For \textbf{Option 1} with $u^k := Fx^k$,  if $\beta = 1$, then we obtain exactly the \myti{classical extragradient method} \cite{Korpelevic1976} for solving \eqref{eq:NE}.
If $\beta < 1$, then we recover the \mytb{extragradient-plus (EG$+$)} scheme from \cite{diakonikolas2021efficient} for solving \eqref{eq:NE}.
If we compute $x^k = y^k + \frac{\eta}{\beta} Fx^k$ from the first line of \eqref{eq:EG4NE} and substitute it into the second line of \eqref{eq:EG4NE}, then we get $x^{k+1} = y^k - \eta(Fy^k - \frac{1}{\beta}Fx^k)$.
In this case, we obtain from \eqref{eq:EG4NE} a \mytb{forward-backward-forward splitting} variant of Tseng's method in \cite{tseng2000modified} as follows:
\begin{equation}\label{eq:FBFS4NE}
\arraycolsep=0.2em
\left\{\begin{array}{lcl}
y^k &:= & x^k - \frac{\eta}{\beta} Fx^k, \vspace{1ex}\\
x^{k+1} &:= & y^k - \eta (Fy^k - \frac{1}{\beta}Fx^k).
\end{array}\right.
\tag{FBFS}
\end{equation}
Clearly, if $\beta = 1$, then we recover  exactly \mytb{Tseng's method} for solving \eqref{eq:NE}.

For \textbf{Option 2} with $u^k := Fy^{k-1}$, we can show that it is equivalent to the following variants.
First, we can rewrite \eqref{eq:EG4NE} as
\begin{equation}\label{eq:EG4NE_v1}
\arraycolsep=0.2em
\left\{\begin{array}{lcl}
x^{k+1} & := &  x^k - \eta Fy^k \vspace{1ex}\\
y^{k+1} & := & x^{k+1} - \tfrac{\eta}{\beta}Fy^k.
\end{array}\right.
\tag{PEG}
\end{equation}
This form shows us that \eqref{eq:EG4NE_v1} saves one evaluation $Fx^k$ of $F$ at each iteration compared to \textbf{Option 1}.
If $\beta = 1$, then we obtain exactly the \mytb{Popov's method} in  \cite{popov1980modification}.
If we rotate the second line up and use $\beta = 1$ as $y^k = x^k - \eta Fy^{k-1}$, then we get the \mytb{past-extragradient method}.

Now, under this choice of $u^k$, from the first line of \eqref{eq:EG4NE}, we have $x^k = y^k + \frac{\eta}{\beta}u^k = y^k + \frac{\eta}{\beta}Fy^{k-1}$.
Substituting this expression into the first line of \eqref{eq:EG4NE_v1}, we get $x^{k+1} = y^k - \eta Fy^k + \frac{\eta}{\beta}Fy^{k-1}$.
Substituting this relation into the second line of \eqref{eq:EG4NE_v1}, we can eliminate $x^{k+1}$ to get the following variant:
\begin{equation}\label{eq:FRBS4NE} 
y^{k+1} := y^k - \tfrac{\eta}{\beta}\big( (1 + \beta) Fy^k  - Fy^{k-1}\big).
\tag{FRBS}
\end{equation}
This scheme can be considered as a simplified variant of the \mytb{forward-reflected-backward splitting} scheme in \cite{malitsky2020forward} for solving \eqref{eq:NE} when we set $\beta := 1$ as $y^{k+1} := y^k - \eta ( 2Fy^k  - Fy^{k-1} )$.

Alternatively, from \eqref{eq:EG4NE_v1}, we have $x^{k-1} - x^k =  \eta Fy^{k-1}$ and $\beta (x^k - y^k) = \eta Fy^{k-1}$, leading to $x^{k-1} - x^k = \beta (x^k - y^k)$.
Therefore, we get $y^k = \frac{1}{\beta}( (1+\beta)x^k - x^{k-1})$.
Substituting this expression into the first line of \eqref{eq:EG4NE_v1}, we can show that
\begin{equation}\label{eq:RGD4NE} 
x^{k+1} := x^k - \eta F\big( \tfrac{1}{\beta}( (1+\beta)x^k - x^{k-1}) \big).
\tag{RFB}
\end{equation}
In particular, if $\beta = 1$, then we obtain $x^{k+1} := x^k - \eta F(2x^k - x^{k-1})$, which turns out to be the \mytb{reflected gradient} method in \cite{malitsky2015projected} or the \mytb{reflected forward-backward splitting} scheme in \cite{cevher2021reflected} for solving \eqref{eq:NE}.

Using the relation $x^{k-1} - x^k = \beta (x^k - y^k)$ above, we can compute that $x^{k} = \frac{\beta}{1+\beta}y^k + \frac{1}{1+\beta}x^{k-1} = \frac{(\omega-1)}{\omega}y^k + \frac{1}{\omega}x^{k-1}$, where $\omega := 1 + \beta$.
Combining the two lines of \eqref{eq:EG4NE}, we get $y^{k+1} :=  x^{k+1} - \tfrac{\eta}{\beta}Fy^k = x^k - \frac{\eta(1 + \beta)}{\beta}Fy^k$.
Putting both expressions together, we get
\begin{equation}\label{eq:GR4NE} 
\arraycolsep=0.2em
\left\{\begin{array}{lcl}
x^k  & := & \tfrac{(\omega-1)}{\omega}y^k + \tfrac{1}{\omega}x^{k-1}, \vspace{1ex}\\
y^{k+1}  & := &  x^{k} - \tfrac{\eta(1+\beta)}{\beta}Fy^k.
\end{array}\right.
\tag{GR}
\end{equation}
This method is a simplified variant of the \mytb{golden-ratio} method in \cite{malitsky2019golden} for solving \eqref{eq:NE}.
Overall, the template \eqref{eq:EG4NE} covers a class of EG algorithms with many common instances as discussed. 

%%%% 3.2. Convergence analysis
\beforesubsec
\subsection{Convergence analysis}\label{subsec:EG4NE_convergence}
\aftersubsec
The results presented in this section were obtained in \cite{luo2022last}, but here we provide a different proof and unify several methods in one.
To analyze the convergence of \eqref{eq:EG4NE}, we first prove the following lemmas.

%%% Lemma 3.1.
\begin{lemma}\label{le:EG4NE_key_estimate1}
If $\set{(x^k, y^k)}$ is generated by \eqref{eq:EG4NE}, then for any $\gamma > 0$ and any $\hat{x}\in\dom{F}$, we have
\begin{equation}\label{eq:EG4NE_key_est1}
\arraycolsep=0.2em
\begin{array}{lcl}
\norms{x^{k+1} - \hat{x}}^2 & \leq & \norms{x^k - \hat{x}}^2 - \beta\norms{y^k - x^k}^2 +  \tfrac{\eta^2}{\gamma}\norms{Fy^k - u^k}^2 -  2\eta\iprods{Fy^k, y^k - \hat{x}} \vspace{1ex}\\
&& - {~} (\beta - \gamma)\norms{x^{k+1} - y^k}^2 - (1 - \beta)\norms{x^{k+1} - x^k}^2.
\end{array}
\end{equation}
\end{lemma}

%%% Proof of Lemma 3.1.
\begin{proof}
First, for any $\hat{x}\in\dom{F}$, using $x^{k+1} - x^k = -\eta Fy^k$ from the second line of \eqref{eq:EG4NE}, we have
\begin{equation*} 
\arraycolsep=0.2em
\begin{array}{lcl}
\norms{x^{k+1} - \hat{x}}^2 &= & \norms{x^k - \hat{x}}^2 + 2\iprods{x^{k+1} - x^k, x^{k+1} - \hat{x}} - \norms{x^{k+1} - x^k}^2 \vspace{1ex}\\
&= & \norms{x^k - \hat{x}}^2 - 2\eta\iprods{Fy^k, x^{k+1} - \hat{x}} -  \norms{x^{k+1} - x^k}^2.
\end{array}
\end{equation*}
Next, using  $\eta u^k = \beta(x^k - y^k)$ from the first line of \eqref{eq:EG4NE}, 
the Cauchy-Schwarz inequality, the identity $2\iprods{x^{k+1} - y^k, x^k - y^k} = \norms{x^k - y^k}^2 + \norms{x^{k+1} - y^k}^2 - \norms{x^{k+1} - x^k}^2$, and an elementary inequality $2wz \leq \gamma w^2 + \frac{z^2}{\gamma}$  for any $\gamma > 0$ and $w, z \geq 0$, we can  derive that
\begin{equation*} 
\arraycolsep=0.2em
\begin{array}{lcl}
2\eta\iprods{Fy^k, x^{k+1} - \hat{x}} & = & 2\eta\iprods{Fy^k, y^k - \hat{x}} + 2\eta\iprods{Fy^k - u^k, x^{k+1} - y^k} + 2\eta\iprods{u^k, x^{k+1} - y^k} \vspace{1ex}\\
&\geq & 2\eta\iprods{Fy^k, y^k - \hat{x}} - 2\eta\norms{Fy^k - u^k}\norms{x^{k+1} - y^k} +  2\beta \iprods{x^{k+1} - y^k, x^k - y^k} \vspace{1ex}\\
&\geq & 2\eta\iprods{Fy^k, y^k - \hat{x}} - \frac{\eta^2}{\gamma}\norms{Fy^k - u^k}^2 - \gamma \norms{x^{k+1} - y^k}^2 \vspace{1ex}\\
&& + {~} \beta\big[ \norms{x^k - y^k}^2 + \norms{x^{k+1} - y^k}^2 - \norms{x^{k+1} - x^k}^2\big] \vspace{1ex}\\
&= & 2\eta\iprods{Fy^k, y^k - \hat{x}} + \beta\norms{y^k - x^k}^2  - \frac{\eta^2}{\gamma}\norms{Fy^k - u^k}^2 \vspace{1ex}\\
&& + {~} (\beta - \gamma)\norms{x^{k+1} - y^k}^2 - \beta\norms{x^{k+1} - x^k}^2.
\end{array}
\end{equation*}
Finally, combining the last two expressions, we obtain \eqref{eq:EG4NE_key_est1}.
\end{proof}
%%% End of proof.

%%% Lemma 3.2.
\begin{lemma}\label{le:EG4NE_monotonicity}
Let $F$ be $\rho$-co-hypomonotone, i.e. there exists $\rho \geq 0$ such that $\iprods{Fx - Fy, x-y} \geq -\rho\norms{Fx - Fy}^2$ for all $x, y \in\dom{F}$ and $L$-Lipschitz continuous.
Let $\set{(x^k, y^k)}$ be generated by \eqref{eq:EG4NE}.
Then,  for any $c > 0$ and $\omega > 0$, we have
\begin{equation}\label{eq:EG4NE_monotonicity}
\arraycolsep=0.2em
\begin{array}{lcl}
\norms{Fx^{k+1}}^2 & \leq & \norms{Fx^k}^2 -  \frac{[c\eta - 2(1+c)\rho]}{c\eta} \norms{Fy^k - Fx^k}^2 + \frac{\left[\eta\omega + 2(1+c)\rho \right] L^2\eta}{\beta^2}\norms{\beta Fy^k - u^k}^2 \vspace{1ex}\\
&& - {~} (\omega - 1)\norms{Fx^{k+1} - Fy^k}^2.
\end{array}
\end{equation}
\end{lemma}

%%% Proof of Lemma 3.2.
\begin{proof}
Since  $F$ is $\rho$-cohypomonotone, we have $\iprods{Fx^{k+1} - Fx^k, x^{k+1} - x^k} + \rho\norms{Fx^{k+1} - Fx^k}^2 \geq 0$.
Substituting $x^{k+1} - x^k = -\eta Fy^k$ from the second line of \eqref{eq:EG4NE} into this inequality, we can show that
\begin{equation*} 
\arraycolsep=0.2em
\begin{array}{lcl}
0 & \leq & 2\iprods{Fx^k, Fy^k} - 2\iprods{Fx^{k+1}, Fy^k} + \frac{2\rho}{\eta}\norms{Fx^{k+1} - Fx^k}^2 \vspace{1ex}\\
&\leq & \norms{Fx^k}^2 - \norms{Fy^k - Fx^k}^2 - \norms{Fx^{k+1}}^2  + \norms{Fx^{k+1} - Fy^k}^2 + \frac{2\rho}{\eta}\norms{Fx^{k+1} - Fx^k}^2.
\end{array}
\end{equation*}
Now, by utilizing Young's inequality, the $L$-Lipschitz continuity of $F$, and $x^{k+1} - y^k = -\eta (Fy^k - \frac{1}{\beta}u^k)$ from \eqref{eq:EG4NE}, for any $c > 0$ and $\omega \geq 1$, the last estimate leads to
\begin{equation*} 
\arraycolsep=0.2em
\begin{array}{lcl}
\norms{Fx^{k+1}}^2 &\leq & \norms{Fx^k}^2 - \norms{Fy^k - Fx^k}^2 + \omega \norms{Fx^{k+1} - Fy^k}^2  + \frac{2\rho}{\eta}\norms{Fx^{k+1} - Fx^k}^2 - (\omega - 1)\norms{Fx^{k+1} - Fy^k}^2 \vspace{1ex}\\
&\leq & \norms{Fx^k}^2 -  \frac{c\eta - 2(1+c)\rho}{c\eta} \norms{Fy^k - Fx^k}^2 +  \frac{[\eta\omega  + 2(1+c)\rho] L^2}{\eta} \norms{x^{k+1} - y^k}^2 - (\omega - 1)\norms{Fx^{k+1} - Fy^k}^2 \vspace{1ex}\\
&\leq & \norms{Fx^k}^2 -  \frac{c\eta - 2(1+c)\rho}{c\eta} \norms{Fy^k - Fx^k}^2 +  \left[\eta\omega + 2(1+c)\rho \right] L^2\eta \norms{Fy^k - \frac{1}{\beta}u^k}^2 \vspace{1ex}\\
&& - {~}  (\omega - 1)\norms{Fx^{k+1} - Fy^k}^2 ,
\end{array}
\end{equation*}
which  exactly proves \eqref{eq:EG4NE_monotonicity}.
\end{proof}
%%% End of proof.

Now, we are ready to establish both the best-iterate and the last-iterate convergence rates of  \eqref{eq:EG4NE}. 

%%% Theorem 3.1.
\begin{theorem}\label{th:EG4NE_convergence}
Suppose that $F$ in \eqref{eq:NE} is $L$-Lipschitz continuous and $\zer{F}\neq\emptyset$.
Let $\set{(x^k, y^k)}$ be generated by \eqref{eq:EG4NE} for solving \eqref{eq:NE}.
Then, we have the following statements.
\begin{itemize}
\item[$\mathrm{(a)}$] $($\mytb{Extragradient method}$)$ 
Let us choose $u^k := Fx^k$ and assume that there exists $\rho \geq 0$ such that $\iprods{Fx, x - x^{\star}} \geq -\rho\norms{Fx}^2$ for all $x\in\dom{F}$ and a given $x^{\star}\in\zer{F}$ $($this condition holds if, in particular, $F$ is $\rho$-co-hypomonotone on $\dom{F}$$)$.
Then, if $L\rho \leq \frac{3\sqrt{2} -2}{12} \approx 0.1869$, $\beta \in (0, 1]$, and $\eta$ is chosen such that 
\begin{equation}\label{eq:EG4NE_stepsize1}
\arraycolsep=0.2em
\begin{array}{l}
0 \leq \frac{\beta [ 1 - \sqrt{1 - 24 L\rho(3L\rho + 1)} ]}{2L(3L\rho + 1)} < \eta < \frac{\beta [ 1 + \sqrt{1 - 24 L\rho(3L\rho + 1)} ]}{2L(3L\rho + 1)} \leq \frac{\beta}{L}, 
\end{array}
\end{equation}
then we have
\begin{equation}\label{eq:EG4NE_convergence_est1a}
\min_{0\leq l \leq k}\norms{Fx^l}^2 \leq \frac{1}{k+1}\sum_{l=0}^k\norms{Fx^l}^2 \leq \frac{C_{\rho} \norms{x^0 - x^{\star}}^2}{k+1}, 
\end{equation} 
where $C_{\rho} :=   \frac{\beta^2}{\eta[ \eta\beta - 6\beta^2\rho - (3L\rho + 1)L\eta^2] } > 0$.
Consequently, $\sets{\norms{x^k - x^{\star}}}$ is nonincreasing and $\lim_{k\to\infty}\norms{x^k - y^k} = \lim_{k\to\infty}\norms{Fx^k} = \lim_{k\to\infty}\norms{Fy^k} =  0$.
Moreover, we have $\min_{0\leq l \leq k}\norms{Fx^l} = \BigOs{1/\sqrt{k}}$ showing the $\BigOs{1/\sqrt{k}}$ best-iterate convergence rate of $\sets{x^k}$.

In particular, if $\beta := 1$ and $F$ is $\rho$-co-hypomonotone on $\dom{F}$ such that $L\rho \leq \frac{3\sqrt{2} -2}{12}$, then
\begin{equation}\label{eq:EG4NE_convergence_est1b}
\norms{Fx^{k+1}}^2 \leq  \norms{Fx^k}^2 - \psi \cdot \norms{Fy^k - Fx^k}^2 \quad\text{and}\quad \norms{Fx^k} \leq \frac{\sqrt{C_{\rho}} \norms{x^0 - x^{\star}} }{\sqrt{k+1}},
\end{equation}
where $\psi := 1 - \frac{4\rho}{\eta} - L^2\eta(\eta + 4\rho) > 0$.
Hence, we have $\norms{Fx^k} = \BigOs{1/\sqrt{k}}$ on the last-iterate $x^k$.

\item[$\mathrm{(b)}$] $($\mytb{Past-extragradient method}$)$ 
Let us choose $u^k := Fy^{k-1}$ and $y^{-1} := x^0$ and assume that there exists $\rho \geq 0$ such that $\iprods{Fx, x - x^{\star}} \geq -\rho\norms{Fx}^2$ for all $x\in\dom{F}$ and a given $x^{\star}\in\zer{F}$ $($in particular, if $F$ is $\rho$-co-hypomonotone on $\dom{F}$$)$.
Then, for fixed $\beta \in (0, 1]$, if $L\rho \leq  \frac{\beta^2}{12}$ and $\eta$ is chosen such that 
\begin{equation}\label{eq:EG4NE_stepsize2}
\arraycolsep=0.2em
\begin{array}{l}
0 \leq \frac{\beta - \sqrt{\beta^2 - 12L\rho}}{6L} < \eta < \frac{\beta + \sqrt{\beta^2 - 12L\rho}}{6L} \leq \frac{\beta}{3L},
\end{array}
\end{equation}
then we have
\begin{equation}\label{eq:EG4NE_convergence_est2a}
\min_{0\leq l\leq k}  \norms{Fx^l}^2  \leq \frac{1}{k+1}\sum_{l=0}^k \big[ \norms{Fx^l}^2 + \kappa\norms{x^l - y^{l-1}}^2\big] \leq \frac{\hat{C}_{\rho}\norms{x^0 - x^{\star}}^2}{k+1},
\end{equation} 
where $\kappa := \frac{(2L^2\eta^2 + 3)L}{2(\beta\eta - 3L\eta^2 - 4\rho)} > 0$ and $\hat{C}_{\rho} :=   \frac{2L^2\eta^2 + 3}{(\beta\eta - 3L\eta^2 - 4\rho)\eta} > 0$.
Consequently, we also have $\lim_{k\to\infty}\norms{x^k - y^k} = \lim_{k\to\infty}\norms{x^{k+1} - y^k} = \lim_{k\to\infty}\norms{Fx^k} = \lim_{k\to\infty}\norms{Fy^k} =  0$, and $\min_{0\leq l \leq k}\norms{Fx^l} = \BigOs{1/\sqrt{k}}$ showing the $\BigOs{1/\sqrt{k}}$ best-iterate convergence rate of $\sets{x^k}$.

In particular, if $\beta := 1$ and $F$ is $\rho$-co-hypomonotone on $\dom{F}$ such that $12L\rho \leq 1$, then
\begin{equation}\label{eq:EG4NE_convergence_est2b}
\norms{Fx^{k+1}}^2 + \hat{\kappa}\norms{Fx^{k+1} - Fy^k}^2 \leq  \norms{Fx^k}^2 + \hat{\kappa}\norms{Fx^k - Fy^{k-1}}^2 \quad\text{and}\quad \norms{Fx^k} \leq \frac{\sqrt{M_{\rho}} \norms{x^0 - x^{\star}} }{\sqrt{k+1}},
\end{equation}
where $\hat{\kappa} := \frac{2(\eta + 4\rho)L^2\eta}{1 - 2L^2\eta^2}$ and $M_{\rho} :=  \hat{C}_{\rho} \cdot \max\set{\frac{L^2\hat{\kappa}}{\kappa}, 1}$.
Moreover, we have the last-iterate convergence rate as $\norms{Fx^k} = \BigOs{1/\sqrt{k}}$.
\end{itemize}
\end{theorem}
%%% End of Theorem 3.1

Before proving Theorem~\ref{th:EG4NE_convergence}, we give the following remarks.
%%% Remark 1.
\begin{remark}\label{re:EG4NE_rm1}
If $\rho = 0$ in Theorem~\ref{th:EG4NE_convergence}, i.e. $F$ is star-monotone (and in particular, monotone), then our condition on the stepsize  $\eta$ reduces to $0 < \eta < \frac{\beta}{L}$ for the extragradient method and $0 < \eta < \frac{\beta}{3L}$ for the past-extragradient method. 
These choices are standard and often seen in both methods.
However, we have not yet optimized the choice of these parameters for \eqref{eq:EG4NE} stated in Theorem~\ref{th:EG4NE_convergence}.
Here, the stepsize $\eta$ can be carefully chosen so that we can possibly enlarge the range of $L\rho$ (see  \cite{gorbunov2022convergence} as an example).
\end{remark}

%%%% Remark 3.2.
\begin{remark}\label{re:EG4NE_comparison}
The results in Theorem~\ref{th:EG4NE_convergence} were proven in \cite{luo2022last}, and then were revised in \cite{gorbunov2022convergence}.
The last-iterate convergence rates were proven in previous works such as \cite{golowich2020last} for the monotone case but with an additional assumption.
Note that the best-iterate rates for the monotone or the star-monotone case are classical, which can be found, e.g., in \cite{Facchinei2003,Korpelevic1976}.
The last-iterate convergence for the monotone case can be found in recent works such as \cite{golowich2020last,gorbunov2022extragradient}.
The best-iterate rates for the co-hypomonotone or the star-co-hypomonotone case can be found in \cite{diakonikolas2021efficient}, while the last-iterate convergence rates were recently proven in \cite{luo2022last}.
Nevertheless, in this survey, we provide a unified analysis for all of these variants of EG, which covers  both the monotone and co-hypomonotone cases altogether.
Our analysis is also different from \cite{luo2022last}.
\end{remark}

%%%% Remark 3.3.
\begin{remark}\label{re:EG4NE_star_strong_mono}
One can easily modify the proof of Theorem~\ref{th:EG4NE_convergence} to handle the star-strongly-monotone case of $F$.
Indeed, if $F$ is $\mu$-star-strongly monotone, then we have $\mu\norms{x^k - x^{\star}}^2 \leq \iprods{Fx^k, x^k - x^{\star}} \leq \norms{Fx^k}\norms{x^k - x^{\star}}$.
This leads to $\norms{Fx^k}^2 \geq \mu^2\norms{x^k - x^{\star}}^2$.
Using this inequality, $\hat{x} := x^{\star}$, and $u^k := Fx^k$ into \eqref{eq:EG4NE_key_est1}, we obtain $\norms{x^{k+1} - x^{\star}}^2 \leq (1 - \eta^2(1 - L^2\eta^2)\mu^2 )\norms{x^k - x^{\star}}^2$.
Clearly, if we choose $\eta \in \big(0, \frac{1}{L}\big)$, then $\varphi := 1 - \eta^2(1 - L^2\eta^2)\mu^2 \in (0, 1)$, and we obtain a linear convergence rate of $\sets{\norms{x^k - x^{\star}}^2}$ with a contraction factor $\varphi$.
Note that since $\mu \leq L$, we have $\mu^4 - 4L^2\mu^2 < 0$. Hence, $L^2\mu^2\eta^4 - \mu^2\eta^2 + 1 \geq 0$ always holds to guarantee that $\varphi \in (0, 1)$. 
Another proof for the monotone case can be found, e.g., in \cite{Facchinei2003}.
\end{remark}

%%% Proof of Theorem 3.1.
\begin{proof}[\mytb{Proof of Theorem~\ref{th:EG4NE_convergence}}]
(a)~\mytb{Extragradient method.} Since $u^k := Fx^k$, by the $L$-Lipschitz continuity of $F$, we have $\norms{Fy^k - u^k} = \norms{Fy^k - Fx^k} \leq L\norms{x^k - y^k}$.
Using this inequality, $\hat{x} := x^{\star} \in \zer{F}$, and $\iprods{Fy^k, y^k - x^{\star}} \geq -\rho\norms{Fy^k}^2$ into \eqref{eq:EG4NE_key_est1}, we have
\begin{equation}\label{eq:EG4NE_key_est1_for_EG}
\arraycolsep=0.2em
\begin{array}{lcl}
\norms{x^{k+1} - x^{\star} }^2 & \leq & \norms{x^k - x^{\star} }^2 - \left(\beta -  \tfrac{L^2\eta^2}{\gamma}\right) \norms{y^k - x^k}^2  +  2\eta\rho\norms{Fy^k}^2 \vspace{1ex}\\
&& - {~} (\beta - \gamma)\norms{x^{k+1} - y^k}^2 - (1 - \beta)\norms{x^{k+1} - x^k}^2.
\end{array}
\end{equation}
Now, using the first line of \eqref{eq:EG4NE} with $u^k := Fx^k$ as $\beta(x^k - y^k) = \eta Fx^k$, the $L$-Lipschitz continuity of $F$, and Young's inequality,  we have 
\begin{equation}\label{eq:EG4NE_key_est1_for_EG2a} 
\arraycolsep=0.2em
\begin{array}{lcl}
\norms{Fy^k}^2 &\leq & \frac{3}{2} \norms{Fy^k - Fx^k}^2 +  3\norms{Fx^k}^2 \leq \frac{3(2 \beta^2 + L^2\eta^2)}{2\eta^2}\norms{x^k - y^k}^2.
\end{array}
\end{equation}
Substituting this inequality, $\gamma := L\eta$, and $\beta^2\norms{y^k - x^k}^2 = \eta^2\norms{Fx^k}^2$ into \eqref{eq:EG4NE_key_est1_for_EG}, we obtain
\begin{equation}\label{eq:EG4NE_key_est1_for_EG2}
\arraycolsep=0.2em
\begin{array}{lcl}
\norms{x^{k+1} - x^{\star} }^2 & \leq & \norms{x^k - x^{\star} }^2 - (\beta - L\eta)\norms{x^{k+1} - y^k}^2  - (1 - \beta)\norms{x^{k+1} - x^k}^2 \vspace{1ex}\\
&& - {~} \frac{\eta[ \eta\beta - 6\beta^2\rho - (3L\rho + 1)L\eta^2] }{\beta^2}  \norms{Fx^k}^2. 
\end{array}
\end{equation}
Let us choose $\eta > 0$ such that $\eta\beta - 6\beta^2\rho - (3L\rho + 1)L\eta^2 > 0$, which holds if $\eta$ satisfies
\begin{equation*} 
\arraycolsep=0.2em
\begin{array}{lcl}
0 \leq \frac{\beta \big[ 1 - \sqrt{1 - 24 L\rho(3L\rho + 1)} \big]}{2L(3L\rho + 1)} < \eta < \frac{\beta \big[ 1 + \sqrt{1 - 24 L\rho(3L\rho + 1)} \big]}{2L(3L\rho + 1)} \leq \frac{\beta}{L},
\end{array}
\end{equation*}
provided that $L\rho \leq \frac{3\sqrt{2} -2}{12} \approx 0.1869$.
This condition is exactly \eqref{eq:EG4NE_stepsize1}.
In this case, we also have $\beta - L\eta \geq 0$, and \eqref{eq:EG4NE_key_est1_for_EG2} implies \eqref{eq:EG4NE_convergence_est1a}.
The next statement is a consequence of \eqref{eq:EG4NE_key_est1_for_EG2} combining with \eqref{eq:EG4NE_key_est1_for_EG2a} using standard arguments.

Next, if $F$ is $\rho$-co-hypomonotone, then with $\beta := 1$, using $\norms{\beta Fy^k - u^k} = \norms{Fy^k - Fx^k}$ and $c = \omega = 1$ into  \eqref{eq:EG4NE_monotonicity}, we obtain
\begin{equation*} 
\arraycolsep=0.2em
\begin{array}{lcl}
\norms{Fx^{k+1}}^2 \leq \norms{Fx^k}^2 - \frac{\left[\eta - 4\rho - L^2\eta^2(\eta + 4\rho)\right]}{\eta}\norms{Fy^k - Fx^k}^2.
\end{array}
\end{equation*}
However, by the choice of $\eta$ as in \eqref{eq:EG4NE_stepsize1}, one has $\eta - 6\rho - (3L\rho + 1)L\eta^2 > 0$.
It is obvious to check that $\psi  := \eta - 4\rho - L^2\eta^2(\eta + 4\rho) \geq \eta - 6\rho - (3L\rho + 1)L\eta^2 > 0$. 
This condition leads to the first part of \eqref{eq:EG4NE_convergence_est1b}.
The second part of \eqref{eq:EG4NE_convergence_est1b} is a consequence of the first part and \eqref{eq:EG4NE_convergence_est1a}.

(b)~\mytb{Past-extragradient method.}
If we choose $u^k := Fy^{k-1}$, then by Young's inequality and the $L$-Lipschitz continuity of $F$, we have
\begin{equation}\label{eq:EG4NE_proof_p2_1} 
\arraycolsep=0.2em
\begin{array}{lcl}
\norms{Fy^k - u^k}^2 & =  & \norms{Fy^k - Fy^{k-1}}^2 \leq 3\norms{Fy^k - Fx^k}^2 + \frac{3}{2}\norms{Fx^k - Fy^{k-1}}^2 \vspace{1ex}\\
& \leq & 3L^2\norms{x^k - y^k}^2 + \frac{3L^2}{2}\norms{x^k - y^{k-1}}^2.
\end{array}
\end{equation}
Moreover, since $\iprods{Fx, x - x^{\star}} \geq -\rho\norms{Fx}^2$ for all $x\in\dom{F}$, using this condition, $\eta Fy^k = x^{k+1} - x^k$ from \eqref{eq:EG4NE}, and Young's inequality, we can lower bound that
\begin{equation}\label{eq:EG4NE_proof_p2_2}  
\arraycolsep=0.2em
\begin{array}{lcl}
2\eta\iprods{Fy^k, y^k - x^{\star}} &\geq & -2\rho\eta\norms{Fy^k}^2 = -\frac{2\rho}{\eta}\norms{x^{k+1} - x^k}^2 \geq - \frac{4\rho}{\eta}\big[ \norms{x^{k+1} -y^k}^2 + \norms{y^k - x^k}^2 \big].
\end{array}
\end{equation}
Substituting \eqref{eq:EG4NE_proof_p2_1}, \eqref{eq:EG4NE_proof_p2_2}, and $\hat{x} := x^{\star} \in \zer{F}$ into \eqref{eq:EG4NE_key_est1}, we obtain
\begin{equation}\label{eq:EG4NE_proof_p2_3}   
\arraycolsep=0.2em
\begin{array}{lcl}
\norms{x^{k+1} - x^{\star}}^2 & + & \left(\frac{3L^2\eta^2}{\gamma} - \gamma \right)\norms{x^{k+1} - y^k}^2  \leq  \norms{x^k - x^{\star}}^2 + \left( \tfrac{3L^2\eta^2}{\gamma} - \gamma\right)\norms{x^k - y^{k-1}}^2 \vspace{1ex}\\
&& - {~} \left(\beta - \frac{3L^2\eta^2}{\gamma} - \frac{4\rho}{\eta}  \right)\left[ \norms{x^{k+1} - y^k}^2  +  \norms{y^k - x^k}^2 \right] -   \left( \tfrac{3L^2\eta^2}{2\gamma} - \gamma\right)\norms{x^k - y^{k-1}}^2 \vspace{1ex}\\
&& - {~}   (1 - \beta)\norms{x^{k+1} - x^k}^2.
\end{array}
\end{equation}
Next, by the second line of \eqref{eq:EG4NE} and Young's inequality, we can easily show that $\norms{Fy^k}^2 = \frac{1}{\eta^2}\norms{x^{k+1} - x^k}^2 \leq \frac{3}{2\eta^2}\norms{x^k - y^k}^2 + \frac{3}{\eta^2}\norms{x^{k+1} - y^k}^2$.
Alternatively,  by Young's inequality and the $L$-Lipschitz continuity of $F$, we also have $\norms{Fx^k}^2 \leq \left(L^2 + \frac{3}{2\eta^2}\right)\norms{x^k - y^k}^2 + \left(1 + \frac{2L^2\eta^2}{3}\right)\norms{Fy^k}^2$.
Combining both inequalities, we get 
\begin{equation*}
\arraycolsep=0.2em
\begin{array}{lcl}
\norms{Fx^k}^2 \leq   \frac{2L^2\eta^2 + 3}{\eta^2}  \left[ \norms{x^k - y^k}^2  + \norms{x^{k+1} - y^k}^2\right].
\end{array}
\end{equation*}
Now, let us choose $\gamma := L\eta$ and using the last inequality into \eqref{eq:EG4NE_proof_p2_3}, we can show that
\begin{equation}\label{eq:EG4NE_proof_p2_4}
\arraycolsep=0.2em
\begin{array}{lcl}
\norms{x^{k+1} - x^{\star}}^2  +  2L\eta\norms{x^{k+1} - y^k}^2 & \leq &  \norms{x^k - x^{\star}}^2 + 2L\eta\norms{x^k - y^{k-1}}^2 - (1 - \beta)\norms{x^{k+1} - x^k}^2 \vspace{1ex}\\
&& - {~} \frac{\eta(\beta\eta - 3L\eta^2 - 4\rho)}{2L^2\eta^2 + 3} \norms{Fx^k}^2  -    \tfrac{L\eta}{2}\norms{x^k - y^{k-1}}^2.
\end{array}
\end{equation}
If we choose $\eta > 0$ such that $\beta\eta - 3L\eta^2 - 4\rho > 0$, then \eqref{eq:EG4NE_proof_p2_4} implies \eqref{eq:EG4NE_convergence_est2a}.
The next statement of Part (b) follows from \eqref{eq:EG4NE_convergence_est2a} and the fact that $\norms{Fy^{k-1}}^2 \leq \frac{L^2+\kappa}{\kappa}\big[ \norms{Fx^k}^2 + \kappa \norms{x^k - y^{k-1}}^2 \big]$.
Note that, the condition  $\beta\eta - 3L\eta^2 - 4\rho > 0$ holds if  $0 \leq \frac{\beta - \sqrt{\beta^2 - 12L\rho}}{6L} < \eta < \frac{\beta + \sqrt{\beta^2 - 12L\rho}}{6L} \leq \frac{\beta}{3L}$, provided that $12L\rho \leq \beta^2$ as stated in Theorem~\ref{th:EG4NE_convergence}.

To prove the last-iterate rate, let us choose $\omega := \frac{1 + 8L^2\rho\eta}{1 - 2L^2\eta^2} > 1$, provided that  $\sqrt{2}L\eta <  1$.
Moreover, for $\beta = 1$ and $u^k := Fy^{k-1}$,  we have $\norms{\beta Fy^k - u^k}^2 = \norms{Fy^k - Fy^{k-1}}^2 \leq 2\norms{Fy^k - Fx^k}^2 + 2\norms{Fx^k - Fy^{k-1}}^2$.
Substituting this inequality, $c := 1$, and $\omega := \frac{1 + 8L^2\rho\eta}{1 - 2L^2\eta^2}$  into \eqref{eq:EG4NE_monotonicity}, we obtain 
\begin{equation}\label{eq:EG4NE_monotonicity_proof2}
\arraycolsep=0.2em
\begin{array}{lcl}
\norms{Fx^{k+1}}^2 + (\omega - 1)\norms{Fx^{k+1} - Fy^k}^2 & \leq & \norms{Fx^k}^2 + (\omega - 1) \norms{Fx^k - Fy^{k-1}}^2 \vspace{1ex}\\
&& - {~}  \left[ 1 - \frac{4\rho}{\eta} -  \frac{2L^2\eta(\eta + 4\rho)}{1 - 2L^2\eta^2}\right] \norms{Fy^k - Fx^k}^2. 
\end{array}
\end{equation}
It is obvious to show that the conditions $1 - \frac{4\rho}{\eta} - 3L\eta > 0$ and $\sqrt{2}L\eta < 1$ guarantee that $1 - \frac{4\rho}{\eta} -  \frac{2L^2\eta(\eta + 4\rho)}{1 - 2L^2\eta^2} \geq 0$.
Hence, if we define $\hat{\kappa} := \omega - 1 = \frac{2(\eta + 4\rho)L^2\eta}{1 - 2L^2\eta^2}$, then \eqref{eq:EG4NE_monotonicity_proof2} reduces to 
\begin{equation}\label{eq:EG4NE_monotonicity_proof3} 
\arraycolsep=0.2em
\begin{array}{lcl}
\norms{Fx^{k+1}}^2 +  \hat{\kappa} \norms{Fx^{k+1} - Fy^k}^2 & \leq & \norms{Fx^k}^2 + \hat{\kappa}\norms{Fx^k - Fy^{k-1}}^2.
\end{array}
\end{equation}
For $C_0 := \max\set{\frac{L^2\hat{\kappa}}{\kappa}, 1}$, we have $\norms{Fx^k}^2 + \hat{\kappa}\norms{Fx^k - Fy^{k-1}}^2 \leq  C_0\left[ \norms{Fx^k}^2 +  \frac{\kappa}{L^2}\norms{Fx^k - Fy^{k-1}}^2 \right] \leq C_0 \left[ \norms{Fx^k}^2 + \kappa\norms{x^k - y^{k-1}}^2 \right]$.
Combining this inequality and \eqref{eq:EG4NE_convergence_est2a}, we get
\begin{equation*} 
\frac{1}{k+1}\sum_{l=0}^k\big[ \norms{Fx^l}^2 + \hat{\kappa}\norms{Fx^l - Fy^{l-1}}^2\big]  \leq \frac{C_0}{k+1}\sum_{l=0}^k\big[ \norms{Fx^l}^2 + \kappa\norms{x^l - y^{l-1}}^2\big] \leq \frac{C_0\hat{C}_{\rho}\norms{x^0 - x^{\star}}^2}{k+1},
\end{equation*} 
Using \eqref{eq:EG4NE_monotonicity_proof3} into the last  bound, we obtain \eqref{eq:EG4NE_convergence_est2b}.
Note that the condition $3L\eta < 1$ guarantees that $\sqrt{2}L\eta < 1$.
The remaining statement of (b) in Theorem~\ref{th:EG4NE_convergence} is a direct consequence of \eqref{eq:EG4NE_convergence_est2b} and \eqref{eq:EG4NE_proof_p2_2}.
\end{proof}
%%% End of proof.

%%%%%%%%%%%%%%%%%%%%%%%%%%%%%%%%%%%%%%%%%%
%%%%% 4. Extragradient-Type Methods for Monotone Inclusions
%%%%%%%%%%%%%%%%%%%%%%%%%%%%%%%%%%%%%%%%%%
\beforesec
\section{Extragradient-Type Methods for Monotone Inclusions}\label{sec:EG4NI}
\aftersec
In this section, we go beyond \eqref{eq:NE} to survey recent results on both best-iterate and last-iterate convergence rates of the EG method and its variants for solving \eqref{eq:NI}.
Again, we provide a unified analysis that covers a wide class of EG variants of the monotone instances of \eqref{eq:NI} as can be seen below.

%%%% 4.1. The class of extragradient methods.
\beforesubsec
\subsection{The class of extragradient methods}\label{subsect:EG4NI}
\aftersubsec
The  class of EG methods for solving \eqref{eq:NI} we consider in this section can be described as follows.
Starting from an initial point $x^0 \in \dom{\Phi}$, at each iteration $k\geq 0$, we update
\begin{equation}\label{eq:EG4NI}
\arraycolsep=0.2em
\left\{\begin{array}{lcl}
y^k &:= & J_{\frac{\eta}{\beta}T}(x^k - \frac{\eta}{\beta} u^k), \vspace{1ex}\\
x^{k+1} &:= & J_{\eta T}(x^k - \eta Fy^k),
\end{array}\right.
\tag{EG2}
\end{equation}
where $J_{\eta T}$ is the resolvent of $\eta T$, $\eta > 0$ is a given stepsize, and $\beta > 0$ is a scaling factor.
Here, we consider two different choices of $u^k$ as follows:
\begin{itemize}
\item\mytb{Option 1.} If $u^k := Fx^k$, then we obtain $y^k := J_{\frac{\eta}{\beta}T}(x^k - \frac{\eta}{\beta} Fx^k)$, leading to the well-known \mytb{extragradient method} (or \mytb{extragradient-plus} -- EG+ in \cite{diakonikolas2021efficient} if $\beta \in (0, 1)$) for solving \eqref{eq:NI}.
\item\mytb{Option 2.} If $u^k := Fy^{k-1}$, then we obtain $y^k := J_{\frac{\eta}{\beta}T}(x^k - \frac{\eta}{\beta} Fy^{k-1})$, leading to the \mytb{past-extragradient method} (or equivalently, \mytb{Popov's method} \cite{popov1980modification}) for solving \eqref{eq:NI}.
\end{itemize}
Clearly, when $T = \Nc_{\Xc}$, the normal cone of a nonempty, closed, and convex set $\Xc$, then $J_{\gamma T} = \proj_{\Xc}$, the projection onto $\Xc$ and hence, \eqref{eq:EG4NI} reduces to the extragradient variant for solving \eqref{eq:VIP} widely studied in the literature \cite{Facchinei2003,Konnov2001}.
In terms of computational complexity, \eqref{eq:EG4NI} requires two evaluations of $F$ at $x^k$ and $y^k$, and two evaluations of the resolvent $J_{\eta T}$ at each iteration.
It costs as twice as one iteration of the forward-backward splitting method \eqref{eq:FBS4NI}.
However, its does not require the co-coerciveness of $F$ to guarantee convergence.
Again, we use a scaling factor $\beta$ as in \eqref{eq:EG4NE}, which covers EG+ in \cite{diakonikolas2021efficient} as a special case.

Now, for given $\zeta^k \in Ty^k$ and $\xi^{k+1} \in Tx^{k+1}$, we denote  $\tilde{w}^k := Fx^k + \zeta^k$, and $\hat{w}^{k+1} := Fy^k + \xi^{k+1}$.
Then, we can rewrite \eqref{eq:EG4NI} equivalently to
\begin{equation}\label{eq:EG4NI_reform}
\arraycolsep=0.2em
\left\{\begin{array}{lclclcll}
y^k &:= & x^k - \frac{\eta}{\beta}(u^k + \zeta^k) & = & x^k - \frac{\eta}{\beta}(\tilde{w}^k + u^k - Fx^k), \quad &\zeta^k & \in & Ty^k, \vspace{1ex}\\
x^{k+1} &:= & x^k - \eta (Fy^k + \xi^{k+1}) & = & x^k - \eta\hat{w}^{k+1}, \quad & \xi^{k+1} & \in & Tx^{k+1}.
\end{array}\right.
\end{equation}
This representation makes \eqref{eq:EG4NI} looks like \eqref{eq:EG4NE}, and it is a key step for our convergence analysis.

%%%%% 4.3. One-iteration analysis
\beforesubsec
\subsection{One-iteration analysis}\label{subsec:EG4NI_key_lemmas}
\aftersubsec
We establish both the best-iterate and last-iterate convergence rates of \eqref{eq:EG4NI} under the assumption that $F$ is monotone and $T$ is maximally $3$-cyclically monotone.
Note that if $T$ is maximally cyclically monotone, then $T = \partial{g}$, the subdifferential of a proper, closed, and convex function due to \cite[Theorem 22.18]{Bauschke2011}.
In this case, \eqref{eq:NI} reduces to \eqref{eq:MVIP}.
However, we do not require $T$ to be maximally cyclically monotone, but only $3$-maximally cyclically monotone, which may not be necessarily identical to $\partial{g}$.
Therefore, our result below is more general than existing variants in the recent literature, including \cite{cai2022tight}.

To analyze the convergence of \eqref{eq:EG4NI}, we also define
\begin{equation}\label{eq:EG4NI_w}
w^k := Fx^k + \xi^k \quad \text{for some}\quad \xi^k \in Tx^k.
\end{equation}
The following lemma provides key estimates to establish convergence of \eqref{eq:EG4NI}.

%%% Lemma 1. 
\begin{lemma}\label{le:EG4NI_key_estimate}
Suppose that $\sets{(x^k, y^k)}$ is generated by \eqref{eq:EG4NI}, $w^k$ is defined by \eqref{eq:EG4NI_w} and $T$ is maximally $3$-cyclically monotone.
Then, for any $\gamma > 0$, any $x^{\star} \in \zer{\Phi}$, we have
\begin{equation}\label{eq:EG4NI_key_est1}
\arraycolsep=0.2em
\begin{array}{lcl}
\norms{x^{k+1} - x^{\star}}^2 & \leq & \norms{x^k - x^{\star}}^2 -  (1 - \beta)\norms{x^{k+1} - x^k}^2 -  (\beta - \gamma)\norms{x^{k+1} - y^k}^2 \vspace{1ex}\\
&& - {~} \beta\norms{x^k - y^k}^2 + \frac{\eta^2}{\gamma}\norms{Fy^k - u^k}^2 - 2\eta\iprods{Fy^k - Fx^{\star}, y^k - x^{\star}}.
\end{array}
\end{equation}
If, in addition, $F$ is monotone and $\beta := 1$, then for $\omega \geq 1$, $\gamma > 0$, and $t > 0$, we have
\begin{equation}\label{eq:EG4NI_monotone_est1}
\arraycolsep=0.2em
\begin{array}{lcl}
\norms{w^{k+1}}^2 + (\omega - 1)\norms{w^{k+1} - \hat{w}^{k+1}}^2 &\leq & \norms{w^k}^2 - (1 - \gamma)\norms{w^k -  \tilde{w}^k}^2  + \left[\frac{1}{\gamma} + \frac{\omega(1+t) L^2\eta^2}{t} \right] \norms{Fx^k - u^k}^2  \vspace{1ex}\\
&& - {~}  \big[ 1 -  \omega (1+t) L^2\eta^2 \big] \norms{\hat{w}^{k+1} - \tilde{w}^k}^2.
\end{array}
\end{equation}
\end{lemma}

%%% Proof of Lemma 4.1.
\begin{proof}
Firstly, since $\xi^{k+1} \in Tx^{k+1}$, $\zeta^k \in Ty^k$, and  $\xi^{\star} = -Fx^{\star} \in Tx^{\star}$,  by the maximally $3$-cyclic monotonicity of $T$, we have $\iprods{\xi^{k+1}, x^{k+1} - x^{\star}} + \iprods{\xi^{\star}, x^{\star} - y^k} + \iprods{\zeta^k, y^k - x^{k+1}} \geq 0$, leading to $\iprods{\xi^{k+1} - \zeta^k, x^{k+1} - x^{\star}}  \geq \iprods{\zeta^k - \xi^{\star}, x^{\star} - y^k} = -\iprods{Fx^{\star} + \zeta^k, y^k - x^{\star}}$.
Utilizing this inequality and the second line $x^k - x^{k+1} = \eta (Fy^k + \xi^{k+1})$ of \eqref{eq:EG4NI_reform}, for any $x^{\star} \in \zer{\Phi}$, we can derive that
\begin{equation*}
\arraycolsep=0.2em
\begin{array}{lcl}
\norms{x^{k+1} - x^{\star}}^2 & = & \norms{x^k - x^{\star}}^2 - 2\iprods{x^k - x^{k+1}, x^{k+1} - x^{\star}} - \norms{x^{k+1} - x^k}^2 \vspace{1ex}\\
&= & \norms{x^k - x^{\star}}^2 - 2\eta\iprods{Fy^k + \xi^{k+1}, x^{k+1} - x^{\star}} - \norms{x^{k+1} - x^k}^2 \vspace{1ex}\\
&= & \norms{x^k -x^{\star}}^2 - 2\eta\iprods{Fy^k + \zeta^k, x^{k+1} - x^{\star}} - \norms{x^{k+1} - x^k}^2 - 2\eta\iprods{\xi^{k+1} - \zeta^k, x^{k+1} - x^{\star}} \vspace{1ex}\\
&\leq & \norms{x^k - x^{\star}}^2  - \norms{x^{k+1} - x^k}^2 - 2\eta\iprods{Fy^k + \zeta^k, x^{k+1} - y^k} - 2\eta\iprods{Fy^k - Fx^{\star}, y^k - x^{\star}}.
\end{array}
\end{equation*}
Next, from the first line of  \eqref{eq:EG4NI_reform}, we have $\eta(Fy^k + \zeta^k) = \beta(x^k - y^k) + \eta(Fy^k - u^k)$.
Therefore, by the Cauchy-Schwarz inequality and Young's inequality, for any $\gamma > 0$, we can derive that
\begin{equation*}
\arraycolsep=0.2em
\begin{array}{lcl}
2\eta\iprods{Fy^k + \zeta^k, x^{k+1} - y^k} &= & 2\beta \iprods{x^k - y^k, x^{k+1} - y^k} + 2\eta\iprods{Fy^k - u^k, x^{k+1} - y^k} \vspace{1ex}\\
&\geq & \beta\left[\norms{x^k - y^k}^2 + \norms{x^{k+1} - y^k}^2 - \norms{x^{k+1} - x^k}^2\right]  -  2\eta\norms{Fy^k - u^k}\norms{x^{k+1} - y^k} \vspace{1ex}\\
&\geq & \beta \norms{x^k - y^k}^2 + (\beta - \gamma) \norms{x^{k+1} - y^k}^2 - \beta \norms{x^{k+1} - x^k}^2 - \frac{\eta^2}{\gamma}\norms{Fy^k - u^k}^2.
\end{array}
\end{equation*}
Finally, substituting this inequality into the above estimate, we obtain \eqref{eq:EG4NI_key_est1}.

To prove \eqref{eq:EG4NI_monotone_est1}, we process as follows.
Using again the $3$-cyclic monotonicity of $T$ but with $\xi^k \in Tx^k$, we have $\iprods{\xi^{k+1}, x^{k+1} - x^k}  + \iprods{\xi^k, x^k - y^k} + \iprods{\zeta^k, y^k - x^{k+1}}  \geq 0$.
By the monotonicity of $F$, we get $\iprods{Fx^{k+1} - Fx^k, x^{k+1} - x^k} \geq 0$.
Summing up these inequalities and using $w^k = Fx^k + \xi^k$ and $\tilde{w}^k := Fx^k + \zeta^k$, we have 
\begin{equation}\label{eq:EG4NI_lm2_proof1}
\iprods{w^{k+1} - \tilde{w}^k, x^{k+1} - x^k} + \iprods{w^k - \tilde{w}^k, x^k - y^k}  \geq 0.
\end{equation}
From the second line of \eqref{eq:EG4NI}, we have $x^{k+1} - x^k = -\eta(Fy^k + \xi^{k+1}) = -\eta\hat{w}^{k+1}$.
From the first line of  \eqref{eq:EG4NI} and $\beta = 1$, we also have $x^k - y^k = \frac{\eta}{\beta}(\tilde{w}^k + u^k - Fx^k) = \eta \tilde{w}^k + \eta(u^k - Fx^k)$.
Substituting these expressions into \eqref{eq:EG4NI_lm2_proof1}, and using an elementary inequality $2\iprods{z, s} \leq \gamma \norms{s}^2 +  \frac{\norms{z}^2}{\gamma}$ for any $\gamma > 0$, we have
\begin{equation*} 
\arraycolsep=0.2em
\begin{array}{lcl}
0 &\leq & 2 \iprods{\tilde{w}^k, \hat{w}^{k+1}} - 2 \iprods{w^{k+1}, \hat{w}^{k+1}}  + 2\iprods{w^k, \tilde{w}^k} - 2\norms{\tilde{w}^k}^2 + 2\iprods{w^k - \tilde{w}^k, u^k - Fx^k} \vspace{1ex}\\
&= & \norms{w^k}^2  - \norms{w^{k+1}}^2  + \norms{w^{k+1} - \hat{w}^{k+1}}^2 - \norms{\hat{w}^{k+1} - \tilde{w}^k}^2  - \norms{w^k - \tilde{w}^k}^2  + 2\iprods{w^k - \tilde{w}^k, Fx^k - u^k} \vspace{1ex}\\
%&= & \norms{\tilde{w}^k}^2 +  \norms{\hat{w}^{k+1}}^2 - \norms{\hat{w}^{k+1} - \tilde{w}^k}^2 - \norms{w^{k+1}}^2 - \norms{\hat{w}^{k+1}}^2 + \norms{w^{k+1} - \hat{w}^{k+1}}^2 \vspace{1ex}\\
%&& + {~}  \norms{w^k}^2 + \norms{\tilde{w}^k}^2 - \norms{w^k - \tilde{w}^k}^2] - 2\norms{\tilde{w}^k}^2 + 2\norms{w^k - \tilde{w}^k}\norms{Fx^k - u^k} \vspace{1ex}\\
&\leq & \norms{w^k}^2 - \norms{w^{k+1}}^2  + \norms{w^{k+1} - \hat{w}^{k+1}}^2 - \norms{\hat{w}^{k+1} - \tilde{w}^k}^2 - (1 - \gamma)\norms{w^k - \tilde{w}^k}^2  + \frac{1}{\gamma}\norms{Fx^k - u^k}^2.
\end{array}
\end{equation*}
This inequality leads to 
\begin{equation*}%\label{eq:EG4NI_lm2_proof2} 
\arraycolsep=0.2em
\begin{array}{lcl}
\norms{w^{k+1}}^2 &\leq & \norms{w^k}^2 + \norms{w^{k+1} - \hat{w}^{k+1}}^2 + \frac{1}{\gamma}\norms{Fx^k - u^k}^2 -  (1- \gamma)\norms{w^k - \tilde{w}^k}^2 - \norms{\hat{w}^{k+1} - \tilde{w}^k}^2.
\end{array}
\end{equation*}
Now, by the $L$-Lipschitz continuity of $F$, \eqref{eq:EG4NI}, and Young's inequality, for any $t > 0$, we have 
\begin{equation*} 
\arraycolsep=0.2em
\begin{array}{lcl}
\norms{w^{k+1} - \hat{w}^{k+1}}^2 & = & \norms{Fx^{k+1} - Fy^k}^2 \leq L^2\norms{x^{k+1} - y^k}^2 = L^2\eta^2\norms{\hat{w}^{k+1} - \tilde{w}^k + Fx^k - u^k}^2 \vspace{1ex}\\
& \leq & (1 + t)L^2\eta^2\norms{\hat{w}^{k+1} - \tilde{w}^k }^2 + \frac{(1 + t)L^2\eta^2}{t}\norms{Fx^k - u^k}^2.
\end{array}
\end{equation*}
Multiplying this inequality by $\omega \geq 1$ and adding the last inequality, we obtain \eqref{eq:EG4NI_monotone_est1}.
%\begin{equation}\label{eq:EG4NI_monotone_est1}
%\arraycolsep=0.2em
%\begin{array}{lcl}
%\norms{w^k}^2 + \left[\frac{1}{\gamma} + \frac{\omega(1+t) L^2\eta^2}{t} \right] \norms{Fx^k - u^k}^2 & \geq & \norms{w^{k+1}}^2 + \big[ 1 -  \omega (1+t) L^2\eta^2 \big] \norms{\hat{w}^{k+1} - \tilde{w}^k}^2 \vspace{1ex}\\
%&& + {~} (\omega - 1)\norms{w^{k+1} - \hat{w}^{k+1}}^2 +  (1 - \gamma)\norms{w^k -  \tilde{w}^k}^2.
%\end{array}
%\end{equation}
%%for any $r > 0$, $c > 0$, and $\omega \geq 1$.
%%Here $\hat{w}^{k+1} := Fy^k + \xi^{k+1}$ for  \eqref{eq:EG4NI}, and $\hat{w}^{k+1} := Fy^k + \zeta^k$ for \eqref{eq:FBFS4NI}.
%%
%Finally, the first line of \eqref{eq:EG4NI_monotone_est2} is obtained from \eqref{eq:EG4NI_monotone_est1} with $r = c = 0$ and $\omega = 1$.
%The second line of \eqref{eq:EG4NI_monotone_est2} is obtained from \eqref{eq:EG4NI_monotone_est1} with $r := 1$, $c := \frac{1}{8}$, $\omega := \frac{2}{1 - 9L^2\eta^2}$, and the fact that $\norms{Fx^k - u^k} = \norms{Fx^k - Fy^{k-1}} = \norms{w^k - \hat{w}^k}$.
%Here, we assume that $3L\eta < 1$.
\end{proof}
%%% End of proof.

%%%% 4.2. The convergence of the extragradient method.
\beforesubsec
\subsection{Unified convergence analysis}\label{subsec:EG4NI_analysis}
\aftersubsec
The following theorem proves the best-iterate and the last-iterate convergence of  \eqref{eq:EG4NI}.

%%% Theorem 4.1.
\begin{theorem}\label{th:EG4NI_convergence}
Suppose that $\zer{\Phi} \neq \emptyset$, $F$ in \eqref{eq:NI} is $L$-Lipschitz continuous and satisfies $\iprods{Fx - Fx^{\star}, x - x^{\star}} \geq 0$ for all $x\in\dom{F}$ and some $x^{\star} \in \zer{\Phi}$, and $T$ is maximally $3$-cyclically monotone.
Let $\sets{(x^k, y^k)}$ be generated by \eqref{eq:EG4NI}.
Then, the following statements hold.
\begin{itemize}
\item[$\mathrm{(a)}$] \mytb{$($EG method$)$} 
If we  choose $u^k := Fx^k$ and  $0 < \eta < \frac{\beta}{L}$, then we have
%\begin{equation}\label{eq:EG4NI_stepsize} 
%\arraycolsep=0.2em
%\begin{array}{lcl}
%0 \leq \frac{\beta \big[ 1 - \sqrt{1 - 24 L\rho(3L\rho + 1)} \big]}{2L(3L\rho + 1)} < \eta < \frac{\beta \big[ 1 + \sqrt{1 - 24 L\rho(3L\rho + 1)} \big]}{2L(3L\rho + 1)} \leq \frac{\beta}{L},
%\end{array}
%\end{equation}
%Then, we have
\begin{equation}\label{eq:EG4NI_convergence_02}
\min_{1\leq l\leq k+1}\norms{Fx^l + \xi^l}^2 \leq \frac{1}{k+1}\sum_{l=1}^{k+1}\norms{Fx^l + \xi^l}^2 \leq \frac{C_0\norms{x^0 - x^{\star}}^2}{k+1}, \quad \xi^l \in Tx^l,
\end{equation} 
where $C_0 :=  \frac{3 + 2L^2}{\eta^2(\beta  - L\eta)} > 0$.
As a consequence, $\sets{\norms{x^k - x^{\star}}}$ is nonincreasing and
\begin{equation}\label{eq:EG4NI_convergence_01}
\lim_{k\to\infty}\norms{x^k - y^k} = \lim_{k\to\infty}\norms{Fy^k + \zeta^k} =  \lim_{k\to\infty}\norms{Fx^k + \xi^k} = 0.
\end{equation}
Moreover, $\sets{x^k}$ converges to $x^{\star}$, a solution of \eqref{eq:NI}.

\item[]$($\textbf{\textit{Last-iterate convergence rate of EG}}$)$ If, in addition, $F$ is monotone, then we have 
\begin{equation}\label{eq:EG4NI_monotone_est1}
\arraycolsep=0.2em
\begin{array}{lcl}
\norms{Fx^{k+1} + \xi^{k+1}}^2 \leq \norms{Fx^k + \xi^k}^2 \quad \text{and} \quad  \norms{Fx^k + \xi^k} \leq \frac{\sqrt{C_0}\norms{x^0 - x^{\star}}}{\sqrt{k}}.
\end{array}
\end{equation}
Hence, we have a last-iterate convergence rate $\norms{Fx^k + \xi^k} = \BigOs{1/\sqrt{k}}$ of $\norms{Fx^k + \xi^k}$, where $\xi^k \in Tx^k$.

\item[$\mathrm{(b)}$] \mytb{$($Past-EG method$)$} 
If we choose $u^k := Fy^{k-1}$ with $y^{-1} := x^0$ and $0 < \eta < \frac{\beta}{3L}$, then
\begin{equation}\label{eq:PEG4NI_convergence_03}
\min_{1\leq l\leq k+1}\norms{Fx^l + \xi^l}^2 \leq \frac{1}{k+1}\sum_{l=1}^{k+1} \left[ \norms{Fx^l + \xi^l}^2 + \psi \cdot \norms{x^l - y^{l-1}}^2 \right] \leq \frac{\hat{C}_0 \norms{x^0 - x^{\star}}^2}{k+1},
\end{equation} 
where $\hat{C}_0 := \frac{3 + 2L^2}{\eta^2(\beta -  3L\eta)} > 0$ and $\psi :=  \frac{L(3+2L^2)}{2\eta(\beta - 3L\eta)}$.
Moreover, $\sets{x^k}$ converges to  $x^{\star}$ and \eqref{eq:EG4NI_convergence_01} still holds.

\item[]$($\textbf{\textit{Last-iterate convergence rate of Past-EG}}$)$ If, in addition, $F$ is monotone, then we have 
\begin{equation}\label{eq:EG4NI_monotone_est2}
\arraycolsep=0.4em
\begin{array}{ll}
& \norms{Fx^{k+1} + \xi^{k+1}}^2 + \kappa\norms{Fx^{k+1} - Fy^k}^2 \leq \norms{Fx^k + \xi^k}^2 + \kappa\norms{Fx^k - Fy^{k-1}}^2 \vspace{1ex}\\
\text{and} & \norms{Fx^k + \xi^k}^2 \leq \norms{Fx^k + \xi^k}^2  + \kappa\norms{Fx^k - Fy^{k-1}}^2 \leq \frac{\hat{C}_0\norms{x^0 - x^{\star}}^2}{m_0 k},
\end{array}
\end{equation}
where $\kappa := \frac{1+ 9L^2\eta^2}{1 - 9L^2\eta^2} > 0$ and $m_0 := \max\set{\frac{\kappa}{\psi}, 1}$.
Therefore, we have  $\norms{Fx^k + \xi^k} = \BigOs{1/\sqrt{k}}$.
\end{itemize}
In both cases $\mathrm{(a)}$ and $\mathrm{(b)}$, if $T$ is maximally monotone, then we also have
\begin{equation}\label{eq:EG4NI_convergence_02_03}
\min_{1\leq l \leq k}\norms{G_{\eta\Phi}x^l} = \BigO{\frac{1}{\sqrt{k}}} \quad \text{and} \quad \min_{1\leq l \leq k}\norms{G_{\eta\Phi}y^l} = \BigO{\frac{1}{ \sqrt{k}}}, 
\end{equation}
where $G_{\eta\Phi}x := \frac{1}{\eta}(x - J_{\eta T}(x - \eta Fx))$ is given by \eqref{eq:FB_residual}.
Moreover, $\lim_{k\to\infty}\norms{G_{\eta\Phi}y^k} = \lim_{k\to\infty}\norms{G_{\eta\Phi}x^k} = 0$.
\end{theorem}

%%% Remark 4.1.
\begin{remark}\label{re:EG4NI_remark1}
If $\beta = 1$, then the condition $0 < \eta < \frac{1}{L}$ in Part (a) is the same as in classical EG methods \cite{Facchinei2003}.
Similarly, when $\beta = 1$, the condition $0 < \eta \leq \frac{1}{3L}$ in Part (b) is the same as the one in \cite{popov1980modification}.
It remains open to establish both best-iterate and last-iterate convergence rates of \eqref{eq:EG4NI} under weak-Minty solution assumption or the co-hypomonotonicity of $\Phi$.
\end{remark}

%%%% Proof of Theorem 4.1
\begin{proof}[Proof of Theorem~\ref{th:EG4NI_convergence}]
(a)~\mytb{(EG method)} 
First, since $u^k := Fx^k$, by the $L$-Lipschitz continuity of $F$,  we have $\norms{Fy^k - u^k} = \norms{Fy^k - Fx^k} \leq L\norms{x^k - y^k}$.
Using this inequality and $\iprods{Fy^k - Fx^{\star}, y^k - x^{\star}} \geq 0$ from our assumption into \eqref{eq:EG4NI_key_est1}, we have
\begin{equation}\label{eq:EG4NI_th41_proof1}
\arraycolsep=0.2em
\begin{array}{lcl}
\norms{x^{k+1} - x^{\star}}^2 & \leq & \norms{x^k - x^{\star}}^2 -  (1 - \beta)\norms{x^{k+1} - x^k}^2 -  (\beta - \gamma)\norms{x^{k+1} - y^k}^2  - \left(\beta - \frac{L^2\eta^2}{\gamma}\right)\norms{x^k - y^k}^2.
\end{array}
\end{equation}
Next, using the second line of \eqref{eq:EG4NI_reform},  Young's inequality, and the $L$-Lipschitz continuity of $F$, we have
\begin{equation}\label{eq:EG4NI_th41_proof2}
\arraycolsep=0.2em
\begin{array}{lcl}
\eta^2\norms{Fx^{k+1} + \xi^{k+1}}^2 & \overset{\tiny\eqref{eq:EG4NI_reform}}{ = } & \norms{x^{k+1} - x^k + \eta(Fx^{k+1} - Fy^k)}^2 \vspace{1ex}\\
&\leq & (1 + \frac{2L^2}{3})\norms{x^{k+1} - x^k}^2 + (1 + \frac{3}{2L^2}) \norms{Fx^{k+1} - Fy^k}^2 \vspace{1ex}\\
&\leq &\frac{3}{2}(1 + \frac{2L^2}{3})\norms{x^{k+1} - y^k}^2 + (1 + \frac{3}{2L^2})L^2\norms{x^{k+1} - y^k}^2 +3(1 + \frac{2L^2}{3})\norms{x^k - y^k}^2 \vspace{1ex}\\ 
&= & (3 + 2L^2) [ \norms{x^{k+1} - y^k}^2 + \norms{x^k - y^k}^2].
\end{array}
\end{equation}
%Next, since $\iprods{u, x - x^{\star}} \geq - \rho \norms{u}^2$ for any $(x, u) \in \gra{\Phi}$,  we have  $\iprods{Fy^k + \zeta^k, y^k - x^{\star}} \geq - \rho \norms{Fy^k + \zeta^k}^2$ since $(y^k, \zeta^k) \in \gra{T}$.
Substituting \eqref{eq:EG4NI_th41_proof2}  into \eqref{eq:EG4NI_th41_proof1} with  $\gamma := L\eta$, and assuming that $L\eta < \beta$, we get
\begin{equation*}%\label{eq:EG4NI_th41_proof3}
\arraycolsep=0.2em
\begin{array}{lcl}
\norms{x^{k+1} - x^{\star}}^2 %& \leq & \norms{x^k - x^{\star}}^2 -  (\beta -   L\eta) \norms{x^k - y^k}^2 - (\beta -   L\eta) \norms{x^{k+1} - y^k}^2 + 2\eta\rho\norms{Fy^k + \zeta^k}^2 \vspace{1ex}\\
&\leq & \norms{x^k - x^{\star}}^2 -  (\beta - L\eta) \big[ \norms{x^{k+1} - y^k}^2  +  \norms{x^k - y^k}^2 \big] \vspace{1ex}\\
& \leq & \norms{x^k - x^{\star}}^2 -  \frac{(\beta - L\eta)\eta^2}{3+2L^2}\norms{Fx^{k+1} + \xi^{k+1}}^2.
\end{array}
\end{equation*}
%Let us impose the condition that $\beta\eta - 6\beta^2\rho - (3L\rho + 1)L\eta^2 > 0$, which holds if
%\begin{equation*} 
%\arraycolsep=0.2em
%\begin{array}{lcl}
%0 \leq \frac{\beta \big[ 1 - \sqrt{1 - 24 L\rho(3L\rho + 1)} \big]}{2L(3L\rho + 1)} < \eta < \frac{\beta \big[ 1 + \sqrt{1 - 24 L\rho(3L\rho + 1)} \big]}{2L(3L\rho + 1)} \leq \frac{\beta}{L},
%\end{array}
%\end{equation*}
%provided that  $L\rho \leq \frac{3\sqrt{2} -2}{12} \approx 0.1869$. 
%This choice of $\eta$ is exactly get \eqref{eq:EG4NI_stepsize}.
%Let $\psi := \tfrac{\beta\eta - 6\beta^2\rho - (3L\rho + 1)L\eta^2}{\eta} = \beta - L\eta - \frac{3\rho}{\eta}(2\beta^2 + L^2\eta^2) > 0$.
%Then, we have $\beta - L\eta - \psi \geq 0$.
%In this case, \eqref{eq:EG4NI_th41_proof3} becomes $\norms{x^{k+1} - x^{\star}}^2 \leq \norms{x^k - x^{\star}}^2 - \psi[\norms{x^k - y^k}^2 + \norms{x^{k+1} - y^k}^2] - (\beta - L\eta - \psi)\norms{x^{k+1} - y^k}^2 \leq  \norms{x^k - x^{\star}}^2 - \psi[\norms{x^k - y^k}^2 + \norms{x^{k+1} - y^k}^2]$.
Now, using this estimate, we can easily prove \eqref{eq:EG4NI_convergence_02} in  Theorem~\ref{th:EG4NI_convergence}.
Note that, by \eqref{eq:FBR_bound2}, Young's inequality, the $L$-Lipschitz continuity of $F$, and $x^k - y^k = \frac{\eta}{\beta}(u^k + \zeta^k) = \frac{\eta}{\beta}(Fx^k + \zeta^k)$ from the first line of \eqref{eq:EG4NI_reform}, we have
\begin{equation}\label{eq:EG4NI_th41_proof2a}  
\arraycolsep=0.2em
\begin{array}{lcl}
\norms{Fy^k + \zeta^k}^2 & \leq & \frac{3}{2} \norms{Fy^k - Fx^k}^2 + 2\norms{Fx^k + \zeta^k}^2 \leq \frac{3(L^2\eta^2 + 2\beta^2)}{2\eta^2}\norms{x^k - y^k}^2.  
\end{array}
\end{equation}
Therefore, the remaining statements are direct consequences of  \eqref{eq:EG4NI_convergence_02},  \eqref{eq:EG4NI_th41_proof2a}, and \eqref{eq:EG4NI_th41_proof2} using standard arguments.

Finally, since $\beta = 1$, using $u^k := Fx^k$, $\omega := 1$, and $t = \gamma \to 0^{+}$ into \eqref{eq:EG4NI_monotone_est1}, and noting that $L\eta \leq 1$, we get 
\begin{equation*} 
\arraycolsep=0.2em
\begin{array}{lcl}
\norms{w^{k+1}}^2 &\leq & \norms{w^k}^2  - (1 -   L^2\eta^2)\norms{\hat{w}^{k+1} - \tilde{w}^k}^2 -  \norms{w^k -  \tilde{w}^k}^2 \leq \norms{w^k}^2.
\end{array}
\end{equation*}
This shows that $\set{\norms{w^k}}$ is monotonically nonincreasing.
Combining this property and \eqref{eq:EG4NI_convergence_02}, we obtain \eqref{eq:EG4NI_monotone_est1}.

(b)~\mytb{(Past-extragradient method)} 
If we choose $u^k := Fy^{k-1}$, then by Young's inequality and the $L$-Lipschitz continuity of $F$, similar to the proof of \eqref{eq:EG4NE_proof_p2_1}, we have 
\begin{equation*}
\arraycolsep=0.2em
\begin{array}{lcl}
\norms{Fy^k - u^k}^2 = \norms{Fy^k - Fy^{k-1}}^2 \leq 3L^2\norms{x^k - y^k}^2 + \frac{3L^2}{2}\norms{x^k - y^{k-1}}^2.
\end{array}
\end{equation*} 
Substituting this expression into \eqref{eq:EG4NI_key_est1}, using $\iprods{Fy^k - Fx^{\star}, y^k - x^{\star}} \geq 0$, and choosing $\gamma := L\eta$, we obtain
\begin{equation}\label{eq:EG4NI_th41p2_proof3}
\arraycolsep=0.2em
\begin{array}{lcl}
\norms{x^{k+1} - x^{\star}}^2 + \frac{3L\eta}{2} \norms{x^{k+1} - y^k}^2 & \leq &  \norms{x^k - x^{\star}}^2 + \frac{3L\eta}{2} \norms{x^k - y^{k-1}}^2 - (1 - \beta)\norms{x^{k+1} - x^k}^2 \vspace{1ex}\\
&& - {~}   \frac{L\eta}{2}\norms{x^{k+1} - y^k}^2 - ( \beta -  3L\eta)\big[ \norms{x^{k+1} - y^k}^2  +  \norms{x^k - y^k}^2\big].
\end{array}
\end{equation}
%Let us define $\hat{\psi} :=  \beta -  3L\eta - \frac{4\rho(\beta^2 + 2L^2\eta^2)}{\eta}$ and $\hat{\varphi} := \beta - \frac{5L\eta}{2} - 8\rho L^2\eta$ and impose the condition that $\hat{\psi} > 0$.
%Then, it is clear that $\hat{\varphi} - \hat{\psi} = \frac{L\eta}{2} + \frac{4\beta^2\rho}{\eta} \geq 0$.
%Moreover, \eqref{eq:EG4NI_th41p2_proof3} becomes $\Vc_{k+1} \leq \Vc_k - \hat{\psi}\big[ \norms{x^{k+1} - y^k}^2 + \norms{x^k - y^k}^2 \big] - (\hat{\varphi} - \hat{\psi})\norms{x^{k+1} - y^k}^2$.
Assuming that $3L\eta < \beta$.
Then, combining  \eqref{eq:EG4NI_th41p2_proof3} and \eqref{eq:EG4NI_th41_proof2}, we can easily prove \eqref{eq:PEG4NI_convergence_03} in  Theorem~\ref{th:EG4NI_convergence}.
%Note that the condition $\hat{\psi} > 0$ holds if $\eta$ is chosen as in \eqref{eq:PEG4NI_stepsize} provided that $L\rho \leq \frac{2\sqrt{3} - 3}{24} \approx 0.01934$.
Moreover,  \eqref{eq:EG4NI_th41p2_proof3} also implies $\lim_{k\to\infty}\norms{x^k - y^k} = \lim_{k\to\infty}\norms{x^{k+1} - y^k} = \lim_{k\to\infty}\norms{Fx^k + \xi^k} = 0$.
By the second line of \eqref{eq:EG4NI_reform}, we have $\eta\norms{Fy^k + \zeta^k} \leq L\eta\norms{x^k - y^k} + L\eta\norms{x^k - y^{k-1}} + \norms{\eta(Fy^{k-1} + \zeta^k} = (L\eta + 1)\norms{x^k - y^k} + L\eta\norms{x^k - y^{k-1}}$.
Using this relation and $\lim_{k\to\infty}\norms{x^k - y^k} = \lim_{k\to\infty}\norms{x^{k+1} - y^k} =  0$, we obtain $\lim_{k\to\infty}\norms{Fy^k + \zeta^k} = 0$.
The convergence of $\sets{x^k}$ to  $x^{\star}$ follows from standard arguments.

Next, assume that $3L\eta < 1$ and $u^k := Fy^{k-1}$.
Then, substituting $\gamma := 1$, $t := \frac{1}{8}$, $\omega := \frac{2}{1 - 9L^2\eta^2} > 1$ into \eqref{eq:EG4NI_monotone_est1}, and using $\norms{Fx^k - u^k} = \norms{Fx^k - Fy^{k-1}} = \norms{w^k - \hat{w}^k}$, we obtain
\begin{equation*} 
\arraycolsep=0.2em
\begin{array}{lcl}
 \norms{w^{k+1}}^2 + \kappa \norms{w^{k+1} - \hat{w}^{k+1}}^2 & \leq & \norms{w^k}^2 + \kappa \norms{w^k - \hat{w}^k}^2 - \frac{4 - 27L^2\eta^2}{4(1-9L^2\eta^2)}\norms{\hat{w}^{k+1} - \tilde{w}^k}^2.
\end{array}
\end{equation*}
This is exactly  first line of \eqref{eq:EG4NI_monotone_est2}.
Since $3L\eta < 1$, we have $\kappa := \frac{1 + 9L^2\eta^2}{1 - 9L^2\eta^2} > 0$ and $ \frac{4 - 27L^2\eta^2}{4(1-9L^2\eta^2)} > 0$.
Let $m_0 := \max\sets{\frac{\kappa}{\psi}, 1}$, where $\psi$ is given in \eqref{eq:PEG4NI_convergence_03}.
Then, we have 
\begin{equation*}
\arraycolsep=0.2em
\begin{array}{lcl}
\norms{Fx^k + \xi^k}^2 + \kappa\norms{Fx^k - Fy^{k-1}}^2  & \leq & \norms{Fx^k + \xi^k}^2 + \kappa L^2\norms{x^k - y^{k-1}}^2 \leq  m_0\big[\norms{Fx^k + \xi^k}^2  + \psi \cdot \norms{x^k - y^{k-1}}^2 \big].
\end{array}
\end{equation*}
Combining this inequality and \eqref{eq:PEG4NI_convergence_03}, we obtain  $\frac{1}{k+1}\sum_{l=1}^{k+1} \left[\norms{Fx^l + \xi^l}^2 + \kappa\norms{Fx^l - Fy^{l-1}}^2  \right] \leq \frac{\hat{C}_0\norms{x^0 - x^{\star}}^2}{m_0(k+1)}$.
This bound together with the first line of \eqref{eq:EG4NI_monotone_est2} imply the second line of \eqref{eq:EG4NI_monotone_est2}.

Finally, since $T$ is maximally monotone,  $J_{\eta T}$ is single-valued and nonexpansive.
By using $\norms{G_{\Phi}x^k} \leq \norms{Fx^k + \xi^k}$ from \eqref{eq:FBR_bound2} and either \eqref{eq:EG4NI_convergence_02} or \eqref{eq:PEG4NI_convergence_03}, we obtain \eqref{eq:EG4NI_convergence_02_03}.
Using again \eqref{eq:FBR_bound2}  and the limits \eqref{eq:EG4NI_convergence_01}, we obtain $\lim_{k\to\infty}\norms{G_{\Phi}x^k} \leq \lim_{k\to\infty}\norms{Fx^k + \xi^k} = 0$ and $\lim_{k\to\infty}\norms{G_{\Phi}y^k} \leq \lim_{k\to\infty}\norms{Fy^k + \zeta^k} = 0$.
\end{proof}
\beforesec
\section{Forward-Backward-Forward Splitting-Type Methods for \eqref{eq:NI}}\label{sec:FBFS4NI}
\aftersec
Alternative to \eqref{eq:EG4NI}, we now survey recent results on the best-iterate  convergence rates of the FBFS method and its variants for solving \eqref{eq:NI}.
As before, we provide a unified analysis that covers a wide class of FBFS variants, which can also solve \eqref{eq:NI} under a weak-Minty solution and particularly, the co-hypomonotonicity. 

%%%% 5.1. The class of forward-backward-forward splitting methods.
\beforesubsec
\subsection{The class of forward-backward-forward splitting methods}\label{subsec:FBF4NI}
\aftersubsec
The forward-backward-forward splitting (FBFS) method was proposed by P. Tseng in \cite{tseng2000modified} for solving \eqref{eq:NI}, which is originally called a \textit{modified forward-backward splitting} method.
This method was developed to solve \eqref{eq:NI} with additional constraints.
Instead of presenting the original scheme in \cite{tseng2000modified}, we  modify it using the idea of EG$+$ in \cite{diakonikolas2021efficient} and combine two variants in one.
Starting from an initial point $x^0 \in \dom{\Phi}$, at each iteration $k\geq 0$, we update
\begin{equation}\label{eq:FBFS4NI}
\arraycolsep=0.2em
\left\{\begin{array}{lcl}
y^k &\in & J_{\frac{\eta}{\beta} T}(x^k - \frac{\eta}{\beta} u^k), \vspace{1ex}\\
x^{k+1} &:= & \beta y^k + (1-\beta)x^k - \eta(Fy^k - u^k),
\end{array}\right.
\tag{FBFS2}
\end{equation}
where $\eta > 0$ is a given stepsize, $\beta \in (0, 1]$ is a scaling factor, and $u^k$ is one of the following choices:
\begin{itemize}
\itemsep=0.1em
\item\mytb{Option 1.} If we choose $u^k := Fx^k$, then we obtain a variant of \mytb{Tseng's FBFS method}.
In particular, if $\beta = 1$, then we get exactly \mytb{Tseng's FBFS method} in \cite{tseng2000modified} for solving \eqref{eq:NI}.
Note that one can extend \eqref{eq:FBFS4NI} to cover the case $\zer{\Phi}\cap\Cc \neq\emptyset$ for some subset $\Cc$ of $\R^p$ as presented in \cite{tseng2000modified}.
Nevertheless, for simplicity, we assume that $\Xc = \R^p$.
As shown in \eqref{eq:FBFS4NE}, if $T = 0$, then \eqref{eq:FBFS4NI} reduces to the classical extragradient method \eqref{eq:EG4NE}.
However, if $T\neq 0$, then \eqref{eq:FBFS4NI} is different from \eqref{eq:EG4NI}.

\item\mytb{Option 2.} If we choose $u^k := Fy^{k-1}$, where $y^{-1} := x^0$, then we obtain a \mytb{past-FBFS variant}.
This variant can also be referred to as a generalized variant of the \mytb{optimistic gradient (OG) method}, see, e.g., \cite{daskalakis2018training,mokhtari2020unified,mokhtari2020convergence}.
If $\beta = 1$, then $y^{k+1} \in J_{\eta T}(x^{k+1} - \eta Fy^k)$ and $x^{k+1} = y^k - \eta(Fy^k - Fy^{k-1})$.
Combining these two expressions, \eqref{eq:FBFS4NI} reduces to 
\begin{equation}\label{eq:FRBS4NI2}
\arraycolsep=0.2em
\begin{array}{lcl}
y^{k+1} &\in &  J_{\eta T}\left(y^k - \eta (2Fy^k - Fy^{k-1}) \right).
\end{array}
\tag{FRBS2}
\end{equation}
This is exactly the \mytb{forward-reflected-backward splitting (FRBS) method} in \cite{malitsky2020forward}.

\end{itemize}
Compared to \eqref{eq:EG4NI}, we do not require $T$ to be monotone in \eqref{eq:FBFS4NI}.
However, to guarantee the well-definedness of $\sets{(x^k, y^k)}$, we need $y^k \in \ran{J_{\eta T}}$ and $y^k \in \dom{F}$.
Hence, we can assume that $\ran{J_{\eta T}}\subseteq\dom{F} = \R^p$ and $\dom{J_{\eta T}} = \R^p$. 
This requirement makes \eqref{eq:FBFS4NI} cover a broader class of problems than \eqref{eq:EG4NI}, and it obviously holds if $T$ is maximally monotone and $F$ is monotone and Lipschitz continuous as in \eqref{eq:EG4NI}.
In addition, \eqref{eq:FBFS4NI} only requires one evaluation of $J_{\eta T}$ instead of two  as in \eqref{eq:EG4NI}, reducing the per-iteration complexity when $J_{\eta T}$ is expensive to evaluate. 
 
Similar to \eqref{eq:EG4NI_reform}, we can rewrite  \eqref{eq:FBFS4NI} equivalently to 
\begin{equation}\label{eq:FBFS4NI_reform}
\arraycolsep=0.2em
\left\{\begin{array}{lcl}
y^k &:= &  x^k - \frac{\eta}{\beta}(u^k + \zeta^k), \quad \zeta^k \in Ty^k, \vspace{1ex}\\
x^{k+1} &:= & x^k + \beta (y^k - x^k) - \eta(Fy^k - u^k).
\end{array}\right.
\end{equation}
This representation is an important step for our convergence analysis below.

%%%%% 5.2. One-iteration analysis
\beforesubsec
\subsection{One-iteration analysis}\label{subsec:FBFS4NI_key_lemmas}
\aftersubsec
The following lemma provides a key estimate to establish convergence of \eqref{eq:FBFS4NI}.

%%% Lemma 5.1. 
\begin{lemma}\label{le:FBFS4NI_key_estimate}
Suppose that  $\sets{(x^k, y^k)}$ is generated by \eqref{eq:FBFS4NI} and $T$ is not necessary monotone, but $\ran{J_{\eta T}} \subseteq \dom{F} = \R^p$ and $\dom{J_{\eta T}} = \R^p$.
Then, for any $\gamma > 0$, any $x^{\star} \in \zer{\Phi}$, we have
\begin{equation}\label{eq:FBFS4NI_key_est1}
\arraycolsep=0.2em
\begin{array}{lcl}
\norms{x^{k+1} - x^{\star}}^2 & \leq & \norms{x^k - x^{\star}}^2 -  (1 - \beta)\norms{x^{k+1} - x^k}^2 -  (\beta - \gamma)\norms{x^{k+1} - y^k}^2 \vspace{1ex}\\
&& - {~} \beta\norms{x^k - y^k}^2 + \frac{\eta^2}{\gamma}\norms{Fy^k - u^k}^2 - 2\eta\iprods{Fy^k + \zeta^k, y^k - x^{\star}}.
\end{array}
\end{equation}
\end{lemma}

%%% Proof of Lemma 4.1.
\begin{proof}
First, combining the first and second lines of  \eqref{eq:FBFS4NI_reform}, we obtain $x^{k+1} = x^k - \beta(x^k - y^k) + \eta(u^k - Fy^k) = x^k - \eta(Fy^k + \zeta^k)$.
Using this relation, we have
\begin{equation*} 
\arraycolsep=0.2em
\begin{array}{lcl}
\norms{x^{k+1} - x^{\star}}^2  &= & \norms{x^k - x^{\star}}^2 - 2\eta\iprods{Fy^k + \zeta^k, x^{k+1} - y^k} - \norms{x^{k+1} - x^k}^2 -  2\eta\iprods{Fy^k + \zeta^k, y^k - x^{\star}}.
\end{array} 
\end{equation*}
Next, from the first line of \eqref{eq:FBFS4NI_reform}, we have $\eta(Fy^k + \zeta^k) = \beta(x^k - y^k) + \eta(Fy^k - u^k)$.
Therefore, by the Cauchy-Schwarz inequality and Young's inequality, we can derive that
\begin{equation*}
\arraycolsep=0.2em
\begin{array}{lcl}
2\eta\iprods{Fy^k + \zeta^k, x^{k+1} - y^k} &= & 2\beta \iprods{x^k - y^k, x^{k+1} - y^k} + 2\eta\iprods{Fy^k - u^k, x^{k+1} - y^k} \vspace{1ex}\\
&\geq & \beta\left[\norms{x^k - y^k}^2 + \norms{x^{k+1} - y^k}^2 - \norms{x^{k+1} - x^k}^2\right] - 2\eta\norms{Fy^k - u^k}\norms{x^{k+1} - y^k} \vspace{1ex}\\
&\geq & \beta \norms{x^k - y^k}^2 + (\beta - \gamma) \norms{x^{k+1} - y^k}^2 - \beta \norms{x^{k+1} - x^k}^2 - \frac{\eta^2}{\gamma}\norms{Fy^k - u^k}^2.
\end{array}
\end{equation*}
Finally, substituting this inequality into the above estimate, we obtain \eqref{eq:FBFS4NI_key_est1}.
\end{proof}
%%% End of Proof.

%%%% 4.2. The convergence of the extragradient method.
\beforesubsec
\subsection{Unified convergence analysis}\label{subsec:FBFS4NI_analysis}
\aftersubsec
The following theorem establishes the best-iterate convergence rates of  \eqref{eq:FBFS4NI} under star-co-hypomonotonicity.

%%% Theorem 4.1.
\begin{theorem}\label{th:FBFS4NI_convergence}
Suppose that $F$ in \eqref{eq:NI} is $L$-Lipschitz continuous, and $\zer{\Phi} \neq \emptyset$.
Suppose additionally that $\Phi$ is $\rho$-star co-hypomonotone $($i.e. $\iprods{u, x - x^{\star}} \geq - \rho \norms{u}^2$ for all $(x, u) \in\gra{\Phi}$ and for any $x^{\star} \in \zer{\Phi}$, where $\rho \geq 0$, 
$T$ is not necessarily monotone, but $\ran{J_{\eta T}} \subseteq \dom{F} = \R^p$ and $\dom{J_{\eta T}} = \R^p$.
Let $\sets{(x^k, y^k)}$ be generated by \eqref{eq:FBFS4NI}.
Then, the following statements hold.
\begin{itemize}
\item[$\mathrm{(a)}$] \mytb{$($FBFS method$)$} 
Let us choose $u^k := Fx^k$ and assume that $L\rho \leq \frac{3\sqrt{2} -2}{12} \approx 0.1869$.
Then, for any $\beta \in (0, 1]$, if $\eta$ is chosen such that
\begin{equation}\label{eq:FBFS4NI_stepsize} 
\arraycolsep=0.2em
\begin{array}{lcl}
0 \leq \frac{\beta \big[ 1 - \sqrt{1 - 24 L\rho(3L\rho + 1)} \big]}{2L(3L\rho + 1)} < \eta < \frac{\beta \big[ 1 + \sqrt{1 - 24 L\rho(3L\rho + 1)} \big]}{2L(3L\rho + 1)} \leq \frac{\beta}{L},
\end{array}
\end{equation}
then we have
\begin{equation}\label{eq:FBFS4NI_convergence_02}
\min_{1\leq l \leq k+1}\norms{Fx^l + \xi^l}^2 \leq \frac{1}{k+1}\sum_{l=1}^{k+1}\norms{Fx^{l} + \xi^{l}}^2 \leq \frac{C_{\rho}\norms{x^0 - x^{\star}}^2}{k+1}, \quad \xi^l \in Tx^l,
\end{equation} 
where $C_{\rho} :=  \frac{3 + 2L^2}{\eta(\beta\eta - 6\beta^2\rho - (3L\rho + 1)L\eta^2)} > 0$.
As a consequence, $\sets{\norms{x^k - x^{\star}}}$ is nonincreasing and
\begin{equation}\label{eq:FBFS4NI_convergence_01}
\lim_{k\to\infty}\norms{x^k - y^k} = \lim_{k\to\infty}\norms{Fy^k + \zeta^k} =  \lim_{k\to\infty}\norms{Fx^k + \xi^k} = 0.
\end{equation}
Moreover, $\sets{x^k}$ converges to $x^{\star} \in \zer{\Phi}$, a solution of \eqref{eq:NI}.

\item[$\mathrm{(b)}$] \mytb{$($Past-FBFS/OG method$)$} 
Let us choose $u^k := Fy^{k-1}$ with $y^{-1} := x^0$ and assume that $L\rho \leq \frac{2\sqrt{3} - 3}{24} \approx 0.01934$.
Then, for any $\beta \in (0, 1]$, if  $\eta$ is chosen such that
\begin{equation}\label{eq:FBFS4NI_stepsize}
\arraycolsep=0.2em
\begin{array}{lcl}
0 \leq \frac{\beta \big[ 1 - \sqrt{1 - 48 L\rho(4L\rho + 1)} \big]}{6L(4L\rho + 1)} < \eta < \frac{\beta \big[ 1 + \sqrt{1 - 48 L\rho(4L\rho + 1)} \big]}{6L(4L\rho + 1)} \leq \frac{\beta}{3L},
\end{array}
\end{equation}
then we have
\begin{equation}\label{eq:FBFS4NI_convergence_03}
\min_{1\leq l \leq k+1}\norms{Fx^l + \xi^l}^2 \leq \frac{1}{k+1}\sum_{l=1}^{k+1} \left[ \norms{Fx^l + \xi^l}^2 + \frac{L\eta^2 + 8\beta^2\rho}{2\eta} \norms{x^l - y^{l-1}}^2 \right] \leq \frac{\hat{C}_{\rho}\norms{x^0 - x^{\star}}^2}{k + 1},
\end{equation} 
where $\hat{C}_{\rho} := \frac{3 + 2L^2}{\eta(\beta\eta - 4\beta^2\rho - 3(4L\rho + 1)L\eta^2)} > 0$.
Moreover, $\sets{x^k}$ converges to  $x^{\star} \in \zer{\Phi}$ and \eqref{eq:EG4NI_convergence_01} still holds.
\end{itemize}
In both cases $\mathrm{(a)}$ and $\mathrm{(b)}$, if $J_{\eta T}$ is single-valued and nonexpansive, then 
\begin{equation}\label{eq:FBFS4NI_convergence_02_03}
\min_{1\leq l \leq k}\norms{G_{\eta\Phi}x^l} = \BigO{\frac{1}{\sqrt{k}}} \quad \text{and} \quad \min_{1\leq l \leq k}\norms{G_{\eta\Phi}y^l} = \BigO{\frac{1}{ \sqrt{k}}}, 
\end{equation}
where $G_{\eta\Phi}x := \frac{1}{\eta}(x - J_{\eta T}(x - \eta Fx))$ is given by \eqref{eq:FB_residual}.
Moreover, $\lim_{k\to\infty}\norms{G_{\eta\Phi}y^k} = \lim_{k\to\infty}\norms{G_{\eta\Phi}x^k} = 0$.
\end{theorem}

%%%% Proof of Theorem 4.1
\begin{proof}
(a)~\mytb{(FBFS method)} 
For $u^k := Fx^k$, by the $L$-Lipschitz continuity of $F$,  we have $\norms{Fy^k - u^k} = \norms{Fy^k - Fx^k} \leq L\norms{y^k - x^k}$.
Using this inequality into \eqref{eq:FBFS4NI_key_est1}, we have
\begin{equation}\label{eq:FBFS4NI_th41_proof1}
\arraycolsep=0.2em
\begin{array}{lcl}
\norms{x^{k+1} - x^{\star}}^2 & \leq & \norms{x^k - x^{\star}}^2 -  (1 - \beta)\norms{x^{k+1} - x^k}^2 -  (\beta - \gamma)\norms{x^{k+1} - y^k}^2 \vspace{1ex}\\
&& - {~} \left(\beta - \frac{L^2\eta^2}{\gamma}\right)\norms{x^k - y^k}^2  - 2\eta\iprods{Fy^k + \zeta^k, y^k - x^{\star}}.
\end{array}
\end{equation}
Now, by Young's inequality, the $L$-Lipschitz continuity of $F$, and $x^k - y^k = \frac{\eta}{\beta}(u^k + \zeta^k) = \frac{\eta}{\beta}(Fx^k + \zeta^k)$ from the first line of \eqref{eq:FBFS4NI_reform}, we have
\begin{equation}\label{eq:FBFS4NI_th41_proof2a}  
\arraycolsep=0.2em
\begin{array}{lcl}
\norms{Fy^k + \zeta^k}^2 & \leq & \frac{3}{2} \norms{Fy^k - Fx^k}^2 + 2\norms{Fx^k + \zeta^k}^2 \leq \frac{3(L^2\eta^2 + 2\beta^2)}{2\eta^2}\norms{x^k - y^k}^2.  
\end{array}
\end{equation}
Similarly, using the second line of \eqref{eq:FBFS4NI_reform},  Young's inequality, and the $L$-Lipschitz continuity of $F$, we have
\begin{equation}\label{eq:FBFS4NI_th41_proof2}
\arraycolsep=0.2em
\begin{array}{lcl}
\eta^2\norms{Fx^{k+1} + \xi^{k+1}}^2 & = & \norms{x^{k+1} - x^k + \eta(Fx^{k+1} - Fy^k)}^2 \vspace{1ex}\\
&\leq & (1 + \frac{2L^2}{3})\norms{x^{k+1} - x^k}^2 + (1 + \frac{3}{2L^2}) \norms{Fx^{k+1} - Fy^k}^2 \vspace{1ex}\\
&\leq &\frac{3}{2}(1 + \frac{2L^2}{3})\norms{x^{k+1} - y^k}^2 + (1 + \frac{3}{2L^2})L^2\norms{x^{k+1} - y^k}^2 +3(1 + \frac{2L^2}{3})\norms{x^k - y^k}^2 \vspace{1ex}\\ 
&= & (3 + 2L^2) [ \norms{x^{k+1} - y^k}^2 + \norms{x^k - y^k}^2].
\end{array}
\end{equation}
Next, since $\iprods{u, x - x^{\star}} \geq - \rho \norms{u}^2$ for any $(x, u) \in \gra{\Phi}$,  we have  $\iprods{Fy^k + \zeta^k, y^k - x^{\star}} \geq - \rho \norms{Fy^k + \zeta^k}^2$ since $(y^k, \zeta^k) \in \gra{T}$.
Using this relation and \eqref{eq:FBFS4NI_th41_proof2a} into \eqref{eq:FBFS4NI_th41_proof1} with  $\gamma := L\eta$, we get
\begin{equation}\label{eq:FBFS4NI_th41_proof3}
\arraycolsep=0.2em
\begin{array}{lcl}
\norms{x^{k+1} - x^{\star}}^2 & \leq & \norms{x^k - x^{\star}}^2 -  (\beta -   L\eta) \norms{x^k - y^k}^2 - (\beta -   L\eta) \norms{x^{k+1} - y^k}^2 + 2\eta\rho\norms{Fy^k + \zeta^k}^2 \vspace{1ex}\\
&\leq & \norms{x^k - x^{\star}}^2 -  (\beta - L\eta) \norms{x^{k+1} - y^k}^2 -   \tfrac{\beta\eta - 6\beta^2\rho - (3L\rho + 1)L\eta^2}{\eta}  \norms{x^k - y^k}^2.
\end{array}
\end{equation}
Let us impose $\beta\eta - 6\beta^2\rho - (3L\rho + 1)L\eta^2 > 0$, which holds if
\begin{equation*} 
\arraycolsep=0.2em
\begin{array}{lcl}
0 \leq \frac{\beta \big[ 1 - \sqrt{1 - 24 L\rho(3L\rho + 1)} \big]}{2L(3L\rho + 1)} < \eta < \frac{\beta \big[ 1 + \sqrt{1 - 24 L\rho(3L\rho + 1)} \big]}{2L(3L\rho + 1)} \leq \frac{\beta}{L},
\end{array}
\end{equation*}
provided that  $L\rho \leq \frac{3\sqrt{2} -2}{12} \approx 0.1869$. 
This choice of $\eta$ is exactly \eqref{eq:FBFS4NI_stepsize}.
Let $\psi := \tfrac{\beta\eta - 6\beta^2\rho - (3L\rho + 1)L\eta^2}{\eta} = \beta - L\eta - \frac{3\rho}{\eta}(2\beta^2 + L^2\eta^2) > 0$.
Then, we have $\beta - L\eta - \psi \geq 0$.
In this case, \eqref{eq:FBFS4NI_th41_proof3} becomes 
\begin{equation*} 
\arraycolsep=0.2em
\begin{array}{lcl}
\norms{x^{k+1} - x^{\star}}^2 & \leq & \norms{x^k - x^{\star}}^2 - \psi[\norms{x^k - y^k}^2 + \norms{x^{k+1} - y^k}^2] - (\beta - L\eta - \psi)\norms{x^{k+1} - y^k}^2 \vspace{1ex}\\
& \leq &  \norms{x^k - x^{\star}}^2 - \psi[\norms{x^k - y^k}^2 + \norms{x^{k+1} - y^k}^2].
\end{array}
\end{equation*}
Finally, using this estimate, \eqref{eq:FBFS4NI_th41_proof2}, we can easily prove \eqref{eq:FBFS4NI_convergence_02} in  Theorem~\ref{th:FBFS4NI_convergence}.
The remaining statements are direct consequences of  \eqref{eq:FBFS4NI_convergence_02},  \eqref{eq:FBFS4NI_th41_proof2a}, and \eqref{eq:FBFS4NI_th41_proof2} using standard arguments.

(b)~\mytb{(Past-FBFS/OG method)} 
For $u^k := Fy^{k-1}$, using Young's inequality and the $L$-Lipschitz continuity of $F$,  we can derive
\begin{equation*}
\arraycolsep=0.2em
\begin{array}{lcl}
\norms{Fy^k - u^k}^2 = \norms{Fy^k - Fy^{k-1}}^2 \leq 3L^2\norms{x^k - y^k}^2 + \frac{3L^2}{2}\norms{x^k - y^{k-1}}^2.
\end{array}
\end{equation*} 
Next, since $\iprods{u, x - x^{\star}} \geq -\rho\norms{u}^2$ for all $(x, u) \in \gra{\Phi}$, we have $\iprods{Fy^k + \zeta^k, y^k - x^{\star}} \geq -\rho\norms{Fy^k + \zeta^k}^2$.
Using this inequality and the first line of \eqref{eq:FBFS4NI_reform} as $\eta(Fy^k + \zeta^k) = \beta(x^k - y^k) + \eta(Fy^k - Fy^{k-1})$.
\begin{equation*} 
\arraycolsep=0.2em
\begin{array}{lcl}
\iprods{Fy^k + \zeta^k, y^k - x^{\star}} &\geq & -\rho\norms{Fy^k + \zeta^k}^2 = -\frac{\rho}{\eta^2}\norms{\beta(x^k - y^k) + \eta(Fy^k - Fy^{k-1})}^2 \vspace{1ex}\\
& \geq & -\frac{2\rho\beta^2}{\eta^2}\norms{x^k -y^k}^2 - 4\rho \norms{Fy^k - Fx^k}^2 - 4\rho\norms{Fx^k - Fy^{k-1}}^2 \vspace{1ex}\\
& \geq & -  \frac{2\rho(\beta^2 + 2L^2\eta^2)}{\eta^2}  \norms{x^k -y^k}^2 - 4\rho L^2\norms{x^k - y^{k-1}}^2.
\end{array}
\end{equation*}
Substituting the last two expressions into \eqref{eq:FBFS4NI_key_est1} and choosing $\gamma := L\eta$, we obtain
\begin{equation}\label{eq:FBFS4NI_th41p2_proof3}
\arraycolsep=0.2em
\begin{array}{lcl}
\Vc_{k+1} & \leq & \Vc_k - (1 - \beta)\norms{x^{k+1} - x^k}^2 -   \left(\beta - \frac{5L\eta}{2} - 8\rho L^2\eta \right)\norms{x^{k+1} - y^k}^2 \vspace{1ex}\\
&& - {~} \left[ \beta -  3L\eta - \frac{4\rho(\beta^2 + 2L^2\eta^2)}{\eta} \right] \norms{x^k - y^k}^2,
\end{array}
\end{equation}
where $\Vc_k :=  \norms{x^k - x^{\star}}^2 + L\eta\left( \frac{3}{2} + 8\rho L \right)\norms{x^k - y^{k-1}}^2$.

Let us define $\hat{\psi} :=  \beta -  3L\eta - \frac{4\rho(\beta^2 + 2L^2\eta^2)}{\eta}$ and $\hat{\varphi} := \beta - \frac{5L\eta}{2} - 8\rho L^2\eta$ and impose the condition that $\hat{\psi} > 0$.
Then, it is clear that $\hat{\varphi} - \hat{\psi} = \frac{L\eta}{2} + \frac{4\beta^2\rho}{\eta} \geq 0$.
Moreover, \eqref{eq:FBFS4NI_th41p2_proof3} becomes $\Vc_{k+1} \leq \Vc_k - \hat{\psi}\big[ \norms{x^{k+1} - y^k}^2 + \norms{x^k - y^k}^2 \big] - (\hat{\varphi} - \hat{\psi})\norms{x^{k+1} - y^k}^2$.
Combining this estimate and \eqref{eq:FBFS4NI_th41_proof2}, we can easily prove \eqref{eq:FBFS4NI_convergence_03} in  Theorem~\ref{th:FBFS4NI_convergence}.
Note that the condition $\hat{\psi} > 0$ holds if $\eta$ is chosen as in \eqref{eq:FBFS4NI_stepsize} provided that $L\rho \leq \frac{2\sqrt{3} - 3}{24} \approx 0.01934$.
Using  \eqref{eq:FBFS4NI_convergence_03}, we can easily prove that $\sets{x^k}$ converges to  $x^{\star}$.
Moreover,  \eqref{eq:FBFS4NI_convergence_03} also implies $\lim_{k\to\infty}\norms{x^k - y^k} = \lim_{k\to\infty}\norms{x^{k+1} - y^k} = \lim_{k\to\infty}\norms{Fx^k + \xi^k} = 0$.
By the second line of \eqref{eq:FBFS4NI_reform}, we have 
\begin{equation*}
\arraycolsep=0.2em
\begin{array}{lcl}
\eta\norms{Fy^k + \zeta^k} \leq L\eta\norms{x^k - y^k} + L\eta\norms{x^k - y^{k-1}} + \norms{\eta(Fy^{k-1} + \zeta^k} = (L\eta + 1)\norms{x^k - y^k} + L\eta\norms{x^k - y^{k-1}}.
\end{array}
\end{equation*}
Combing this inequality and $\lim_{k\to\infty}\norms{x^k - y^k} = \lim_{k\to\infty}\norms{x^{k+1} - y^k} = 0$, we obtain $\lim_{k\to\infty}\eta\norms{Fy^k + \zeta^k} = 0$.

Finally, if $J_{\eta T}$ is single-valued and nonexpansive, then by using $\norms{G_{\Phi}x^k} \leq \norms{Fx^k + \xi^k}$ from \eqref{eq:FBR_bound2} and either \eqref{eq:FBFS4NI_convergence_02} or \eqref{eq:FBFS4NI_convergence_03}, we obtain \eqref{eq:FBFS4NI_convergence_02_03}.
Using again \eqref{eq:FBR_bound2}  and the limits \eqref{eq:FBFS4NI_convergence_01}, we obtain $\lim_{k\to\infty}\norms{G_{\Phi}x^k} \leq \lim_{k\to\infty}\norms{Fx^k + \xi^k} = 0$ and $\lim_{k\to\infty}\norms{G_{\Phi}y^k} \leq \lim_{k\to\infty}\norms{Fy^k + \zeta^k} = 0$.
\end{proof}
%%% End of proof.

\begin{remark}\label{re:FBFS4NI_remark1}
The $\rho$-star co-hypomonotone condition in Theorem~\ref{th:FBFS4NI_convergence}(b) trivially holds if $\Phi := F + T$ is $\rho$-co-hypomonotone.
Hence, the star co-hypomonotone condition is generally weaker than the $\rho$-co-hypomonotonicity.
Note that we have not tried to optimize the parameters in Theorem~\ref{th:FBFS4NI_convergence}, and generally in the whole paper.
By careful tightening bounds and selecting parameters in our analysis (e.g., where Young's inequality is used), we can improve the range of parameters.
Note that the convergence analysis for the monotone case is very classical, which can be found, e.g., in \cite{Facchinei2003,Korpelevic1976}.
However, the convergence rates for $\beta < 1$, and for the star co-hypomonotone case are recent results.
The best-iterate convergence rate of  \eqref{eq:FBFS4NI} was proven in \cite{luo2022last} for the star co-hypomonotone case, but using a potential function.
Here, we provide a different proof using classical results in \cite{Facchinei2003} combining with the star co-hypomonotonicity of $\Phi$.
\end{remark}

\beforesec
\section{Two Other Variants of The Extragradient Method}\label{sec:other_methods}
\aftersec
In this section, we review two additional methods: the reflected forward-backward splitting  (RFBS) algorithm \cite{cevher2021reflected,malitsky2015projected} and the golden ratio (GR) scheme \cite{malitsky2019golden}.
The last-iterate analysis for RFBS scheme was recently given in  \cite{cai2022baccelerated}, but only for \eqref{eq:VIP}.
Here, we provide a new analysis for both the best-iterate and the last-iterate rates for RFBS to solve \eqref{eq:NI}, which is more general than \eqref{eq:VIP}.
The best-iterate convergence analysis for GR is modified the proof from  \cite{malitsky2019golden} to expand the range of parameters.
Nevertheless, the last-iterate convergence rate analysis of GR is still open.

\beforesubsec
\subsection{Reflected forward-backward splitting method}\label{subsec:RFBS4NI}
\aftersubsec
The reflected forward-backward splitting method was proposed by Malitsky in  \cite{malitsky2015projected} to solve \eqref{eq:VIP} and it is called the \mytb{projected reflected gradient} method.
It was generalized to solve monotone \eqref{eq:NI} in \cite{cevher2021reflected}, which is called the \textbf{reflected forward-backward splitting (RFBS)} scheme.
The last iterate convergence rate of the \mytb{projected reflected gradient} method for \eqref{eq:VIP} was recently proven in \cite{cai2022baccelerated}.
In this subsection, we survey this method for solving \eqref{eq:NI}.
We provide a new best-iterate convergence rate analysis compared to  \cite{cevher2021reflected}.
We also present an elementary proof for the last-iterate convergence rate of RFBS for solving monotone \eqref{eq:NI}.

The reflected forward-backward splitting (RFBS) method to approximate a solution of \eqref{eq:NI} is described as follows.
Starting from $x^0 \in \dom{\Phi}$, we choose $x^{-1} := x^0$ and at each iteration $k\geq 0$, we update
\begin{equation}\label{eq:RFBS4NI}
\arraycolsep=0.2em
\left\{\begin{array}{lcl}
y^k &:= & 2x^k - x^{k-1}, \vspace{1ex}\\
x^{k+1} &:= & J_{\eta T}(x^k - \eta Fy^k),
\end{array}\right.
\tag{RFBS2}
\end{equation}
where $\eta > 0$ is a given step-size, determined later.
Clearly, if we eliminate $y^k$, then \eqref{eq:RFBS4NI} can be written as 
\begin{equation*}
x^{k+1} := J_{\eta T}(x^k - \eta F(2x^k - x^{k-1})).
\end{equation*}
From the second line of \eqref{eq:RFBS4NI}, we have $\xi^{k+1} := \frac{1}{\eta}( x^k - \eta Fy^k - x^{k+1})  \in  Tx^{k+1}$.
As before, if we denote
\begin{equation}\label{eq:wk_terms}
w^k := Fx^k + \xi^k \quad \text{and} \quad \hat{w}^k := Fy^{k-1} + \xi^k,
\end{equation}
then we can rewrite \eqref{eq:RFBS4NI} equivalently to 
\begin{equation}\label{eq:RFBS4NI_reform}
\arraycolsep=0.2em
\left\{\begin{array}{lclcl}
y^k &:= & x^k + x^k - x^{k-1} &= & x^k - \eta\hat{w}^k, \vspace{1ex}\\
x^{k+1} &:= & x^k - \eta\hat{w}^{k+1}.
\end{array}\right.
\end{equation}
This expression leads to $x^{k+1} - y^k = -\eta(\hat{w}^{k+1} - \hat{w}^k)$. 
Next, we prove the following lemmas for our analysis.

%%% Lemma 3.1.
\begin{lemma}\label{le:RFBS4NI_key_est1}
Assume that $T$ in \eqref{eq:NI} is maximally monotone and $F$ in \eqref{eq:NI} is  $L$-Lipschitz continuous and satisfies $\iprods{Fx - Fx^{\star}, x - x^{\star}} \geq 0$ for all $x\in\dom{\Phi}$ and some $x^{\star}\in\zer{\Phi}$.
Let $\set{(x^k, y^k)}$ be generated by \eqref{eq:RFBS4NI} using $\eta > 0$ and $\Vc_k$ be defined as
\begin{equation}\label{eq:RFBS4NI_potential_func}
\arraycolsep=0.2em
\begin{array}{lcl}
\Vc_k & := & \norms{x^k - x^{\star}}^2 + 2\norms{x^k - x^{k-1}}^2 +  \big( 1 - \sqrt{2}L \eta \big) \norms{x^k - y^{k-1}}^2 + 2\eta \iprods{Fy^{k-1} - Fx^{\star}, x^k - x^{k-1}}.
\end{array}
\end{equation}
Then, we have
\begin{equation}\label{eq:RFBS4NI_est1} 
\arraycolsep=0.2em
\begin{array}{lcl}
\Vc_k  & \geq & \Vc_{k+1} + \big[ 1 - (1+\sqrt{2})L\eta \big] \left[ \norms{y^k - x^k}^2 + \norms{x^k - y^{k-1}}^2 \right], \vspace{1.5ex}\\
\Vc_k  & \geq &   (1 - L\eta) \norms{x^k - x^{\star}}^2 + (1 - (1+\sqrt{2})L\eta) \norms{x^k - y^{k-1}}^2 + 2(1 -  L\eta) \norms{x^k - x^{k-1}}^2.
\end{array}
\end{equation}
\end{lemma}

%%% Proof of Lemma 3.1.
\begin{proof}
First, since \eqref{eq:RFBS4NI} is equivalent to \eqref{eq:RFBS4NI_reform}, we have $x^{k+1} - x^k = -\eta\hat{w}^{k+1}$ from the second line of \eqref{eq:RFBS4NI_reform}.
Using this expression, for any $x^{\star}\in\zer{\Phi}$, we have
\begin{equation}\label{eq:RFBS4NI_proof1}
\arraycolsep=0.2em
\begin{array}{lcl}
\norms{x^{k+1} - x^{\star}}^2 &= & \norms{x^k - x^{\star}}^2 - 2\eta\iprods{\hat{w}^{k+1}, x^{k+1} - x^{\star}} - \norms{x^{k+1} - x^k}^2.
\end{array} 
\end{equation}
Next, since  $Fx^{\star} + \xi^{\star} = 0$ from the fact that $x^{\star} \in \zer{\Phi}$ and $T$ is monotone, we have $\iprods{\xi^{k+1}, x^{k+1} - x^{\star}} \geq \iprods{\xi^{\star}, x^{k+1} - x^{\star}} = -\iprods{Fx^{\star}, x^{k+1} - x^{\star}}$, where $\xi^{\star}\in Tx^{\star}$.
Using this relation, we can prove that
\begin{equation}\label{eq:RFBS4NI_proof2a} 
\arraycolsep=0.2em
\begin{array}{lcl}
\iprods{\hat{w}^{k+1}, x^{k+1} - x^{\star}} &= &  \iprods{Fy^k, x^{k+1} - x^{\star}} + \iprods{\xi^{k+1}, x^{k+1} - x^{\star}} \vspace{1ex}\\
&\geq & \iprods{Fy^k - Fx^{\star}, x^{k+1} - y^k} + \iprods{Fy^k - Fx^{\star}, y^k - x^{\star}}.
\end{array} 
\end{equation}
Utilizing $y^k - x^k = x^k - x^{k-1}$ from the first line of \eqref{eq:RFBS4NI}, we can further expand
\begin{equation}\label{eq:RFBS4NI_proof2}
\hspace{-0ex}
\arraycolsep=0.2em
\begin{array}{lcl}
 \iprods{Fy^k - Fx^{\star}, y^k - x^{k+1}} &= & \iprods{Fy^{k-1} - Fx^{\star}, y^k - x^k} - \iprods{Fy^k - Fx^{\star}, x^{k+1} - x^k}  + \iprods{Fy^k - Fy^{k-1}, y^k - x^k}   \vspace{1ex}\\
&= & \iprods{Fy^{k-1} - Fx^{\star}, x^k - x^{k-1}} - \iprods{Fy^k - Fx^{\star}, x^{k+1} - x^k}  \vspace{1ex}\\
&& + {~} \iprods{Fy^k - Fy^{k-1}, y^k - x^k}.
\end{array} 
\hspace{-6ex}
\end{equation}
Now, from the second line of \eqref{eq:RFBS4NI_reform}, we have $\eta \xi^{k+1} = x^k - x^{k+1} - \eta Fy^k$, leading to $\eta (\xi^{k+1} - \xi^k) = 2x^k - x^{k-1} - x^{k+1} - \eta(Fy^k - Fy^{k-1}) = y^k - x^{k+1} - \eta(Fy^k - Fy^{k-1})$.
By the monotonicity of $T$, we have $ \iprods{y^k - x^{k+1} - \eta(Fy^k - Fy^{k-1}), x^{k+1} - x^k} = \eta\iprods{\xi^{k+1} - \xi^k, x^{k+1} - x^k} \geq 0$, leading to
\begin{equation*} 
\arraycolsep=0.2em
\begin{array}{lcl}
2\eta\iprods{Fy^k - Fy^{k-1}, x^{k+1} - x^k} &\leq & 2\iprods{y^k - x^{k+1}, x^{k+1} - x^k} \vspace{1ex}\\
&= & \norms{y^k - x^k}^2 - \norms{x^{k+1} - x^k}^2 - \norms{x^{k+1} - y^k}^2.
\end{array} 
\end{equation*}
By the Cauchy-Schwarz inequality, the Lipschitz continuity of $F$, and Young's inequality, we can show that
\begin{equation*} 
\arraycolsep=0.2em
\begin{array}{lcl}
2\eta\iprods{Fy^k - Fy^{k-1},y^k - x^{k+1}} 
%&\leq & 2\eta\norms{Fy^k - Fy^{k-1}}\norms{x^{k+1} - y^k} \vspace{1ex}\\
&\leq & 2\eta L\norms{y^k - y^{k-1}}\norms{x^{k+1} - y^k} \vspace{1ex}\\
%& \leq & \frac{\eta L}{\sqrt{2}}\norms{y^k - y^{k-1}}^2 + \sqrt{2}\eta L \norms{x^{k+1} - y^k}^2 \vspace{1ex}\\
& \leq & \eta L(\sqrt{2} + 1)\norms{y^k - x^k}^2 + \eta L\norms{x^k - y^{k-1}}^2 + \sqrt{2}\eta L \norms{x^{k+1} - y^k}^2.
\end{array} 
\end{equation*}
Summing up the last two inequalities, we get
\begin{equation*} 
\arraycolsep=0.2em
\begin{array}{lcl}
2\eta \iprods{Fy^k - Fy^{k-1}, y^k - x^k} &\leq & \left[ 1 +  \eta L(\sqrt{2} + 1)\right] \norms{y^k - x^k}^2 + \eta L\norms{x^k - y^{k-1}}^2 \vspace{1ex}\\
&& - {~} (1 - \sqrt{2} L \eta ) \norms{x^{k+1} - y^k}^2 - \norms{x^{k+1} - x^k}^2.
\end{array} 
\end{equation*}
Combining this inequality, \eqref{eq:RFBS4NI_proof2a}, and \eqref{eq:RFBS4NI_proof2} with \eqref{eq:RFBS4NI_proof1}, we can eventually prove that
\begin{equation*} 
\arraycolsep=0.2em
\begin{array}{lcl}
\norms{x^{k+1} - x^{\star}}^2 &= & \norms{x^k - x^{\star}}^2 + 2\eta\iprods{\hat{w}^{k+1}, x^{\star} - x^{k+1}} - \norms{x^{k+1} - x^k}^2 \vspace{1ex}\\
%&\leq & \norms{x^k - x^{\star}}^2 +  2\eta \iprods{Fy^{k-1} - Fx^{\star}, x^k - x^{k-1}}  - 2\eta \iprods{Fy^k - Fx^{\star}, x^{k+1} - x^k} \vspace{1ex}\\
%&& + {~} 2\eta\iprods{Fy^k - Fy^{k-1}, y^k - x^k}  - \norms{x^{k+1} - x^k}^2 - 2\eta \iprods{Fy^k - Fx^{\star}, y^k - x^{\star}} \vspace{1ex}\\
&\leq & \norms{x^k - x^{\star}}^2 +  2\eta \iprods{Fy^{k-1} - Fx^{\star}, x^k - x^{k-1}}  - 2\eta \iprods{Fy^k - Fx^{\star}, x^{k+1} - x^k} \vspace{1ex}\\
&& + {~} \left[ 1 +   L\eta(\sqrt{2} + 1)\right] \norms{y^k - x^k}^2 - 2\norms{x^{k+1} - x^k}^2 \vspace{1ex}\\
&& +  {~} L\eta \norms{x^k - y^{k-1}}^2 -  (1 - \sqrt{2}L\eta ) \norms{x^{k+1} - y^k}^2 - 2\eta \iprods{Fy^k - Fx^{\star}, y^k - x^{\star}}.
\end{array} 
\end{equation*}
Using the definition \eqref{eq:RFBS4NI_potential_func} of $\Vc_k$, $\iprods{Fy^k - Fx^{\star}, y^k - x^{\star}} \geq 0$, and $y^k - x^k = x^k - x^{k-1}$, this estimate becomes
\begin{equation*} 
\arraycolsep=0.2em
\begin{array}{lcl}
\Vc_{k+1} &:= & \norms{x^{k+1} - x^{\star}}^2 + 2\norms{x^{k+1} - x^k}^2 +  (1 - \sqrt{2}L\eta ) \norms{x^{k+1} - y^k}^2 + 2\eta \iprods{Fy^k - Fx^{\star}, x^{k+1} - x^k} \vspace{1ex}\\
&\leq & \norms{x^k - x^{\star}}^2 +  \left[ 1 +  (\sqrt{2} + 1)L\eta \right] \norms{y^k - x^k}^2 +  L\eta \norms{x^k - y^{k-1}}^2 +  2\eta \iprods{Fy^{k-1} - Fx^{\star}, x^k - x^{k-1}} \vspace{1ex}\\
&= & \Vc_k - \left[1 - (\sqrt{2} + 1)L\eta \right] \left(  \norms{y^k - x^k}^2  +  \norms{x^k - y^{k-1}}^2 \right),
\end{array} 
\end{equation*}
which proves the first inequality of \eqref{eq:RFBS4NI_est1}.

\noindent Next, using Young's inequality twice, we can show that
\begin{equation*}
\arraycolsep=0.2em
\begin{array}{lcl}
2\eta \iprods{Fy^{k-1} - Fx^{\star}, x^k - x^{k-1}} %&\geq & -2\eta L \norms{y^{k-1} - x^{\star}}\norms{x^k - x^{k-1}} \vspace{1ex}\\
& \geq & - \frac{L\eta}{2} \norms{y^{k-1} - x^{\star}}^2 - 2 L\eta \norms{x^k - x^{k-1}}^2 \vspace{1ex}\\
& \geq & - L\eta \norms{x^k - x^{\star}}^2 - L\eta \norms{x^k - y^{k-1}}^2 - 2 L\eta \norms{x^k - x^{k-1}}^2.
\end{array} 
\end{equation*}
Substituting this estimate into \eqref{eq:RFBS4NI_potential_func}, we get
\begin{equation*}
\arraycolsep=0.2em
\begin{array}{lcl}
\Vc_k & := &  \norms{x^k - x^{\star}}^2 + 2\norms{x^k - x^{k-1}}^2 +  \left(1 - \sqrt{2}L \eta\right) \norms{x^k - y^{k-1}}^2 + 2\eta \iprods{Fy^{k-1} - Fx^{\star}, x^k - x^{k-1}} \vspace{1ex}\\
&\geq & (1 - L\eta) \norms{x^k - x^{\star}}^2 + (1 - (1+\sqrt{2})L\eta) \norms{x^k - y^{k-1}}^2 + 2(1 -  L\eta) \norms{x^k - x^{k-1}}^2,
\end{array} 
\end{equation*}
which proves the second line of \eqref{eq:RFBS4NI_est1}.
\end{proof}
%%%% End of proof.

\begin{lemma}\label{le:RFBS4NI_key_est2}
Assume that $F+T$ in \eqref{eq:NI} is maximally monotone and $F$ in \eqref{eq:NI} is $L$-Lipschitz continuous.
Let $\set{(x^k, y^k)}$ be generated by \eqref{eq:RFBS4NI} using $\eta > 0$.
Then, we have
\begin{equation}\label{eq:RFBS4NI_est2}
\hspace{-0ex}
\arraycolsep=0.2em
\begin{array}{lcl}
& \norms{ Fx^k + \xi^k}^2 \leq  \frac{5L^2\eta^2 + 3}{3\eta^2}  \norms{x^k - y^k}^2 + \frac{5L^2\eta^2 + 3}{5\eta^2} \norms{x^k - y^{k-1}}^2.
\end{array}
\end{equation}
Moreover, if $\sqrt{2}L\eta < 1$, then  with $\kappa := \frac{2L^2\eta^2}{1 - 2L^2\eta^2} > 0$, we also have
\begin{equation}\label{eq:RFBS4NI_est3}
\hspace{-0ex}
\arraycolsep=0.2em
\begin{array}{lcl}
\norms{Fx^{k+1} + \xi^{k+1}}^2 + \kappa \norms{Fx^{k+1} - Fy^k}^2 &\leq &  \norms{Fx^k + \xi^k}^2 + \kappa \norms{Fx^k - Fy^{k-1}}^2 \vspace{1ex}\\
&& - {~} \left( \frac{1 - 4L^2\eta^2 }{1 - 2L^2\eta^2} \right) \norms{Fy^k - Fx^k + \xi^{k+1} - \xi^k}^2.
\end{array} 
\hspace{-1ex}
\end{equation}
\end{lemma}

\begin{proof}
First, by Young's inequality, $x^k - y^k = -\eta(Fy^{k-1} + \xi^k)$ from \eqref{eq:RFBS4NI_reform}, and the $L$-Lipschitz continuity of $F$, we have
\begin{equation*}
\arraycolsep=0.2em
\begin{array}{lcl}
\norms{ Fx^k + \xi^k }^2 & \leq &   \big(1 + \frac{5L^2\eta^2}{3}\big) \norms{ Fy^{k-1} + \xi^k }^2  + \left(1 + \frac{3}{5L^2\eta^2}\right)\norms{ Fx^k - Fy^{k-1} }^2  \vspace{1ex}\\
& \leq & \frac{5L^2\eta^2 + 3}{3\eta^2}\norms{x^k - y^k}^2 + \frac{5L^2\eta^2 + 3}{5\eta^2} \norms{x^k - y^{k-1}}^2.
\end{array} 
\end{equation*}
This estimate is exactly the first line of \eqref{eq:RFBS4NI_est2}.

Next, by the monotonicity of $F+T$, we have $\iprods{w^{k+1} - w^k, x^{k+1} - x^k}  \geq 0$.
Substituting $x^{k+1} - x^k = -\eta\hat{w}^{k+1}$ into this inequality, we get
\begin{equation*}
\arraycolsep=0.2em
\begin{array}{lcl}
0 &\leq &  2\iprods{w^k, \hat{w}^{k+1}} -2\iprods{w^{k+1}, \hat{w}^{k+1}}  =  \norms{w^k}^2 - \norms{w^{k+1}}  + \norms{w^{k+1} - \hat{w}^{k+1}}^2  - \norms{\hat{w}^{k+1} - w^k}^2.
\end{array} 
\end{equation*}
This inequality implies that
\begin{equation*}
\arraycolsep=0.2em
\begin{array}{lcl}
\norms{w^{k+1}}^2 &\leq & \norms{w^k}^2 + \norms{w^{k+1} - \hat{w}^{k+1}}^2  - \norms{\hat{w}^{k+1} - w^k}^2.
\end{array} 
\end{equation*}
On the other hand, by the Lipschitz continuity of $F$ and $x^{k+1} - y^k = -\eta(\hat{w}^{k+1} - \hat{w}^k)$, we have $\norms{w^{k+1} - \hat{w}^{k+1}}^2 = \norms{Fx^{k+1} - Fy^k}^2 \leq L^2\norms{x^{k+1} - y^k}^2 = L^2\eta^2\norms{\hat{w}^{k+1} - \hat{w}^k}^2 \leq 2L^2\eta^2\norms{\hat{w}^{k+1} - w^k}^2 + 2L^2\eta^2\norms{w^k - \hat{w}^k}^2$.  
Multiplying this inequality by $\omega \geq 1$, and adding the result to the last inequality, we get
\begin{equation*}
\arraycolsep=0.2em
\begin{array}{lcl}
\norms{w^{k+1}}^2 + (\omega - 1)\norms{w^{k+1} - \hat{w}^{k+1}}^2 &\leq & \norms{w^k}^2  + 2\omega L^2\eta^2\norms{w^k - \hat{w}^k}^2 - (1 -  2\omega L^2\eta^2) \norms{\hat{w}^{k+1} - w^k}^2.
\end{array} 
\end{equation*}
Finally, let us choose $\omega \geq 1$ such that $\omega - 1 = 2\omega L^2\eta^2$. 
If $2L^2\eta^2 < 1$, then $\omega := \frac{1}{ 1 - 2L^2\eta^2}$ satisfies $\omega - 1 = 2\omega L^2\eta^2$. 
Consequently, the last estimate leads to \eqref{eq:RFBS4NI_est2}.
\end{proof}
%%% End of Lemma 2.

Now, we are ready to establish the convergence of \eqref{eq:RFBS4NI} in the following theorem.

%%% Theorem 3.1.
\begin{theorem}\label{th:RFBS_convergence}
Assume that $\zer{\Phi} \neq\emptyset$, $T$ in \eqref{eq:NI} is maximally monotone, and $F$ in \eqref{eq:NI} is  $L$-Lipschitz continuous and satisfies $\iprods{Fx - Fx^{\star}, x - x^{\star}} \geq 0$ for all $x\in\dom{\Phi}$ and some $x^{\star} \in \zer{\Phi}$.
Let $\set{(x^k, y^k)}$ be generated by \eqref{eq:RFBS4NI} using $\eta \in \left(0, \frac{\sqrt{2}-1}{L}\right)$.
Then, we have the following statements.
\begin{itemize}
\item \textbf{The $\BigOs{1/\sqrt{k}}$ best-iterate convergence rate.} The following bound holds:
\begin{equation}\label{eq:RFBS4NI_convergence1}
\frac{1}{k+1}\sum_{l=0}^k\norms{Fx^l + \xi^l}^2  \leq \frac{1}{k+1}\sum_{l=0}^k\left[ \norms{Fx^l + \xi^l}^2 + \kappa \norms{Fx^l - Fy^{l-1}}^2\right]  \leq \frac{C_0 \norms{x^0 - x^{\star}}^2}{k+1},
\end{equation}
where $\kappa :=  \frac{2L^2\eta^2}{1 - L^2\eta^2} > 0$ and $C_0 := \frac{5L^2\eta^2 + 3}{3\eta^2\left[ 1 - (1 + \sqrt{2})L\eta \right]} > 0$.
\item \textbf{The $\BigOs{1/\sqrt{k}}$ last-iterate convergence rate.}
If $\Phi$ is additionally monotone, then we also have
\begin{equation}\label{eq:RFBS4NI_convergence2}
\norms{ Fx^k + \xi^k}^2 \leq \norms{ Fx^k + \xi^k}^2 + \kappa \norms{Fx^k - Fy^{k-1}}^2   \leq  \frac{C_0\norms{x^0 - x^{\star}}^2}{k+1}.
\end{equation}
As a consequence, we have the last-iterate convergence rate $\BigO{\frac{1}{\sqrt{k}}}$ of the residual norm $\norms{Fx^k + \xi^k}$.
\end{itemize}
\end{theorem}

%%% The proof of Theorem 3.1.
\begin{proof}
First, since $0 < \eta < \frac{\sqrt{2} - 1}{L}$, we have $1 - (1 + \sqrt{2})L\eta > 0$ and $\kappa := \frac{2L^2\eta^2}{1 - 2L^2\eta^2} < \frac{2}{3}$.
From  \eqref{eq:RFBS4NI_est2}, we have
\begin{equation*}
\arraycolsep=0.2em
\begin{array}{lcl}
\norms{ Fx^k + \xi^k}^2 +  \kappa\norms{Fx^k - Fy^{k-1}}^2   & \leq &  \frac{5L^2\eta^2 + 3}{3\eta^2}\norms{x^k - y^k}^2 + \left( \frac{5L^2\eta^2 + 3}{5\eta^2} + \frac{2L^2}{3}\right)\norms{x^k - y^{k-1}}^2 \vspace{1ex}\\
& \leq & \frac{5L^2\eta^2 + 3}{3\eta^2}\left[ \norms{x^k - y^k}^2 + \norms{x^k - y^{k-1}}^2 \right].
\end{array} 
\end{equation*}
Combining this estimate and \eqref{eq:RFBS4NI_est1}, we get
\begin{equation*}
\arraycolsep=0.2em
\begin{array}{lcl}
\frac{3\eta^2\left[ 1 - (1 + \sqrt{2})L\eta \right]}{5L^2\eta^2 + 3} \left[ \norms{ Fx^k + \xi^k}^2 +  \kappa \norms{Fx^k - Fy^{k-1}}^2  \right]  
&\leq &  (1 - (1 + \sqrt{2})L\eta)\left[\norms{x^k - y^k}^2 + \norms{x^k - y^{k-1}}^2  \right] \vspace{1ex}\\
&  \leq & \Vc_k - \Vc_{k+1}.
\end{array} 
\end{equation*}
Summing up this inequality from $l=0$ to $l=k$, and using $y^{-1} := x^0$, we get
\begin{equation*}
\arraycolsep=0.2em
\begin{array}{lcl}
\frac{3\eta^2\left[ 1 - (1 + \sqrt{2})L\eta \right]}{5L^2\eta^2 + 3} \sum_{l=0}^k \left[ \norms{ Fx^l + \xi^l}^2 +  \kappa \norms{Fx^l - Fy^{l-1}}^2  \right]  
&  \leq & \Vc_0 - \Vc_{k+1} \leq \Vc_0 = \norms{x^0 - x^{\star}}^2.
\end{array} 
\end{equation*}
This inequality implies \eqref{eq:RFBS4NI_convergence1}.
Finally, combining \eqref{eq:RFBS4NI_convergence1} and \eqref{eq:RFBS4NI_est3}, we obtain \eqref{eq:RFBS4NI_convergence2}.
\end{proof}
%%% End of the proof.

\begin{remark}\label{re:RFBS4NI_remark1}
We can modify \eqref{eq:RFBS4NI} to capture adaptive parameters as $y^k := x^k + \beta_k(x^k - x^{k-1})$ and $x^{k+1} := J_{\eta_kT}(x^k - \eta_kFy^k)$, where $\eta_k := \beta_k\eta_{k-1}$ for some $\beta_k > 0$.
Then, by imposing appropriate bounds on $\eta_k$, we can still prove the convergence of this variant by modifying the proof of Theorem~\ref{th:RFBS_convergence}.
We also note that our best-iterate convergence analysis of \eqref{eq:RFBS4NI} in this paper is relatively different from \cite{cevher2021reflected}, while the last-iterate convergence rate analysis is new and very simple.
\end{remark}

%%%% The Golden-Ratio Method.
\beforesubsec
\subsection{The golden ratio method for \eqref{eq:NI}}\label{subsec:GR4NI}
\aftersubsec
The golden ratio (GR) method for solving \eqref{eq:NI} is presented as follows.
Starting from $x^0 \in \dom{\Phi}$, we set $y^{-1} := x^0$, and at each iteration $k \geq 0$, we update
\begin{equation}\label{eq:GR4NI}
\left\{\begin{array}{lcl}
y^k &:= & \frac{\omega -1}{\omega}x^k + \frac{1}{\omega}y^{k-1}, \vspace{1ex}\\
x^{k+1} &:= & J_{\eta T}(y^k - \eta Fx^k),
\end{array}\right.
\tag{GR2}
\end{equation}
where $J_{\eta T}$ is the resolvent of $\eta T$, $\omega > 1$ is given, and $\eta \in (0, \frac{\omega}{2L})$.

This method was proposed by Malitsky in \cite{malitsky2019golden} to solve monotone \eqref{eq:MVIP}, where $\omega$ is chosen as $\omega := \frac{\sqrt{5} + 1}{2}$, leading to the name: \textit{golden ratio}.
We now extend it to solve \eqref{eq:NI} for the case $F$ is monotone and $L$-Lipschitz continuous, and $T$ is maximally $3$-cyclically monotone.
Moreover, we extend our analysis for any $\omega \in (1,  1 + \sqrt{3})$ instead of fixing $\omega := \frac{\sqrt{5} + 1}{2}$.
We call this extension the \ref{eq:GR4NI}$+$ scheme.

Let us denote $\breve{w}^k := Fx^{k-1} + \xi^k$ for $\xi^k \in Tx^k$.
Then, we can rewrite the second line of \eqref{eq:GR4NI} as $x^{k+1} :=  y^k - \eta(Fx^k + \xi^{k+1}) = y^k - \eta \breve{w}^{k+1}$ for $\xi^{k+1} \in Tx^{k+1}$.
In this case, we have $x^k = y^{k-1} - \eta(Fx^{k-1} + \xi^k)$, leading to $y^{k-1} = x^k + \eta \breve{w}^k$.
Combining this expression and the first line of \eqref{eq:GR4NI}, we have $y^k = \frac{\omega - 1}{\omega}x^k + \frac{1}{\omega}(x^k + \eta\breve{w}^k) = x^k + \frac{\eta}{\omega}\breve{w}^k$.
Consequently, we can rewrite \eqref{eq:GR4NI} equivalently as follows:
\begin{equation}\label{eq:GR4NI_reform}
\arraycolsep=0.2em
\left\{\begin{array}{lcllcl}
y^k &:= & x^k + \frac{\eta}{\omega}\breve{w}^k, \vspace{1ex}\\
x^{k+1} &:= & y^k - \eta(Fx^k + \xi^{k+1}) &= & y^k - \eta\breve{w}^{k+1}.
\end{array}\right.
\end{equation}
If we eliminate $y^k$, then we obtain
\begin{equation}\label{eq:GR4NI_reform2}
\arraycolsep=0.2em
\left\{\begin{array}{lcl}
x^{k+1} &:= & J_{\eta T}\left(x^k - \eta \left( Fx^k - \frac{1}{\omega}(Fx^{k-1} + \xi^k)  \right) \right) = x^k - \eta \breve{w}^{k+1} + \frac{\eta}{\omega}\breve{w}^k, \vspace{1ex}\\
\xi^{k+1} &:= & \frac{1}{\eta}(x^k - x^{k+1}) -  \left( Fx^k - \frac{1}{\omega}(Fx^{k-1} + \xi^k)  \right).
\end{array}\right.
\end{equation}
The convergence of \eqref{eq:GR4NI} is established based on the following key lemma.

%%% Lemma 5.3.
\begin{lemma}\label{le:GR4NI_key_est1}
Suppose that $\zer{\Phi} \neq\emptyset$, $T$ in \eqref{eq:NI} is maximally $3$-cyclically monotone, and $F$ is $L$-Lipschitz continuous.
Let $\set{(x^k, y^k)}$ be generated by \eqref{eq:GR4NI} with $\omega > 1$.
Then, for any $x^{\star}\in\zer{\Phi}$, we have
\begin{equation}\label{eq:GR4NI_key_est1}
\hspace{-0ex}
\arraycolsep=0.2em
\begin{array}{lcl}
\omega\norms{y^{k+1} - x^{\star}}^2  +  (\omega - 1)(\omega - \gamma) \norms{x^{k+1} - x^k}^2 & \leq &  \omega\norms{y^k - x^{\star}}^2 + \frac{(\omega-1)L^2\eta^2}{\gamma}\norms{x^k - x^{k-1}}^2 \vspace{1ex}\\
&&  - {~} \omega(\omega-1)\norms{x^k - y^k}^2 -  \frac{(\omega - 1)(1-\omega^2+\omega)}{\omega}\norms{ x^{k+1} - y^k}^2 \vspace{1ex}\\
& & - {~}  2\eta(\omega-1)\iprods{Fx^k - Fx^{\star}, x^k - x^{\star}}.
\end{array}
\hspace{-3ex}
\end{equation}
\end{lemma}

%%% The proof of Lemma 5.3
\begin{proof}
Since $T$ is $3$-cyclically monotone, for $\xi^{k+1} \in Tx^{k+1}$, $\xi^k \in Tx^k$, and $x^{\star} \in Tx^{\star}$, we have 
\begin{equation}\label{eq:GR4NI_lm1_proof1}
\arraycolsep=0.2em
\begin{array}{lcll}
\iprods{\xi^{k+1}, x^{k+1} - x^{\star}} + \iprods{\xi^{\star}, x^{\star} - x^k} + \iprods{\xi^k, x^k - x^{k+1}} \geq 0.
\end{array}
\end{equation}
From \eqref{eq:GR4NI_reform}, we have $\eta\xi^{k+1} = y^k - x^{k+1} - \eta Fx^k$ and $\eta \xi^k = y^{k-1} - x^k - \eta Fx^{k-1}$. 
Moreover, since $x^{\star} \in \zer{\Phi}$, we have $Fx^{\star} + \xi^{\star} = 0$, leading to $\eta\xi^{\star} = - \eta Fx^{\star}$.
Substituting these expressions into \eqref{eq:GR4NI_lm1_proof1}, we have
\begin{equation*}
\begin{array}{lcll}
\iprods{y^k - x^{k+1} - \eta Fx^k, x^{k+1} - x^{\star}} + \iprods{y^{k-1} - x^k - \eta Fx^{k-1}, x^k - x^{k+1}}  - \iprods{Fx^{\star}, x^{\star} - x^k}  \geq 0.
\end{array}
\end{equation*}
However, since $y^k :=  \frac{\omega -1}{\omega}x^k + \frac{1}{\omega}y^{k-1}$ from the first line of \eqref{eq:GR4NI},  we get $y^{k-1} - x^k =  \omega(y^k - x^k)$.
Substituting this relation into the last inequality, rearranging the result, we obtain
\begin{equation}\label{eq:GR4NI_lm1_proof2}
\iprods{y^k - x^{k+1}, x^{k+1} - x^{\star}} +  \omega\iprods{x^k  - y^k,  x^{k+1} - x^k} + \eta\iprods{Fx^{k-1} - Fx^k, x^{k+1} - x^k} - \eta\iprods{Fx^k - Fx^{\star}, x^k - x^{\star}} \geq 0.
\end{equation}
By Young's inequality and the $L$-Lipschitz continuity of $F$, for any $\gamma > 0$, we get the first line of the following:
\begin{equation*}
\arraycolsep=0.2em
\begin{array}{lcl}
2\eta\iprods{Fx^{k-1} - Fx^k, x^{k+1} - x^k} &\leq & \frac{L^2\eta^2}{\gamma}\norms{x^k - x^{k-1}}^2 + \gamma \norms{x^{k+1} - x^k}^2, \vspace{1ex}\\
2\iprods{y^k - x^{k+1}, x^{k+1} - x^{\star} } &= & \norms{y^k - x^{\star}}^2 - \norms{x^{k+1} - x^{\star}}^2 - \norms{x^{k+1} - y^k}^2, \vspace{1ex}\\
2\iprods{x^k - y^k, x^{k+1} - x^k} &= & \norms{x^{k+1} - y^k}^2 -\norms{x^k - y^k}^2 - \norms{x^{k+1} - x^k}^2.
\end{array}
\end{equation*}
Substituting these expressions into  \eqref{eq:GR4NI_lm1_proof2}, and rearranging the result, we can show that
\begin{equation}\label{eq:GR4NI_lm1_proof3}
\arraycolsep=0.2em
\begin{array}{lcl}
\norms{x^{k+1} - x^{\star}}^2 &\leq & \norms{y^k - x^{\star}}^2 + (\omega - 1) \norms{x^{k+1} - y^k}^2 -   \omega \norms{x^k - y^k}^2 - ( \omega  - \gamma )\norms{x^{k+1} - x^k}^2 \vspace{1ex}\\
&& +  {~}   \frac{L^2\eta^2}{\gamma}\norms{x^k - x^{k-1}}^2  - 2\eta \iprods{Fx^k - Fx^{\star}, x^k - x^{\star}}.
\end{array}
\end{equation}
Now, using $(\omega - 1)x^{k+1} = \omega y^{k+1} - y^k$ and $\omega(y^{k+1} - y^k) = (\omega-1)(x^{k+1} - y^k)$ from the first line of \eqref{eq:GR4NI}, we can derive that 
\begin{equation*} 
\arraycolsep=0.2em
\begin{array}{lcl}
(\omega -1)^2\norms{x^{k+1} - x^{\star}}^2 %&= & \omega^2\norms{y^{k+1} - x^{\star}}^2 - 2\omega\iprods{y^{k+1} - x^{\star}, y^k - x^{\star}} + \norms{y^k - x^{\star}}^2 \vspace{1ex}\\
&= & \omega(\omega - 1)\norms{y^{k+1} - x^{\star}}^2 - (\omega - 1) \norms{y^k - x^{\star}}^2 + \omega\norms{y^{k+1} -  y^k}^2 \vspace{1ex}\\
&= & \omega(\omega - 1)\norms{y^{k+1} - x^{\star}}^2 - (\omega - 1) \norms{y^k - x^{\star}}^2 + \frac{(\omega-1)^2}{\omega}\norms{x^{k+1} -  y^k}^2.
\end{array}
\end{equation*}
Simplifying this expression to get $(\omega - 1)\norms{x^{k+1} - x^{\star}}^2 = \omega\norms{y^{k+1} - x^{\star}}^2 - \norms{y^k - x^{\star}}^2 + \frac{(\omega-1)}{\omega}\norms{x^{k+1} -  y^k}^2$.
Combining it and \eqref{eq:GR4NI_lm1_proof3}, and rearranging the result, we obtain \eqref{eq:GR4NI_key_est1}.
\end{proof}

Now, we are ready to state the convergence of \eqref{eq:GR4NI} in the following theorem.

%%% Theorem 5.1.
\begin{theorem}\label{th:GR4NI_convergence}
Assume that $\zer{\Phi} \neq\emptyset$, $T$ in \eqref{eq:NI} is maximally $3$-cyclically monotone, and $F$ in \eqref{eq:NI} is  $L$-Lipschitz continuous and satisfies $\iprods{Fx - Fx^{\star}, x - x^{\star}} \geq 0$ for all $x\in\dom{\Phi}$ and some $x^{\star} \in \zer{\Phi}$.
Let $\set{(x^k, y^k)}$ be generated by \eqref{eq:GR4NI}. 
Then, the following statements hold.
\begin{itemize}
\item \textbf{The best-iterate rate of \ref{eq:GR4NI}.}  If $1 < \omega \leq \frac{1+\sqrt{5}}{2}$ and  $\eta \in \left(0, \frac{\omega}{2L}\right)$, then
\begin{equation}\label{eq:GR4NI_convergence1}
\frac{1}{k+1}\sum_{l=0}^{k}\norms{Fx^l + \xi^l}^2  \leq \frac{1}{k+1}\sum_{l=0}^k (\omega-1)\left[ \omega\norms{x^l - y^l}^2 + \varphi \cdot \norms{x^{l} -  x^{l-1}}^2\right]  \leq \frac{C_0 \norms{x^0 - x^{\star}}^2}{k+1},
\end{equation}
where $\varphi :=  \frac{\omega^2-4L^2\eta^2}{2\omega} > 0$ and $C_0 :=  \frac{ (\omega^2 - 2L^2\eta^2)\omega^2}{(\omega^2 - 4L^2\eta^2)\eta^2(\omega-1)} > 0$.

\item \textbf{The best-iterate rate for \ref{eq:GR4NI}$+$.} If $ \frac{1+\sqrt{5}}{2} < \omega < 1 + \sqrt{3}$ and $0 < \eta < \frac{\psi}{2L}$, then 
\begin{equation}\label{eq:GR4NI_convergence2}
\frac{1}{k+1}\sum_{l=0}^{k}\norms{Fx^l + \xi^l}^2  \leq \frac{1}{k+1}\sum_{l=0}^k(\omega-1)\left[ \psi\cdot \norms{x^l - y^l}^2 + \kappa \cdot \norms{x^{l} -  y^{l-1}}^2\right]  \leq \frac{\hat{C}_0 \norms{x^0 - x^{\star}}^2}{k+1},
\end{equation}
where $\psi :=  \frac{2\omega + 2 - \omega^2}{\omega} > 0$, $\kappa := \frac{\psi^2 - 4L^2\eta^2}{2\psi}$, and $\hat{C}_0 :=  \frac{[\psi^2 - 2L^2\eta^2(2\omega^2 - \psi^2)]\omega}{(\omega-1)(\psi^2 - 4L^2\eta^2)\eta^2\psi} > 0$.

%\item \textbf{The last-iterate rate of \ref{eq:GR4NI}$+$.}
%If $\Phi$ is additionally monotone, $\omega \geq \frac{9}{4}$, and $L^2\eta^2 \leq \frac{\omega+1}{15\omega(4\omega+5)}$, then we also have
%\begin{equation}\label{eq:GR4NI_convergence2}
%\norms{ Fx^K + \xi^K}   \leq  \frac{\sqrt{M_0C_0}\norms{x^0 - x^{\star}}}{\sqrt{K+1}}.
%\end{equation}
%where $M_0 :=  \frac{[2\omega + (3\omega-1)L^2\eta^2](9\omega + 1)}{(\omega-1)\eta^2} + \frac{2[ L^2\eta^2(4\omega+5)(\omega-1) + 4\omega + 1]\omega}{\eta^2}$.
%As a consequence, we have the last-iterate convergence rate $\norms{Fx^K + \xi^K} = \BigO{\frac{1}{\sqrt{K}}}$ for \eqref{eq:GR4NI}.
\end{itemize}
\end{theorem}

%%% Proof of Theorem 5.1.
\begin{proof}
First, to guarantee that $1 + \omega - \omega^2 \geq 0$ and $\omega > 1$, we need to choose $1 < \omega \leq \frac{\sqrt{5} + 1}{2}$.
If $0 < \eta < \frac{\omega}{2L}$, then by choosing $\gamma := \frac{\omega}{2}$, we have  $\psi := \frac{(\omega-1)(\omega \gamma - \gamma^2 - L^2\eta^2)}{\gamma} = \frac{(\omega-1)(\omega^2 - 4L^2\eta^2)}{2\omega} > 0$.
Using this relation and $\iprods{Fx^k - Fx^{\star}, x^k - x^{\star}} \geq 0$, if we define $\Vc_k := \omega \norms{y^k - x^{\star}}^2 + \frac{\omega(\omega - 1)}{2} \norms{x^k - x^{k-1}}^2 \geq 0$, then we can deduce from \eqref{eq:GR4NI_key_est1} that
\begin{equation}\label{eq:GR4NI_th1_proof1}
\arraycolsep=0.2em
\begin{array}{lcl}
\Vc_{k+1} &\leq & \Vc_k  -  \psi \cdot \norms{ x^k - x^{k-1}}^2  - \omega(\omega-1)\norms{x^k - y^k}^2.
\end{array}
\end{equation}
Next, using $y^k - x^k = \frac{\eta}{\omega}\breve{w}^k$ and $\breve{w}^k = Fx^{k-1} + \xi^k$, by Young's inequality, we have 
\begin{equation}\label{eq:GR4NI_th1_proof2}
\arraycolsep=0.2em
\begin{array}{lcl}
\norms{w^k}^2 & = & \norms{Fx^k + \xi^k}^2 \leq \left(1 + \frac{\psi\omega}{L^2\eta^2(\omega-1)} \right) \norms{Fx^k - Fx^{k-1}}^2 + \left( 1+ \frac{L^2\eta^2(\omega-1)}{\psi\omega}\right)\norms{\breve{w}^k}^2 \vspace{1ex}\\
& \leq &  \left(1 + \frac{\psi\omega}{L^2\eta^2(\omega-1)} \right) L^2\norms{x^k - x^{k-1}}^2 +  \left( 1+ \frac{L^2\eta^2(\omega-1)}{\psi\omega}\right)\frac{\omega^2}{\eta^2}\norms{x^k - y^k}^2 \vspace{1ex}\\
& = & \frac{(\omega-1)}{\psi\eta^2(\omega-1)} \left[ \psi\cdot\norms{x^k - x^{k-1}}^2 + \omega(\omega-1)\norms{x^k - y^k}^2\right].
\end{array}
\end{equation} 
Combining this estimate and \eqref{eq:GR4NI_th1_proof1}, and noting that $\Vc_k \geq 0$, we can show that
\begin{equation*}
\arraycolsep=0.2em
\begin{array}{lcl}
\sum_{l=0}^k\norms{w^l}^2 & \leq & \frac{L^2\eta^2(\omega-1) + \psi\omega}{\psi\eta^2(\omega-1)} \sum_{l=0}^k \left[ \psi \cdot \norms{x^{l} - x^{l-1}}^2 +   \omega(\omega-1) \norms{x^l - y^l}^2\right] \vspace{1ex}\\
& \leq &    \frac{L^2\eta^2(\omega-1) + \psi\omega}{\psi\eta^2(\omega-1)}   \left[ \Vc_0 - \Vc_{k+1} \right] \leq   \frac{L^2\eta^2(\omega-1) + \psi\omega}{\psi\eta^2(\omega-1)}  \cdot \Vc_0  \vspace{1ex}\\
& = &  \frac{ (\omega^2 - 2L^2\eta^2)\omega^2}{(\omega^2 - 4L^2\eta^2)\eta^2(\omega-1)} \cdot \norms{x^0 - x^{\star}}^2,
\end{array}
\end{equation*} 
which is exactly \eqref{eq:GR4NI_convergence1},  where we have used $\Vc_0 := \omega\norms{y^0 - x^{\star}}^2 + \frac{\omega(\omega-1)}{2} \norms{x^0 - x^{-1}}^2 = \omega\norms{x^0 - x^{\star}}^2$ due to $x^{-1} = y^0 = x^0$.

Next, if $1.6180 \approx \frac{1 + \sqrt{5}}{2} < \omega < 1 + \sqrt{3} \approx 2.7321$, then we have $\omega^2 - \omega - 1 > 0$ and $\psi := \omega - \frac{2(\omega^2-\omega-1)}{\omega} > 0$.
In this case, using $\norms{x^{k+1} - y^k}^2 \leq 2\norms{x^{k+1} - x^k}^2 + 2\norms{y^k - x^k}^2$ and $\iprods{Fx^k - Fx^{\star}, x^k - x^{\star}} \geq 0$ into \eqref{eq:GR4NI_key_est1}, rearranging the result, and using $\gamma := \frac{\psi}{2}$, we get
\begin{equation}\label{eq:GR4NI_key_est1_v2}
\hspace{-0ex}
\arraycolsep=0.2em
\begin{array}{lcl}
\omega\norms{y^{k+1} - x^{\star}}^2  +  \frac{\psi(\omega - 1)}{2} \norms{x^{k+1} - x^k}^2  & \leq & \omega\norms{y^k - x^{\star}}^2 +  \frac{\psi(\omega - 1)}{2}\norms{x^k-x^{k-1}}^2  - \psi(\omega-1)\norms{x^k - y^k}^2 \vspace{1ex}\\
&& -  {~} \frac{(\omega-1)(\psi^2 - 4L^2\eta^2)}{2\psi} \norms{x^k - x^{k-1}}^2.
\end{array}
\hspace{-3ex}
\end{equation}
Similar to the proof of \eqref{eq:GR4NI_th1_proof2}, we have $\norms{w^k}^2 \leq \frac{\psi^2 - 2L^2\eta^2(2\omega^2 - \psi^2)}{(\psi^2 - 4L^2\eta^2)\eta^2\psi} \big[\frac{\psi^2 - 4L^2\eta^2}{2\psi}\norms{x^k - x^{k-1}}^2 + \psi\norms{x^k - y^k}^2\big]$.
Combining this inequality and \eqref{eq:GR4NI_key_est1_v2}, with same argument as in the proof of \eqref{eq:GR4NI_convergence1}, we obtain \eqref{eq:GR4NI_convergence2}.
%
%
%To prove the last-iterate convergence rate, from \eqref{eq:GR4NI_lm2_monotone} and \eqref{eq:GR4NI_th1_proof2}, with $\beta := L^2\eta^2(4\omega+5)$, we have
%\begin{equation*} 
%\arraycolsep=0.2em
%\begin{array}{lcl}
%\Tc_{k+1} &:= &  (\omega - 1)\norms{w^{k+1}}^2   +  (4\omega + 1)\norms{w^{k+1} - \breve{w}^{k+1}}^2 + 2\beta(\omega-1) \norms{\breve{w}^{k+1}}^2 \vspace{1ex}\\
%&\leq & (9\omega + 1)\norms{w^{k+1}}^2   + 2[\beta(\omega-1) + 4\omega + 1] \norms{\breve{w}^{k+1}}^2 \vspace{1ex}\\
%& \overset{\tiny\eqref{eq:GR4NI_th1_proof2}}{\leq} & \frac{[2\omega + (3\omega-1)L^2\eta^2](9\omega + 1)}{\omega(\omega-1)\eta^2} \big[ (\omega - 1) \norms{x^{k+1} - y^k}^2 +  \omega \norms{x^k - y^k}^2 \big] +  \frac{2[\beta(\omega-1) + 4\omega + 1]}{\eta^2} \norms{x^{k+1} - y^k}^2 \vspace{1ex}\\
%&\leq & M_0 \big[ \frac{(\omega - 1)}{\omega} \norms{x^{k+1} - y^k}^2 +  \norms{x^k - y^k}^2 \big],
%\end{array}
%\end{equation*} 
%where $M_0 :=  \frac{[2\omega + (3\omega-1)L^2\eta^2](9\omega + 1)}{(\omega-1)\eta^2} + \frac{2[ L^2\eta^2(4\omega+5)(\omega-1) + 4\omega + 1]\omega}{\eta^2}$.
%Combining this estimate and \eqref{eq:GR4NI_convergence1}, we can show the last-iterate convergence rates of \eqref{eq:GR4NI} as in \eqref{eq:GR4NI_convergence2}.
\end{proof}
%%% End of proof.

%%%%%%%%%%%%%%%%%%%%%%%%%%%%%%%%%%%%%%%%%%%%
%%%% 6. Accelerated Extragradient Methods for Nonlinear Inclusions.
%%%%%%%%%%%%%%%%%%%%%%%%%%%%%%%%%%%%%%%%%%%%
\beforesec
\section{Accelerated Extragradient Methods for Nonlinear Inclusions}\label{sec:EAG4NI}
\aftersec
\textbf{Introduction.}
The convergence rate on the residual norm $\norms{Fx^k + \xi^k}$ of the EG method and its variants using constant stepsize  discussed so far is $\mathcal{O}\big(1/\sqrt{k}\big)$, which is unimprovable for standard and constant stepsize EG-type methods as shown in \cite{golowich2020last}.
In this section, we survey recent development on accelerated methods that can theoretically achieve a $\BigO{1/k}$ last-iterate convergence rate on $\norms{Fx^k + \xi^k}$ using variable stepsizes.
We will present two different approaches to develop accelerated methods for solving \eqref{eq:NE} and \eqref{eq:NI} without using averaging sequences.
The first one relies on Halpern's fixed-point iteration \cite{halpern1967fixed}, and the second approach leverages  Nesterov's accelerated techniques.
Halpern's fixed-point iteration is a classical method to approximate a fixed-point of a nonexpansive operator, or equivalently, to find a root of a co-coercive operator. 
This method has been intensively studied in fixed-point theory, but the first work showing a $\BigO{1/k}$ last-iterate convergence rate is due to F. Lieder in \cite{lieder2021convergence}.
This method was then extended and intensively studied in \cite{diakonikolas2020halpern} for root-finding problems and VIPs.
In a pioneering work \cite{yoon2021accelerated}, Yoon and Ryu extended Halpern's fixed-point iteration to the EG method for solving \eqref{eq:NE}, which is called \textit{extra-anchored gradient} (EAG) method.
This new method still achieves $\BigO{1/k}$ last-iterate convergence but only requires the monotonicity and Lipschitz continuity of $F$.
Lee and Kim in \cite{lee2021fast} further advanced \cite{yoon2021accelerated} to the co-hypomonotone setting of \eqref{eq:NE} and still achieved the same rates.
The authors in \cite{tran2021halpern} exploited the technique in \cite{yoon2021accelerated} and applied it to the past-extragradient method in \cite{popov1980modification} and obtained a past extra-anchored gradient (PEAG) method that has the same $\BigO{1/k}$-rates (up to a constant factor).
Recently, \cite{cai2022accelerated} and \cite{cai2022baccelerated} expanded the results in \cite{lee2021fast,tran2021halpern,yoon2021accelerated} to develop methods for solving \eqref{eq:VIP} and \eqref{eq:NI}, and preserved the same $\BigO{1/k}$ last-iterate convergence rates on $\norms{Fx^k + \xi^k}$.

In this section, we summarize the above mentioned results and provide a unified convergence analysis obtained from a recent work \cite{tran2023extragradient} that covers all the results from  \cite{cai2022accelerated,cai2022baccelerated,lee2021fast,tran2021halpern,yoon2021accelerated} in a unified fashion.

%In this section, we provide an alternative convergence analysis for the Halpern fixed-point-type variants of \eqref{eq:EG4NI}, which is called EAG (extra-anchored gradient) method \cite{lee2021fast,yoon2021accelerated} and PEAG (past extra-anchored gradient) method in \cite{tran2021halpern}, respectively, where these names are  stemmed from \cite{yoon2021accelerated}.
%These schemes have been extended to \eqref{eq:NI} in \cite{cai2022accelerated} and \cite{cai2022baccelerated}, respectively.
%Our algorithmic derivation as well as our analysis can be viewed as direct extensions of the results in \cite{lee2021fast,tran2021halpern,yoon2021accelerated}.
%Moreover, together with the new analysis, our parameter selection for \eqref{eq:PEAG4NI} is different from that in \cite{cai2022baccelerated}.

\beforesubsec
\subsection{The extra-anchored gradient method for \eqref{eq:NI}}\label{subsec:EAG2_v2}
\aftersubsec
\noindent\textbf{The algorithm.}
The extra-anchored gradient (EAG) for solving \eqref{eq:NI} we discuss here is presented as follows.
Starting from $x^0 \in \dom{\Phi}$, at each iteration $k\geq 0$, we update
\begin{equation}\label{eq:EAG4MVIP}
\arraycolsep=0.2em
\left\{\begin{array}{lcl}
y^k &:= & J_{\hat{\eta}_k T}\left( \tau_kx^0 + (1-\tau_k)x^k - \hat{\eta}_k Fx^k \right),   \vspace{1ex}\\
x^{k+1} &:= & J_{\eta_kT}\left( \tau_kx^0 + (1-\tau_k)x^k - \eta_k Fy^k \right),
\end{array}\right.
\tag{EAG2}
\end{equation}
where $\tau_k \in (0, 1)$, $\hat{\eta}_k > 0$, and $\eta_k > 0$ are given, which will be determined later.
Here, we assume that $T$ is maximally $3$-cyclically monotone, and hence covers the special cases $T = \Nc_{\Xc}$, the normal cone of $\Xc$ and $T = \partial{g}$, the subdifferential of a convex function $g$.
In fact, \cite{cai2022accelerated} considers the special case \eqref{eq:VIP} of \eqref{eq:NI} when  $T := \Nc_{\Xc}$ is normal cone of a nonempty, closed, and convex set $\Xc$, and $\hat{\eta}_k := \eta_k$.
As we can see, the scheme \eqref{eq:EAG4MVIP} purely extends the extra-anchored gradient (EAG) scheme from \cite{yoon2021accelerated} for \eqref{eq:NE} to \eqref{eq:NI}, when $F$ is monotone and Lipschitz continuous, and $T$ is maximally $3$-cyclically monotone.

\vspace{0.75ex}
\noindent\textbf{Convergence analysis.}
For simplicity of analysis, we recall the following quantities defined ealier:
\begin{equation}\label{eq:EAG4MVIP_ex2}
w^k := Fx^k + \xi^k, \quad \hat{w}^k := Fy^{k-1} + \xi^k,  \quad\text{and} \quad \tilde{w}^k := Fx^k + \zeta^k,
\end{equation}
where $\xi^k \in Tx^k$ and $\zeta^k \in Ty^k$.
Then, we can equivalently rewrite \eqref{eq:EAG4MVIP} as follows:
\begin{equation}\label{eq:EAG4MVIP_reform}
\arraycolsep=0.2em
\left\{\begin{array}{lcl}
y^k &:= & \tau_kx^0 + (1-\tau_k)x^k - \hat{\eta}_k\tilde{w}^k, \vspace{1ex}\\
x^{k+1} &:= & \tau_kx^0 + (1-\tau_k)x^k - \eta_k\hat{w}^{k+1}.
\end{array}\right.
\end{equation}
To establish the convergence of \eqref{eq:EAG4MVIP}, we use the following potential function as in \cite{cai2022accelerated,lee2021fast,tran2023extragradient,yoon2021accelerated}:
\begin{equation}\label{eq:EAG4MVIP_potential_func}
\Vc_k :=  a_k\norms{w^k}^2 + b_k\iprods{w^k, x^k - x^0},
\end{equation}
where $a_k > 0$ and $b_k > 0$ are given parameters.
Let us prove the convergence of \eqref{eq:EAG4MVIP}.

%%%% Theorem 6.1.
\begin{theorem}\label{th:EAG4MVIP_convergence}
For \eqref{eq:NI}, assume that $\zer{\Phi} \neq\emptyset$,  $F$ is $L$-Lipschitz continuous and monotone, and $T$ is maximally $3$-cyclically monotone.
Let $\sets{(x^k, y^k)}$ be generated by \eqref{eq:EAG4MVIP} using
\begin{equation}\label{eq:EAG4MVIP_param_update}
\tau_k := \frac{1}{k+2}, \quad \eta_k :=  \eta \in \left(0,  \frac{1}{L}\right], \quad\text{and} \quad  \hat{\eta}_k := (1-\tau_k)\eta.
\end{equation}
Then, for all $k\geq 0$ and any $x^{\star}\in\zer{\Phi}$, the following result holds:
\begin{equation}\label{eq:EAG4MVIP_convergence1}
\norms{Fx^k + \xi^k}^2 \leq  \frac{4\norms{x^0 - x^{\star}}^2 + 2\eta^2\norms{Fx^0 + \xi^0}^2}{\eta^2(k + 1)^2}, \quad\text{where}\quad \xi^k \in Tx^k.
\end{equation}
Consequently, we have the last-iterate convergence rate $\BigO{1/k}$ of the residual norm $\norms{Fx^k + \xi^k}$.
\end{theorem}

%%%% Proof of Theorem 6.1.
\begin{proof}
Since $T$ is maximally $3$-cyclically monotone, $\xi^{k+1} \in Tx^{k+1}$, $\xi^k \in Tx^k$, and $\zeta^k \in Ty^k$, we have
\begin{equation*} 
\arraycolsep=0.2em
\begin{array}{lcl}
\iprods{\xi^{k+1}, x^{k+1} - x^k} + \iprods{\xi^k, x^k - y^k} + \iprods{\zeta^k, y^k - x^{k+1}} \geq 0.
\end{array}
\end{equation*}
By the monotonicity of $F$, we also have $\iprods{Fx^{k+1} - Fx^k, x^{k+1} - x^k} \geq 0$.
Summing up this inequality and the last one, and then using the fact that $w^{k+1} = Fx^{k+1} + \xi^{k+1}$, $w^k := Fx^k + \xi^k$, and $\tilde{w}^k := Fx^k + \zeta^k$, we obtain
\begin{equation}\label{eq:EAG4MVIP_proof1} 
\arraycolsep=0.2em
\begin{array}{lcl}
\iprods{w^{k+1}, x^{k+1} - x^k} - \iprods{\tilde{w}^k, x^{k+1} - x^k} + \iprods{w^k - \tilde{w}^k, x^k - y^k} \geq 0.
\end{array}
\end{equation}
Now, from \eqref{eq:EAG4MVIP_reform}, we have
\begin{equation*} 
\arraycolsep=0.2em
\begin{array}{lcl}
x^{k+1} - x^k &= & -\frac{\tau_k}{1 - \tau_k}(x^{k+1} - x^0) - \frac{\eta_k}{1-\tau_k}\hat{w}^{k+1}, \vspace{1ex}\\
x^{k+1} - x^k &= & - \tau_k(x^k - x^0) - \eta_k\hat{w}^{k+1}, \vspace{1ex}\\
x^k - y^k &= & \tau_k(x^k - x^0) + \hat{\eta}_k\tilde{w}^k.
\end{array}
\end{equation*}
Substituting these relations into \eqref{eq:EAG4MVIP_proof1}, and rearranging terms, we arrive at
\begin{equation*} 
\arraycolsep=0.2em
\begin{array}{lcl}
\tau_k\iprods{w^k, x^k - x^0} - \frac{\tau_k}{1-\tau_k}\iprods{w^{k+1}, x^{k+1} - x^0} &\geq & \frac{\eta_k}{1-\tau_k}\iprods{w^{k+1}, \hat{w}^{k+1}} - \eta_k\iprods{\tilde{w}^k, \hat{w}^{k+1}} - \hat{\eta}_k\iprods{w^k, \tilde{w}^k} + \hat{\eta}_k\norms{\tilde{w}^k}^2.
\end{array}
\end{equation*}
Multiplying this inequality by $\frac{b_k}{\tau_k}$ and assume that $b_{k+1} = \frac{b_k}{1-\tau_k}$, and then using \eqref{eq:EAG4MVIP_potential_func}, we can show that
\begin{equation}\label{eq:EAG4MVIP_proof2}  
\arraycolsep=0.2em
\begin{array}{lcl}
\Vc_k - \Vc_{k+1} 
%&= & a_k\norms{w^k}^2 - a_{k+1}\norms{w^{k+1}}^2 + b_k\iprods{w^k, x^k - x^0} - b_{k+1}\iprods{w^{k+1}, x^{k+1} - x^0} \vspace{1ex}\\
&= & a_k\norms{w^k}^2 - a_{k+1}\norms{w^{k+1}}^2 + b_k\iprods{w^k, x^k - x^0} - \frac{b_k}{1-\tau_k}\iprods{w^{k+1}, x^{k+1} - x^0} \vspace{1ex}\\
&\geq & \frac{b_{k+1}\eta_k}{\tau_k}\iprods{w^{k+1} - \tilde{w}^k, \hat{w}^{k+1}} + b_{k+1}\eta_k\iprods{\tilde{w}^k, \hat{w}^{k+1}} - \frac{b_k\hat{\eta}_k}{\tau_k}\iprods{w^k, \tilde{w}^k} + \frac{b_k\hat{\eta}_k}{\tau_k}\norms{\tilde{w}^k}^2 \vspace{1ex}\\
&& + {~} a_k\norms{w^k}^2 - a_{k+1}\norms{w^{k+1}}^2.
\end{array}
\end{equation}
Next, from  \eqref{eq:EAG4MVIP_reform}, we have $x^{k+1} - y^k = -\eta_k\hat{w}^{k+1} + \hat{\eta}_k\tilde{w}^k$.
Using this expression and the $L$-Lipschitz continuity of $F$, we have $\norms{w^{k+1} - \hat{w}^{k+1}}^2 = \norms{Fx^{k+1} - Fy^k}^2 \leq L^2\norms{x^{k+1} - y^k}^2 = L^2\norms{\eta_k\hat{w}^{k+1} - \hat{\eta}_k\tilde{w}^k}^2$, leading to
\begin{equation*} 
\arraycolsep=0.2em
\begin{array}{lcl}
\norms{w^{k+1}}^2 + (1-L^2\eta_k^2)\norms{\hat{w}^{k+1}}^2 - 2\iprods{w^{k+1} - \tilde{w}^k, \hat{w}^{k+1}} -  2(1 - L^2\eta_k\hat{\eta}_k)\iprods{\hat{w}^{k+1}, \tilde{w}^k} - L^2\hat{\eta}_k^2\norms{\tilde{w}^k}^2 \leq 0.
\end{array}
\end{equation*}
Multiplying this inequality by $\frac{b_{k+1}\eta_k}{2\tau_k}$, adding the result to \eqref{eq:EAG4MVIP_proof1}, and using $\hat{\eta}_k = (1-\tau_k)\eta_k$, we obtain
\begin{equation*} 
\arraycolsep=0.2em
\begin{array}{lcl}
\Vc_k - \Vc_{k+1} &\geq & \left( \frac{b_{k+1}\eta_k}{2\tau_k} - a_{k+1} \right)\norms{w^{k+1}}^2 + \frac{b_{k+1}\eta_k(1-L^2\eta_k^2)}{2\tau_k}\norms{\hat{w}^{k+1}}^2 + a_k\norms{w^k}^2 
+ \frac{b_{k+1}\eta_k(1-\tau_k)^2(2 - L^2\eta_k^2)}{2\tau_k}\norms{\tilde{w}^k}^2 \vspace{1ex}\\
%+ \left(\frac{b_k\hat{\eta}_k}{\tau_k} - \frac{b_{k+1}\eta_k L^2\hat{\eta}_k^2}{2\tau_k} \right)\norms{\tilde{w}^k}^2 \vspace{1ex}\\
& & - {~} \frac{b_{k+1}\eta_k(1 - L^2\eta_k^2)(1 - \tau_k)}{\tau_k}  \iprods{\tilde{w}^k, \hat{w}^{k+1}} - \frac{b_{k+1}\eta_k(1-\tau_k)^2}{\tau_k}\iprods{w^k, \tilde{w}^k} \vspace{1ex}\\
%& & + {~} \left(b_{k+1}\eta_k - \frac{b_{k+1}\eta_k(1-L^2\eta_k\hat{\eta}_k)}{\tau_k} \right)  \iprods{\tilde{w}^k, \hat{w}^{k+1}} - \frac{b_k\hat{\eta}_k}{\tau_k}\iprods{w^k, \tilde{w}^k}.
&= &  \frac{b_{k+1}\eta_k(1-L^2\eta_k^2)}{2\tau_k}\norms{\hat{w}^{k+1} - (1-\tau_k)\tilde{w}^k}^2  + \frac{b_{k+1}\eta_k(1-\tau_k)^2}{2\tau_k}\norms{w^k - \tilde{w}^k}^2 \vspace{1ex}\\
&& + {~} \left( \frac{b_{k+1}\eta_k}{2\tau_k} - a_{k+1} \right)\norms{w^{k+1}}^2  +  \left( a_k -  \frac{b_k\eta_k(1-\tau_k)}{2\tau_k}\right) \norms{w^k}^2.
\end{array}
\end{equation*}
Let us choose $\eta_k := \eta \in \left(0, \frac{1}{L}\right]$ as in \eqref{eq:EAG4MVIP_param_update}, $\tau_k := \frac{1}{k+2}$ and $a_k := \frac{b_k\eta(1-\tau_k)}{2\tau_k} = \frac{\eta b_k(k+1)}{2}$.
Then, we have $a_{k+1} = \frac{\eta b_{k+1}(k+2)}{2} = \frac{\eta b_{k+1}}{2\tau_k}$.
Moreover, since $b_{k+1} = \frac{b_k}{1-\tau_k} = \frac{b_k(k+2)}{k+1}$.
By induction, we obtain $b_k = b_0(k+1)$ for some $b_0 > 0$.
Using these parameters into the last estimate, we  obtain
\begin{equation*} 
\arraycolsep=0.2em
\begin{array}{lcl}
\Vc_k - \Vc_{k+1} &\geq & \frac{b_{k+1}\eta(1-L^2\eta^2)}{2\tau_k}\norms{\hat{w}^{k+1} - (1-\tau_k)\tilde{w}^k}^2  + \frac{b_{k+1}\eta(1-\tau_k)^2}{2\tau_k}\norms{w^k - \tilde{w}^k}^2 \geq 0.
\end{array}
\end{equation*}
Finally, using $\iprods{w^k, x^k - x^{\star}} \geq 0$ for $x^{\star} \in \zer{\Phi}$ and $b_k = b_0(k+1)$, we can lower bound $\Vc_k$ as
\begin{equation*}
\arraycolsep=0.2em
\begin{array}{lcl}
\Vc_k & = & a_k\norms{w^k}^2 + b_k\iprods{w^k, x^{\star} - x^0} + b_k\iprods{w^k, x^k - x^{\star}}  \geq a_k \norms{w^k}^2 - b_k\norms{w^k}\norms{x^0 - x^{\star}} \vspace{1ex}\\
&\geq &  \left( a_k  - \frac{\eta  b_k^2}{4b_0} \right) \norms{w^k}^2 - \frac{b_0}{\eta}\norms{x^0 - x^{\star}}^2 =  \frac{b_0\eta (k + 1)^2}{4}\norms{w^k}^2 - \frac{b_0}{\eta}\norms{x^0 - x^{\star}}^2.
\end{array}
\end{equation*}
Combining the last two estimates, we can easily show that $\frac{b_0\eta (k + 1)^2}{4}\norms{w^k}^2 - \frac{b_0}{\eta}\norms{x^0 - x^{\star}}^2 \leq \Vc_k \leq \Vc_0 = a_0\norms{w^0}^2 = \frac{\eta b_0}{2}\norms{w^0}^2$, leading to \eqref{eq:EAG4MVIP_convergence1}.
\end{proof}
%%%% End of proof.

%%%% Remark 1.
\begin{remark}\label{re:EAG4NI}
Our analysis in Theorem~\ref{th:EAG4MVIP_convergence} essentially relies on the proof technique in \cite{yoon2021accelerated}, and it is also different from \cite{cai2022accelerated}.
We believe that our proof is rather elementary and using simple arguments.
We note that our analysis can also be extended to prove the convergence of the past-extra-anchored gradient method (i.e., replacing $Fx^k$ in \eqref{eq:EAG4NI} by $Fy^{k-1}$) by using similar arguments as in \cite{tran2021halpern}.
We omit the details here.
\end{remark}

\beforesubsec
\subsection{The fast extragradient method for \eqref{eq:NI}}\label{subsec:EAG2}
\aftersubsec
\noindent\textbf{The algorithm.}
The fast extragradient method (FEG) for solving \eqref{eq:NI} developed in \cite{cai2022baccelerated,lee2021fast,tran2023extragradient,yoon2021accelerated} can be written in a unified form as follows.
Starting from $x^0\in\dom{\Phi}$, at each iteration $k \geq 0$, we update
\begin{equation}\label{eq:EAG4NI}
\arraycolsep=0.2em
\left\{\begin{array}{lcl}
y^k        &:= &  x^k + \tau_k(x^0 - x^k) -  ( \hat{\eta}_k - \beta_k) (Fx^k + \xi^k), \vspace{1ex}\\
x^{k+1} &:= & x^k + \tau_k(x^0 - x^k) - \eta_k (Fy^k + \xi^{k+1}) + \beta_k(Fx^k + \xi^k),
\end{array}\right.
\tag{FEG2}
\end{equation}
where $\xi^k \in Tx^k$, $\tau_k \in (0, 1)$, $\beta_k \geq 0$, $\eta_k  > 0$, and $\hat{\eta}_k > 0$ are given, determined later.
\begin{itemize}
\itemsep=0.2em
\item If $T = 0$, $\beta_k := 0$, and $\eta_k = \hat{\eta}_k$, then \eqref{eq:EAG4NI} reduces to the extra-anchored gradient (EAG) scheme for solving \eqref{eq:NE} in \cite{yoon2021accelerated} under the monotonicity of $F$ as
\begin{equation}\label{eq:EAG4NE}
y^k := x^k - \eta_kFx^k \quad\text{and}\quad x^{k+1} = x^k - \eta_kFy^k.
\tag{EAG}
\end{equation}
\item If $T = 0$, $\beta_k := 2\rho(1-\tau_k)$, and $\eta_k := \eta > 0$, then \eqref{eq:EAG4NI} reduces to the fast extragradient variant for solving \eqref{eq:NE} in \cite{lee2021fast}, but under the co-hypomonotonicity of $F$.

\item If $T$ is a maximally monotone operator (e.g., $T := \Nc_{\Xc}$, the normal cone of a nonempty, closed, and convex set $\Xc$), $\beta_k := 2\rho(1-\tau_k)$ and $\eta_k := \eta > 0$, then  \eqref{eq:EAG4NI} is exactly the variant studied in \cite{cai2022accelerated}.
\end{itemize}
In fact, \eqref{eq:EAG4NI} is rooted from Tseng's forward-backward-forward splitting method \eqref{eq:FBFS4NI} instead of \eqref{eq:EG4NI} because it only requires one resolvent evaluation $J_{\eta T}$ per iteration.
Recently, \cite{tran2023extragradient} provides an elementary convergence analysis for \eqref{eq:EAG4NI}, which relies on the technique in \cite{yoon2021accelerated}.
We survey this method here and present the convergence analysis from  \cite{yoon2021accelerated}.

Let $w^k$ and $\hat{w}^k$ be defined as \eqref{eq:EAG4MVIP_ex2}.
Then, we can equivalently rewrite \eqref{eq:EAG4NI}   as follows:
\begin{equation}\label{eq:EAG4NI_ex3}
\arraycolsep=0.2em
\left\{\begin{array}{lcl}
y^k        &:= &  x^k + \tau_k(x^0 - x^k) -  (\hat{\eta}_k - \beta_k) w^k, \vspace{1ex}\\
x^{k+1} &:= & x^k + \tau_k(x^0 - x^k) - \eta \hat{w}^{k+1} + \beta_k w^k.
\end{array}\right.
\end{equation}
Clearly, \eqref{eq:EAG4NI_ex3} has the same form as the fast extragradient scheme in \cite{lee2021fast} for  solving \eqref{eq:NE}, where $w^k$ and $\hat{w}^{k+1}$ reduce to $Fx^k$ and $Fy^k$, respectively.
Since $x^{k+1}$ are in both sides of line 2 of \eqref{eq:EAG4NI_ex3}, we can rewrite \eqref{eq:EAG4NI} as
\begin{equation}\label{eq:EAG4NI_impl}
\arraycolsep=0.2em
\left\{\begin{array}{lcl}
y^k        &:= &  x^k + \tau_k(x^0 - x^k) - ( \hat{\eta}_k - \beta_k) (Fx^k + \xi^k), \vspace{1ex}\\
x^{k+1}  & \in & J_{\eta T}\left( y^k - \eta Fy^k  +  \hat{\eta}_k(Fx^k + \xi^k)  \right), \vspace{1ex}\\
\xi^{k+1} & := &  \frac{1}{\eta}\left( y^k - \eta Fy^k  +  \hat{\eta}_k(Fx^k + \xi^k)  - x^{k+1} \right),
\end{array}\right.
\end{equation}
where $\xi^0 \in Tx^0$ is arbitrary, and $J_{\eta T}$ is the resolvent of $\eta T$, which may not be single-valued in our case.
However, for our iterates to be well-defined, we will assume that $\ran{J_{\eta T}} \subseteq\dom{F} = \R^p$, and $\dom{J_{\eta T}} = \R^p$.

\vspace{0.75ex}
\noindent\textbf{Convergence analysis.}
Using the same potential function $\Vc_k$ as in \eqref{eq:EAG4MVIP_potential_func}, we can prove the convergence of  \eqref{eq:EAG4NI} in the following theorem.

%%% Theorem 2.4.
\begin{theorem}\label{th:cEAG_convergence}
Assume that $\Phi$ in \eqref{eq:NI} is $\rho$-co-hypomonotone, $F$ is $L$-Lipschitz continuous such that $2L\rho < 1$, $\zer{\Phi}\neq\emptyset$, $\ran{J_{\eta T}} \subseteq\dom{F} = \R^p$, and $\dom{J_{\eta T}} = \R^p$.
Let $\sets{(x^k, y^k)}$ be generated by \eqref{eq:EAG4NI} using
\begin{equation}\label{eq:EAG4NI_param_update}
\tau_k := \frac{1}{k+2}, \quad \beta_k := 2\rho(1-\tau_k), \quad \eta_k :=  \eta \in \left(2\rho,  \frac{1}{L}\right], \quad\text{and} \quad  \hat{\eta}_k := (1-\tau_k)\eta.
\end{equation}
Then, for all $k\geq 0$ and any $x^{\star}\in\zer{\Phi}$, we have
\begin{equation}\label{eq:EAG4NI_convergence1}
\norms{Fx^k + \xi^k}^2 \leq  \frac{4\norms{x^0 - x^{\star}}^2 + 2\eta(\eta-2\rho)\norms{Fx^0 + \xi^0}^2}{(\eta - 2\rho)^2(k + 1)^2}, \quad\text{where}\quad \xi^k \in Tx^k.
\end{equation}
Consequently, we have the last-iterate convergence rate $\BigO{1/k}$ of the residual norm $\norms{Fx^k + \xi^k}$.
\end{theorem}

%%% The proof of Theorem 2.4.
\begin{proof}
Since \eqref{eq:EAG4NI} is equivalent to \eqref{eq:EAG4NI_ex3}, from the second line of \eqref{eq:EAG4NI}, we can easily show that
\begin{equation}\label{eq:EAG4NI_proof1}
\arraycolsep=0.2em
\left\{\begin{array}{lcl}
x^{k+1} - x^k & = &  - \tau_k(x^k - x^0) - \eta \hat{w}^{k+1} + \beta_kw^k \vspace{1ex}\\
x^{k+1} - x^k & = & -\tfrac{\tau_k}{1-\tau_k}(x^{k+1} - x^0) - \tfrac{\eta}{1-\tau_k}\hat{w}^{k+1} + \frac{\beta_k}{1-\tau_k}w^k.
\end{array}\right.
\end{equation}
Next, since  $\Phi$ is $\rho$-co-hypomonotone and $w^k \in \Phi x^k = Fx^k + Tx^k$, we have $\iprods{w^{k+1} - w^k, x^{k+1} - x^k} + \rho\norms{w^{k+1} - w^k}^2 \geq 0$.
This relation together with $\eta_k := \eta \in (2\rho, \frac{1}{L}]$ and $\beta_k := 2\rho(1-\tau_k)$ lead to
\begin{equation*} 
\arraycolsep=0.2em
\begin{array}{lcl}
0 &\leq & \iprods{w^{k+1}, x^{k+1} - x^k} - \iprods{w^k, x^{k+1} - x^k}  + \rho\norms{w^{k+1} - w^k}^2 \vspace{1ex}\\
&\overset{\tiny\eqref{eq:EAG4NI_proof1}}{=} & \tau_k\iprods{w^k, x^k - x^0} - \frac{\tau_k}{1-\tau_k}\iprods{w^{k+1}, x^{k+1} - x^0} + \eta \iprods{w^k, \hat{w}^{k+1}} - 2\rho(1-\tau_k)\norms{w^k}^2 \vspace{1ex}\\
&& -  {~}  \frac{\eta}{1-\tau_k}\iprods{w^{k+1}, \hat{w}^{k+1}} + 2\rho\iprods{w^{k+1}, w^k} + \rho\norms{w^{k+1} - w^k}^2.
\end{array}
\end{equation*}
Multiplying this expression by $\frac{b_k}{\tau_k}$, rearranging the result, and using $b_{k+1} = \frac{b_k}{1-\tau_k}$, we obtain
\begin{equation*} 
\arraycolsep=0.2em
\begin{array}{lcl}
\Tc_{[1]} &:= & b_k\iprods{w^k, x^k - x^0} - b_{k+1}\iprods{w^{k+1}, x^{k+1} - x^0} \vspace{1ex} \\ 
&\geq & \frac{\eta b_{k+1}}{\tau_k}\iprods{w^{k+1} - w^k, \hat{w}^{k+1}} + \eta b_{k+1} \iprods{\hat{w}^{k+1}, w^k} - \frac{\rho b_k}{\tau_k}\norms{w^{k+1}}^2 + \frac{\rho b_k(1 - 2\tau_k)}{\tau_k}\norms{w^k}^2.
\end{array}
\end{equation*}
Adding $a_k\norms{w^k}^2 - a_{k+1}\norms{w^{k+1}}^2$ to $\Tc_{[1]}$ and  using $\Vc_k$ from \eqref{eq:EAG4MVIP_potential_func}, we can show that
\begin{equation}\label{eq:EAG4NI_proof2}
\hspace{-3ex}
\arraycolsep=0.2em
\begin{array}{lcl}
\Vc_k - \Vc_{k+1} &= & a_k\norms{w^k}^2 - a_{k+1}\norms{w^{k+1}}^2 + b_k\iprods{w^k, x^k - x^0} - b_{k+1}\iprods{w^{k+1}, x^{k+1} - x^0} \vspace{1ex}\\
&\geq &  \big(a_k + \frac{\rho b_k(1 - 2\tau_k)}{\tau_k} \big) \norms{w^k}^2 - \big(a_{k+1} + \frac{\rho b_k}{\tau_k}\big)\norms{w^{k+1}}^2  \vspace{1ex}\\
&& +  {~}  \frac{\eta b_{k+1} }{\tau_k}\iprods{w^{k+1} - w^k, \hat{w}^{k+1}} + \eta b_{k+1} \iprods{\hat{w}^{k+1}, w^k}.
\end{array}
\hspace{-5ex}
\end{equation}
Now, from \eqref{eq:EAG4NI_ex3}, we have $x^{k+1} - y^k = -\eta \hat{w}^{k+1} + \hat{\eta}_kw^k$.
By the $L$-Lipschitz continuity of $F$, we have $\norms{w^{k+1} - \hat{w}^{k+1}}^2 = \norms{Fx^{k+1} - Fy^k}^2 \leq L^2\norms{x^{k+1} - y^k}^2 = L^2\norms{\eta\hat{w}^{k+1} - \hat{\eta}_kw^k}^2$.
Expanding this inequality, and rearranging the result, we obtain
\begin{equation*} 
\arraycolsep=0.2em
\begin{array}{lcl}
0 & \geq & \norms{w^{k+1}}^2 + (1 - L^2\eta^2)\norms{\hat{w}^{k+1}}^2 - 2\iprods{w^{k+1} - w^k, \hat{w}^{k+1}} -  2\big(1 - L^2\eta\hat{\eta}_k)\iprods{\hat{w}^{k+1}, w^k} - L^2\hat{\eta}_k^2\norms{w^k}^2.
\end{array}
\end{equation*}
Multiplying this estimate by $\frac{\eta b_{k+1} }{2\tau_k}$ and adding the result to \eqref{eq:EAG4NI_proof2}, we eventually arrive at
\begin{equation}\label{eq:EAG4NI_proof3}
\hspace{-4ex}
\arraycolsep=0.2em
\begin{array}{lcl}
\Vc_k - \Vc_{k+1} &\geq & \left(a_k  - \frac{L^2\eta \hat{\eta}_k^2b_{k+1} - 2\rho b_k(1-2\tau_k)}{2\tau_k} \right)\norms{w^k}^2 + \left(\frac{\eta b_{k+1} - 2\rho b_k}{2\tau_k}  - a_{k+1} \right) \norms{w^{k+1}}^2  \vspace{2ex}\\
&& +  {~}  \frac{\eta(1 - L^2\eta^2)b_{k+1}}{2\tau_k}\norms{\hat{w}^{k+1}}^2  -  \frac{\eta(1 - \tau_k - L^2\eta\hat{\eta}_k) b_{k+1} }{\tau_k} \iprods{\hat{w}^{k+1}, w^k}.
\end{array}
\hspace{-7ex}
\end{equation}
Let us choose $\tau_k := \frac{1}{k + 2}$ and $\hat{\eta}_k := (1-\tau_k)\eta$ as in \eqref{eq:EAG4NI_param_update}, and $a_{k+1} := \frac{b_{k+1}[ \eta  - 2\rho(1-\tau_k)]}{2\tau_k} = \frac{[(\eta - 2\rho)(k + 2) + 2\rho] b_{k+1}}{2}$.
Since $b_{k+1} = \frac{b_k}{1-\tau_k} $, we have $b_k = b_0(k+1)$ and hence $a_k = \frac{b_0[(\eta - 2\rho)(k+1) + 2\rho](k+1) }{2}$.

Using the above choice of parameters and noting that $L\eta \leq 1$, we can simplify \eqref{eq:EAG4NI_proof3} as
\begin{equation}\label{eq:EAG4NI_proof3_b}
\arraycolsep=0.2em
\begin{array}{lcl}
\Vc_k - \Vc_{k+1} &\geq & \frac{\eta(1 - L^2\eta^2)b_{k+1}}{2\tau_k}\norms{\hat{w}^{k+1} - (1-\tau_k)w^k}^2 \geq 0.
\end{array}
\end{equation}
Finally, since $x^{\star}\in\zer{\Phi}$, we have $\iprods{w^k, x^k - x^{\star}} \geq -\rho\norms{w^k}^2$.
Using this bound and \eqref{eq:EAG4MVIP_potential_func}, we can show that
\begin{equation*}
\arraycolsep=0.2em
\begin{array}{lcllcl}
\Vc_k & = & a_k\norms{w^k}^2 + b_k\iprods{w^k, x^{\star} - x^0} + b_k\iprods{w^k, x^k - x^{\star}}  & \geq &  ( a_k - \rho b_k) \norms{w^k}^2 - b_k\norms{w^k}\norms{x^0 - x^{\star}} \vspace{1ex}\\
&\geq &  \left( a_k - \rho b_k - \frac{(\eta - 2\rho) b_k^2}{4b_0} \right) \norms{w^k}^2 - \frac{b_0}{\eta - 2\rho}\norms{x^0 - x^{\star}}^2  & = &  \frac{b_0(\eta - 2\rho)(k + 1)^2}{4}\norms{w^k}^2 - \frac{b_0}{\eta- 2\rho}\norms{x^0 - x^{\star}}^2.
\end{array}
\end{equation*}
Combining this inequality and \eqref{eq:EAG4NI_proof3_b}, we get $\frac{b_0(\eta - 2\rho)(k + 1)^2}{4}\norms{w^k}^2 - \frac{b_0}{\eta- 2\rho}\norms{x^0 - x^{\star}}^2 \leq \Vc_k \leq \Vc_0 = a_0\norms{w^0}^2 + b_0\iprods{w^0, x^0 - x^0} = \frac{b_0\eta}{2}\norms{w^0}^2$.
This bound leads to \eqref{eq:EAG4NI_convergence1}.
\end{proof}
%%% End of the proof.

%%%%%%%%%%%%%%%%%%%%%%%%%%%%%%%%%%%%%%%%%%%%
%%% 3.2. The past- extra-anchored gradient method.
%%%%%%%%%%%%%%%%%%%%%%%%%%%%%%%%%%%%%%%%%%%%
\beforesubsec
\subsection{The  past extra-anchored gradient method for \eqref{eq:NI}}\label{subsec:PEAG2}
\aftersubsec
Both \eqref{eq:EAG4MVIP} and \eqref{eq:EAG4NI} require two evaluations of $F$ per iteration.
In addition,  \eqref{eq:EAG4MVIP} needs two evaluations of $J_{\eta T}$.
To reduce this computation, we can apply Halpern fixed-point iteration to  the past extragradient method \cite{popov1980modification} as done in  \cite{tran2021halpern} for \eqref{eq:NE}.
Our recent work  \cite{yoon2021accelerated} has extended   \cite{tran2021halpern} to solve \eqref{eq:NI} and relaxed assumption from the monotonicity to the co-hypomonotonicity of $F$.
We now survey the results from \cite{cai2022baccelerated,tran2023extragradient} in this subsection. 

%%% The algorithm.
\vspace{0.5ex}
\noindent\textbf{The algorithm.}
Starting from $x^0\in\dom{\Phi}$, we set $y^{-1} := x^0$, and at each iteration $k \geq 0$, we update
\begin{equation}\label{eq:PEAG4NI}
\arraycolsep=0.2em
\left\{\begin{array}{lcl}
y^k        &:= &  x^k + \tau_k(x^0 - x^k) - ( \hat{\eta}_k - \beta_k) (Fy^{k-1} + \xi^k), \vspace{1ex}\\
x^{k+1} &:= & x^k + \tau_k(x^0 - x^k) - \eta (Fy^k + \xi^{k+1}) + \beta_k(Fy^{k-1} + \xi^k), 
\end{array}\right.
\tag{PEAG2}
\end{equation}
where $\xi^k \in Tx^k$, $\tau_k \in (0, 1)$, $\eta > 0$, $\hat{\eta}_k > 0$, and $\beta_k > 0$ are given parameters, determined later.
Clearly, if $T = 0$, then \eqref{eq:PEAG4NI} reduces to the past extra-anchored gradient scheme in \cite{tran2021halpern}.
This scheme can be considered as a modification of \eqref{eq:EAG4NI} by replacing $Fx^k$ by $Fy^{k-1}$ using Popov's trick in \cite{popov1980modification}.

Again, we reuse  $w^k := Fx^k + \xi^k \in Fx^k + Tx^k$ and $\hat{w}^k := Fy^{k-1} + \xi^k$ as in \eqref{eq:EAG4NI}.
Then, \eqref{eq:PEAG4NI} can be rewritten equivalently as 
\begin{equation}\label{eq:PEAG4NI_ex3}
\arraycolsep=0.2em
\left\{\begin{array}{lcl}
y^k        &:= &  x^k + \tau_k(x^0 - x^k) - ( \hat{\eta}_k  - \beta_k) \hat{w}^k, \vspace{1ex}\\
x^{k+1} &:= & x^k + \tau_k(x^0 - x^k) - \eta \hat{w}^{k+1} + \beta_k \hat{w}^k.
\end{array}\right.
\end{equation}
Obviously, \eqref{eq:PEAG4NI_ex3} appears to have the same form as the past extra-anchored gradient scheme in  \cite{tran2021halpern}, but with different parameters.
Since $\xi^{k+1} \in Tx^{k+1}$, we can rewrite \eqref{eq:PEAG4NI} as follows:
\begin{equation}\label{eq:EAG4NI_impl}
\arraycolsep=0.2em
\left\{\begin{array}{lcl}
y^k        &:= &  x^k + \tau_k(x^0 - x^k) - ( \hat{\eta}_k  - \beta_k)\hat{w}^k, \vspace{1ex}\\
x^{k+1}  & \in & J_{\eta T}\left( y^k - \eta Fy^k  +  \hat{\eta}_k\hat{w}^k  \right), \vspace{1ex}\\
\hat{w}^{k+1} & := &  \frac{1}{\eta}\left( y^k +  \hat{\eta}_k\hat{w}^k  - x^{k+1} \right), \vspace{1ex}\\
\end{array}\right.
\end{equation}
where $x^0 \in \dom{\Phi}$ is given, $y^{-1}  := x^0$, and $\hat{w}^0 \in Fy^{-1} + Tx^0$ is arbitrary.

%%% Remark 2.
\begin{remark}\label{re:PEAG4NI_remark1}
Notice that the first accelerated variant of the past extragradient method \cite{popov1980modification} was proposed in \cite{tran2021halpern} to solve \eqref{eq:NE} under the monotonicity of $F$, that achieves $\BigO{1/k}$ rate on $\norms{Fx^k}$.
This method was then extended to \eqref{eq:NI} in \cite{cai2022baccelerated} for the $\rho$-co-hypomonotonicity of $\Phi$ such that  $2\sqrt{34} L\rho < 1$.
However, \cite{tran2023extragradient} provided an alternative proof for \eqref{eq:PEAG4NI} using a different choice of parameters and a more relaxed condition $2\sqrt{34} L\rho < 1$ than \cite{cai2022baccelerated}.
In addition, \cite{tran2023extragradient} does not require the maximal monotonicity of $T$  as \cite{cai2022baccelerated}, and the form \eqref{eq:EAG4NI_impl} appears to be different from \cite{cai2022baccelerated}, though they have the same per-iteration complexity with one evaluation of $F$ and one evaluation of $J_{\eta T}$ per iteration. 
\end{remark}

%%%% Convergence analysis
\vspace{0.75ex}
\noindent
\textbf{Convergence analysis.}
To establish convergence of \eqref{eq:PEAG4NI}, we use the following potential function:
\begin{equation}\label{eq:PEAG4NI_potential_func}
\hat{\Vc}_k :=   a_k\norms{Fx^k + \xi^k}^2 + b_k\iprods{Fx^k + \xi^k, x^k - x^0} + c_k\norms{Fx^k - Fy^{k-1}}^2,
\end{equation}
where $\xi^k \in Tx^k$, $a_k > 0$, $b_k > 0$, and $c_k > 0$ are given parameters.
Using $w^k := Fx^k + \xi^k$ and $\hat{w}^k := Fy^{k-1} + \xi^k$, we can rewrite $\hat{\Vc}_k = a_k\norms{w^k}^2 + b_k\iprods{w^k, x^k - x^0} + c_k\norms{w^k - \hat{w}^k}^2$.
Now, we are ready to state the convergence of \eqref{eq:PEAG4NI} in the following theorem.

%%% Theorem 2.4.
\begin{theorem}\label{th:cPEAG_convergence}
Assume that $\Phi$ in \eqref{eq:NI} is $\rho$-co-hypomonotone, $F$ is $L$-Lipschitz continuous such that $2\sqrt{34}L\rho < 1$, $\zer{\Phi}\neq\emptyset$, $\dom{J_{\eta T}} = \R^p$, and $\ran{J_{\eta T}} \subseteq\dom{F} = \R^p$.  
Let $\eta := \sqrt{\frac{2}{17L}}$ be a given stepsize, and $\sets{(x^k, y^k)}$ be generated by \eqref{eq:PEAG4NI} using the following parameters:
\begin{equation}\label{eq:PEAG4NI_param_update}
\tau_k := \frac{1}{k+2}, \quad \beta_k := \frac{4\rho(1-\tau_k)}{1 + \tau_k},  \quad\text{and} \quad  \hat{\eta}_k := (1-\tau_k)\eta.
\end{equation}
Then, for all $k\geq 0$ and any $x^{\star}\in\zer{\Phi}$, we have 
\begin{equation}\label{eq:PEAG4NI_convergence1}
\norms{Fx^k + \xi^k}^2  \leq    \frac{1}{(k+1)^2}\left[ \frac{4}{3(\eta - 4\rho)^2}\norms{x^0 - x^{\star}}^2 +  \frac{2(3\eta - 2\rho)}{9(\eta - 4\rho)}  \norms{Fx^0 + \xi^0}^2\right], \quad \xi^k \in Tx^k.
\end{equation}
Consequently, we have the last-iterate convergence rate $\BigO{1/k}$ of the residual norm $\norms{Fx^k + \xi^k}$.
\end{theorem}

%%% The proof of Theorem 2.4.
\begin{proof}
First, using the equivalent form \eqref{eq:PEAG4NI_ex3} of \eqref{eq:PEAG4NI}, we can show that
\begin{equation*} 
\hspace{-2ex}
\arraycolsep=0.2em
\left\{\begin{array}{lcl}
x^{k+1} - x^k & = &  -\tau_k(x^k - x^0) - \eta \hat{w}^{k+1} + \beta_kw^k + \beta_k(\hat{w}^k - w^k) \vspace{1ex}\\
x^{k+1} - x^k & = & -\tfrac{\tau_k}{1-\tau_k}(x^{k+1} - x^0) - \tfrac{\eta}{1-\tau_k}\hat{w}^{k+1}  + \frac{\beta_k}{1-\tau_k}w^k +  \frac{\beta_k}{1-\tau_k}(\hat{w}^k - w^k).
\end{array}\right.
\hspace{-2ex}
\end{equation*}
Second, since $\Phi$ is $\rho$-co-hypomonotone, we have $\iprods{w^{k+1} - w^k, x^{k+1} - x^k} + \rho\norms{w^{k+1} - w^k}^2 \geq 0$.
Combining these expressions together, we can derive that
\begin{equation*} 
\arraycolsep=0.2em
\begin{array}{lcl}
\Tc_{[1]} &:= & \tau_k\iprods{w^k, x^k - x^0} - \frac{\tau_k}{1-\tau_k}\iprods{w^{k+1}, x^{k+1} - x^0} \vspace{1ex}\\
&\geq &  \frac{\eta}{1-\tau_k}\iprods{w^{k+1}, \hat{w}^{k+1}}   - \eta \iprods{w^k, \hat{w}^{k+1}}  - \rho\norms{w^{k+1} - w^k}^2 + \beta_k\norms{w^k}^2  \vspace{1ex}\\
&& - {~}   \frac{\beta_k}{1-\tau_k}\iprods{w^{k+1}, w^k} - \frac{\beta_k}{1-\tau_k}\iprods{w^{k+1} - (1-\tau_k)w^k, \hat{w}^k - w^k}.
\end{array}
\end{equation*}
Next, by Young's inequality and assuming that $\beta_k := \frac{4\rho(1-\tau_k)}{1 + \tau_k}$, we can further expand $\Tc_{[1]}$ as
\begin{equation*} 
\arraycolsep=0.2em
\begin{array}{lcl}
\Tc_{[1]} &:= & \tau_k\iprods{w^k, x^k - x^0} - \frac{\tau_k}{1-\tau_k}\iprods{w^{k+1}, x^{k+1} - x^0} \vspace{1ex}\\
&\geq&  \frac{\eta}{1-\tau_k}\iprods{w^{k+1}, \hat{w}^{k+1}} - \eta \iprods{w^k, \hat{w}^{k+1}} -  \rho\norms{w^{k+1} - w^k}^2 + \beta_k\norms{w^k}^2 -  \frac{\beta_k}{1-\tau_k}\iprods{w^{k+1}, w^k}   \vspace{1ex}\\
&& -  {~}    \frac{\beta_k}{4(1-\tau_k)}\norms{w^{k+1} - (1-\tau_k)w^k}^2 - \frac{\beta_k}{1-\tau_k} \norms{\hat{w}^k - w^k}^2 \vspace{1ex}\\
%& = & \frac{\eta}{1-\tau_k}\iprods{w^{k+1}, \hat{w}^{k+1}} - \eta \iprods{w^k, \hat{w}^{k+1}} -  \left[ \frac{(1+\tau_k)\beta_k}{2(1-\tau_k)} - 2\rho \right] \iprods{w^{k+1}, w^k} \vspace{1ex}\\
%&& - {~} \left[  \frac{\beta_k}{4(1-\tau_k)} + \rho \right] \norms{w^{k+1}}^2 + \left[ \beta_k -  \frac{\beta_k(1-\tau_k)}{4} - \rho \right] \norms{w^k}^2   - \frac{\beta_k}{1-\tau_k} \norms{\hat{w}^k - w^k}^2  \vspace{1ex}\\
&= &   \frac{\eta }{1-\tau_k}\iprods{w^{k+1}, \hat{w}^{k+1}}  - \eta \iprods{w^k, \hat{w}^{k+1}} - \frac{\rho(2 + \tau_k)}{1+\tau_k} \norms{w^{k+1}}^2  + \frac{\rho (2 -  3\tau_k - \tau_k^2) }{1 + \tau_k} \norms{w^k}^2 -   \frac{4\rho}{1 + \tau_k} \norms{\hat{w}^k - w^k}^2.
\end{array}
\end{equation*}
Multiplying $\Tc_{[1]}$ by $\frac{b_k}{\tau_k}$ and assuming $b_{k+1} = \frac{b_k}{1-\tau_k}$, then utilizing $\hat{\Vc}_k$ from \eqref{eq:PEAG4NI_potential_func}, we can show that
%\begin{equation*} 
%\arraycolsep=0.2em
%\begin{array}{lcl}
%\hat{\Tc}_{[1]} &:= & b_k\iprods{w^k, x^k - x^0} - b_{k+1}\iprods{w^{k+1}, x^{k+1} - x^0} \vspace{1ex}\\  
%&\geq & \frac{b_{k+1}\eta}{\tau_k}\iprods{w^{k+1} - w^k, \hat{w}^{k+1}}  + b_{k+1}\eta\iprods{w^k, \hat{w}^{k+1}}   - \frac{\rho b_k(2 + \tau_k)}{\tau_k(1 + \tau_k)} \norms{w^{k+1}}^2 \vspace{1ex}\\
%&& + {~}  \frac{\rho b_k ( 2 - 3\tau_k - \tau_k^2 ) }{\tau_k(1+\tau_k)} \norms{w^k}^2 -  \frac{4 \rho b_k}{\tau_k(1+\tau_k)} \norms{\hat{w}^k - w^k}^2.
%\end{array}
%\end{equation*}
%Adding $a_k\norms{w^k}^2 - a_{k+1}\norms{w^{k+1}}^2 + c_k\norms{w^k - \hat{w}^k}^2 - c_{k+1}\norms{w^{k+1} - \hat{w}^{k+1}}^2$ to both sides of $\hat{\Tc}_{[1]}$, then  using $\hat{\Vc}_k$ defined by \eqref{eq:PEAG4NI_potential_func}, we can show that
\begin{equation}\label{eq:PEAG4NI_proof2}
\arraycolsep=0.2em
\begin{array}{lcl}
\hat{\Vc}_k - \hat{\Vc}_{k+1} & = & a_k\norms{w^k}^2 - a_{k+1}\norms{w^{k+1}}^2 + b_k\iprods{w^k, x^k - x^0} - b_{k+1}\iprods{w^{k+1}, x^{k+1} - x^0} \vspace{1ex}\\
&& + {~} c_k\norms{w^k - \hat{w}^k}^2 - c_{k+1}\norms{w^{k+1} - \hat{w}^{k+1}}^2 \vspace{1ex}\\
&\geq &  \left[ a_k +  \frac{\rho b_k ( 2 - 3 \tau_k - \tau_k^2 ) }{\tau_k(1 + \tau_k)}  \right] \norms{w^k}^2 - \left[ a_{k+1} +  \frac{\rho b_k(2 + \tau_k)}{\tau_k(1 + \tau_k)}  \right]\norms{w^{k+1}}^2  \vspace{1ex}\\
&& + \left[c_k -  \frac{4 \rho b_k}{\tau_k(1 + \tau_k)}  \right]\norms{w^k - \hat{w}^k}^2 - c_{k+1}\norms{w^{k+1} - \hat{w}^{k+1}}^2\vspace{1ex}\\
&& +  {~}  \frac{\eta b_{k+1} }{\tau_k}  \iprods{w^{k+1} - w^k, \hat{w}^{k+1}} + \eta b_{k+1} \iprods{\hat{w}^{k+1}, w^k}.
\end{array}
\end{equation}
Now, since  $x^{k+1} - y^k = -\eta\hat{w}^{k+1} + \hat{\eta}_k\hat{w}^k$ from \eqref{eq:PEAG4NI_ex3}, utilizing the $L$-Lipschitz continuity of $F$ and Young's inequality, we can prove that
\begin{equation*} 
\arraycolsep=0.2em
\begin{array}{lcl}
\norms{w^{k+1} - \hat{w}^{k+1}}^2 & = & \norms{Fx^{k+1} - Fy^k}^2 \leq L^2\norms{x^{k+1} - y^k}^2 = L^2\norms{\eta \hat{w}^{k+1} - \hat{\eta}_k\hat{w}^k}^2 \vspace{1ex}\\
&\leq & 2L^2\norms{\eta\hat{w}^{k+1} - \hat{\eta}_kw^k}^2 + 2L\hat{\eta}_k^2\norms{w^k - \hat{w}^k}^2.
\end{array}
\end{equation*}
Multiplying this expression by $(1+\omega)$ for $\omega > 0$, using $M := 2(1 + \omega)L^2$, and expanding the result, we get
\begin{equation*} 
\arraycolsep=0.2em
\begin{array}{lcl}
0 & \geq & \omega\norms{w^{k+1} - \hat{w}^{k+1}}^2 + \norms{w^{k+1}}^2 + (1 - M\eta^2)\norms{\hat{w}^{k+1}}^2 - 2 \iprods{w^{k+1} - w^k, \hat{w}^{k+1}} \vspace{1ex}\\
&& - {~} 2(1 - M\eta\hat{\eta}_k)\iprods{\hat{w}^{k+1}, w^k} - M\hat{\eta}_k^2\norms{w^k}^2 - M\hat{\eta}_k^2\norms{w^k - \hat{w}^k}^2.
\end{array}
\end{equation*}
Further multiplying this inequality by $\frac{\eta b_{k+1} }{2\tau_k}$ and adding the result to \eqref{eq:PEAG4NI_proof2}, we obtain
\begin{equation}\label{eq:PEAG4NI_proof3}
\arraycolsep=0.2em
\begin{array}{lcl}
\hat{\Vc}_k - \hat{\Vc}_{k+1} &\geq & \left[ c_k - \frac{4\rho b_k}{\tau_k(1 + \tau_k)}  - \frac{Mb_{k+1}\eta\hat{\eta}_k^2}{2\tau_k} \right]\norms{w^k - \hat{w}^k}^2 +  \left( \frac{\omega\eta b_{k+1}}{2\tau_k} -  c_{k+1} \right) \norms{w^{k+1} - \hat{w}^{k+1}}^2\vspace{1ex}\\
&& + {~} \left[ a_k + \frac{\rho b_k ( 2  - 3 \tau_k - \tau_k^2 ) }{\tau_k(1 + \tau_k)}  - \frac{Mb_{k+1}\eta\hat{\eta}_k^2}{2\tau_k} \right] \norms{w^k}^2 +  \left[ \frac{\eta b_{k+1} }{2\tau_k} - \frac{\rho b_k(2 + \tau_k)}{\tau_k(1 + \tau_k)}  - a_{k+1} \right]\norms{w^{k+1}}^2 \vspace{1ex}\\
&& + {~}  \frac{\eta(1-M\eta^2)b_{k+1}}{2\tau_k}\norms{\hat{w}^{k+1}}^2  -  \frac{\eta\left(1 -  \tau_k - M\eta\hat{\eta}_k\right)b_{k+1}}{\tau_k}  \iprods{\hat{w}^{k+1}, w^k}.
\end{array}
\end{equation}
Since $\tau_k := \frac{1}{k+2}$ and $\hat{\eta}_k := (1 - \tau_k)\eta$ due to \eqref{eq:PEAG4NI_param_update}, if we choose $a_k := \frac{b_k}{2}\left(\eta(k+1) -  4\rho k + \frac{2\rho(k-1)}{k+3} \right)$, $b_k$ such that $b_{k+1}(1-\tau_k) = b_k$,  and $c_k := \frac{b_k}{2}\left( M\eta^3(k+1) +  \frac{8\rho(k+2)^2}{k+3} \right)$, then \eqref{eq:PEAG4NI_proof3} leads to
\begin{equation}\label{eq:PEAG4NI_proof3_b}
\hspace{-0.75ex}
\arraycolsep=0.2em
\begin{array}{lcl}
\hat{\Vc}_k - \hat{\Vc}_{k+1} &\geq & \frac{\eta(1-M\eta^2)b_{k+1}}{2\tau_k}\norms{\hat{w}^{k+1} - (1-\tau_k)w^k}^2 + \frac{2\rho b_{k+1}}{2(k+4)}\norms{w^{k+1}}^2 \vspace{1ex}\\
&& + {~} \frac{b_{k+1}(k+2)}{2}\left( \omega\eta -   M\eta^3 - \frac{8\rho(k+3)^2}{(k+2)(k+4)}  \right) \norms{w^{k+1} - \hat{w}^{k+1}}^2. 
\end{array}
\hspace{-6ex}
\end{equation}
If $\omega > 0$ satisfies $\omega \eta \geq M\eta^3 + \frac{8\rho(k+3)^2}{(k+2)(k+4)}$ and $1 - M\eta^2 \geq 0$, then \eqref{eq:PEAG4NI_proof3_b} implies that $\hat{\Vc}_{k+1} \leq \hat{\Vc}_k$ for all $k\geq 0$.
The first condition holds if $\omega \eta \geq M\eta^3 + 9\rho$.
However, to assure that $a_k > 0$, we require $\eta > 4\rho$.
Overall, if 
\begin{equation}\label{eq:PEAG4NI_proof4}
2(1+\omega)L^2\eta^2 \leq 1, \quad  \omega \eta \geq 2(1+\omega)L^2\eta^3 + 9\rho, \quad \text{and} \quad \eta > 4\rho,
\end{equation}
then \eqref{eq:PEAG4NI_proof3_b} leads to $\hat{\Vc}_{k+1} \leq \hat{\Vc}_k$ for all $k\geq 0$.

For simplicity, we choose $\omega := \frac{13}{4}$ and $\eta := \frac{1}{L\sqrt{2(1+\omega)}} = \frac{\sqrt{2}}{\sqrt{17}L}$.
Then, the last two conditions of \eqref{eq:PEAG4NI_proof4} hold if $L\rho \leq \frac{\omega - 1}{9\sqrt{2(1+\omega)}}  = \frac{1}{2\sqrt{34}}$ and $L\rho < \frac{1}{4\sqrt{2(1+\omega)}} = \frac{1}{2\sqrt{34}}$, respectively.
Thus  if $2\sqrt{34}L\rho < 1$, then \eqref{eq:PEAG4NI_proof4} holds.

Since  $\tau_k := \frac{1}{k+2}$, we get $b_k = b_0(k+1)$ for some $b_0 > 0$.
Moreover, since $x^{\star} \in \zer{\Phi}$, by the $\rho$-co-hypomonotonicity of $\Phi$, we have $\iprods{w^k, x^k - x^{\star}} \geq -\rho\norms{w^k}^2$.
Using this expression, \eqref{eq:PEAG4NI_potential_func}, Young's inequality, $b_k = b_0(k+1)$, and the choice of $a_k$, we can prove that
\begin{equation*}
\arraycolsep=0.2em
\begin{array}{lcl}
\hat{\Vc}_k %&\geq &  \left( a_k - b_k\rho - \frac{(\eta - 4\rho)b_k^2}{4b_0} \right) \norms{w^k}^2 - \frac{b_0}{\eta - 4\rho}\norms{x^0 - x^{\star}}^2 \vspace{1ex}\\
&\geq & \frac{3b_0(\eta - 4\rho)(k+1)^2}{4} \norms{w^k}^2 - \frac{b_0}{\eta - 4\rho}\norms{x^0 - x^{\star}}^2.
\end{array}
\end{equation*}
Finally, since $y^{-1} = x^0$, we have $\hat{\Vc}_0 = a_0\norms{w^0}^2 = b_0\left(\frac{\eta}{2} - \frac{\rho}{3}\right)\norms{Fx^0 + \xi^0}^2$, leading to $\hat{\Vc}_k \leq \hat{\Vc}_0 = b_0 \left(\frac{\eta}{2} - \frac{\rho}{3}\right)\norms{Fx^0 + \xi^0}^2$ from \eqref{eq:PEAG4NI_proof3_b}.
Combining this expression and the lower bound of $\hat{\Vc}_k$ above, we get
\begin{equation*}
\arraycolsep=0.2em
\begin{array}{lcl}
\norms{w^k}^2 & \leq &  \frac{1}{(k+1)^2}\left[ \frac{4}{3(\eta - 4\rho)^2}\norms{x^0 - x^{\star}}^2 +  \frac{2(3\eta - 2\rho)}{9(\eta - 4\rho)}  \norms{Fx^0 + \xi^0}^2\right],
\end{array}
\end{equation*}
which proves \eqref{eq:PEAG4NI_convergence1} by virtue of  $w^k := Fx^k + \xi^k \in Fx^k + Tx^k$.
\end{proof}
%%% End of the proof.

\begin{remark}\label{re:PEAG4NI_remark2}
Theorem~\ref{th:cPEAG_convergence} only provides one possibility for the stepsize $\eta$, which is $\eta := \frac{\sqrt{2}}{\sqrt{17}L}$.
However, one can revise our proof to provide a range of $\eta$ as in \eqref{eq:EAG4NI} or \eqref{eq:EAG4MVIP}.
\end{remark}

%%%%%%%%%%%%%%%%%%%%%%%%%%%%%%%%%%%%%%%%%
%%%%% 3. Nesterov's Accelerated Extragradient Method
%%%%%%%%%%%%%%%%%%%%%%%%%%%%%%%%%%%%%%%%%
\beforesec
\section{Nesterov's Accelerated Extragradient-Type Methods}\label{sec:AEG4NI}
\aftersec
\textbf{Introduction.}
Nesterov's accelerated method  \cite{Nesterov1983} is an outstanding achievement in convex optimization over the past few decades. 
While the method was invented in 1983, its popularity began with two pioneering works \cite{Nesterov2005c} and \cite{Beck2009}. 
Since then, such a technique has been widely studied and applied to many problems in various fields, including proximal-point, coordinate gradient, stochastic gradient, operator splitting,  conditional gradient, Newton-type, and high-order methods. 
Although the majority of literature on accelerated methods pertains to convex optimization, extensions to monotone equations and inclusions have recently become an interesting research topic. 
Early works in this direction were conducted by, e.g., \cite{attouch2020convergence,attouch2019convergence,bot2022fast,bot2022bfast,kim2021accelerated,mainge2021accelerated,mainge2021fast}. 
Unlike convex optimization, developing accelerated methods for monotone inclusions of the form \eqref{eq:NI} requires a fundamental change in forming an appropriate potential function. 
Existing works often rely on proximal-point methods, which were extended to accelerated schemes in \cite{guler1992new}. 
Another approach, such as in  \cite{gupta2022branch,kim2021accelerated}, is to apply the ``performance estimation problem'' technique developed in \cite{drori2014performance}. 
Nesterov's accelerated methods for different problem classes have been proven to be ``optimal'', meaning their upper bounds of convergence rates or complexity match the respective lower bounds in a certain sense, see, e.g., \cite{Nesterov2004,woodworth2016tight}.

Note that Nesterov's accelerated method can also be viewed as a discretization of an appropriate dynamical system \cite{Su2014}, as often seen in classical gradient methods. 
Exploring this perspective, several new variants and extensions have been extensively studied in the literature, see, e.g., \cite{attouch2016rate,bot2022fast,bot2022bfast,shi2021understanding,wibisono2016variational}.
In addition, connections between Nesterov's and other methods have also been discussed. 
For example, \cite{attouch2022ravine} shows that Nesterov's accelerated methods are equivalent to Ravine's methods, which were proposed in 1961. 
Recently, \cite{partkryu2022,tran2022connection} have shown the relations between Nesterov's accelerated schemes and Halpern fixed-point iterations in fixed-point theory \cite{halpern1967fixed}. 
These methods are indeed equivalent in certain settings. 
Exploiting this perspective, \cite{tran2022connection,tran2023extragradient} have developed several Nesterov's accelerated variants to solve \eqref{eq:NE} and \eqref{eq:NI}, including extragradient methods. 
In this section, we will survey recent results from these works.

%%% 8.1. Nesterov's accelerated extragradient method of (NI).
\beforesubsec
\subsection{Nesterov's accelerated extragradient method for \eqref{eq:NI}}\label{subsec:AEG}
\aftersubsec
\textbf{The algorithm.}
The first Nesterov's accelerated extragradient method proposed in \cite{tran2022connection,tran2023extragradient} can be written as follows.
Starting from $y^0 \in \dom{\Phi}$, set $z^0 := y^0$ and $w^{-1} := 0$, and at each iteration $k\geq 0$, we update
\begin{equation}\label{eq:AEG4NI}
\arraycolsep=0.2em
\left\{\begin{array}{lcl}
x^k & \in  & J_{\eta T}\left(y^k - \eta Fy^k + \hat{\eta}_k w^{k-1}\right), \vspace{1ex}\\
w^k & := &   \frac{1}{\eta}(y^k - x^k + \hat{\eta}_kw^{k-1}) - ( Fy^k - Fx^k), \vspace{1ex}\\
z^{k+1} &:= & x^k - \gamma w^k, \vspace{1ex}\\
y^{k+1} &:= & z^{k+1} + \theta_k(z^{k+1} - z^k)  +  \nu_k(y^k - z^{k+1}), 
\end{array}\right.
\tag{AEG}
\end{equation}
where $\eta$, $\hat{\eta}_k$, $\gamma$, $\theta_k$ and $\nu_k$ are given parameters, which will be determined later, and $J_{\eta T}$ is the resolvent of $\eta T$.

Now, for $\xi^k \in Tx^k$, we can easily show that  $w^k = Fx^k + \xi^k$.
If we additionally denote by $\hat{w}^k := Fy^k + \xi^k$ for $\xi^k \in Tx^k$, then by starting from $x^0 = z^0 = y^0$, we can rewrite \eqref{eq:AEG4NI} equivalently to 
\begin{equation}\label{eq:AEG4NI_v0}
\arraycolsep=0.2em
\left\{\begin{array}{lcl}
z^{k+1} &:= & x^k - \gamma w^k, \vspace{1ex}\\
y^{k+1} &:= & z^{k+1} + \theta_k(z^{k+1} - z^k) + \nu_k(y^k - z^{k+1}), \vspace{1ex}\\
x^{k+1} & := & y^{k+1} - \eta\hat{w}^{k+1} + \hat{\eta}_{k+1}w^k. 
\end{array}\right.
\end{equation}
Clearly, if $T = 0$, then \eqref{eq:AEG4NI_v0} exactly reduces to the one in \cite{tran2022connection}.
However, since we have not yet seen an obvious connection between \eqref{eq:AEG4NI} and \eqref{eq:EG4NI}, we can eliminate $z^k$ to obtain
\begin{equation}\label{eq:AEG4NI_v3}
\arraycolsep=0.2em
\left\{\begin{array}{lcl}
x^k & \in & J_{\eta T}\left(y^k - \eta Fy^k { \ + \ \hat{\eta}_k w^{k-1}} \right), \vspace{1ex}\\
y^{k+1} &:= & x^k - \beta_k(Fx^k - Fy^k)  { \ + \ \theta_k(x^k - x^{k-1}) } { \ + \ \hat{\beta}_k(y^k - x^k)  + \tilde{\beta}_kw^{k-1}}, \vspace{1ex}\\
w^k & := &   \frac{1}{\eta}(y^k - x^k + \hat{\eta}_kw^{k-1}) - ( Fy^k - Fx^k),
\end{array}\right.
\end{equation}
where $\beta_k :=  \gamma(1 + \theta_k - \nu_k)$, $\hat{\beta}_k := \nu_k - \frac{\beta_k}{\eta}$,  and $\tilde{\beta}_k := \gamma \theta_k - \frac{\beta_k\hat{\eta}_k}{\eta}$.
In this case \eqref{eq:AEG4NI_v3} can be viewed as an accelerated variant of \eqref{eq:EG4NI} with correction terms (see \cite{mainge2021accelerated}).

%%%% 8.1. Convergence analysis
\vspace{0.75ex}
\noindent\textbf{Convergence analysis.}
The main tool to establish convergence of \eqref{eq:AEG4NI} is the following potential function:
\begin{equation}\label{eq:AEG4NI_potential_func}
\Pc_k := a_k\norms{w^{k-1}}^2 + b_k\iprods{w^{k-1}, z^k - y^k} + \norms{z^k + t_k(y^k - z^k) - x^{\star}}^2,
\end{equation}
where $w^k := Fx^k + \xi^k$ for $\xi^k \in Tx^k$, $a_k > 0$, $b_k > 0$, and $t_k > 0$ are given parameters.

Now, we can establish a convergence rate of \eqref{eq:AEG4NI} in the following theorem.

%%%% Theorem 3.1.
\begin{theorem}\label{th:AEG4NI_convergence}
Suppose that $\Phi$ in \eqref{eq:NI} is $\rho$-co-hypomonotone, $F$ is $L$-Lipschitz continuous such that $2L\rho <  1$, $x^{\star} \in \zer{\Phi}$, $\dom{J_{\eta T}} = \R^p$, and $\ran{J_{\eta T}} \subseteq\dom{F}=\R^p$.
Let $\gamma > 0$ be such that $L(2\rho + \gamma) \leq 1$ $($e.g., $\gamma := \frac{1}{L} - 2\rho > 0$$)$ and $\sets{(x^k, y^k, z^k)}$ be generated by \eqref{eq:AEG4NI} using
\begin{equation}\label{eq:AEG4NI_para_choice}
t_k := k + 2, \quad \eta := \gamma + 2\rho,  \quad \hat{\eta}_k = \frac{(t_k-1)\eta}{t_k},  \quad  \theta_k := \frac{t_k-1}{t_{k+1}}, \quad \text{and} \quad \nu_k := \frac{t_k}{t_{k+1}}.
\end{equation}
Then, the following bound holds:
\begin{equation}\label{eq:AEG4NI_convergence}
\norms{Fx^k + \xi^k} \leq \frac{2\norms{y^0 - x^{\star}}}{ \gamma (k+2)}, \quad \text{where}\quad \xi^k \in Tx^k.
\end{equation}
Consequently, we have $\norms{Fx^k + \xi^k} = \BigO{\frac{1}{k}}$ showing $\BigO{\frac{1}{k}}$ convergence rate of \eqref{eq:AEG4NI}.
\end{theorem}

%%%% The proof of Lemma 3.1.
\begin{proof}
First, by inserting $(t_k-1)z^{k+1} - (t_k-1)z^{k+1}$, it is obvious to show that
\begin{equation*} 
\arraycolsep=0.2em
\begin{array}{lcl}
\norms{z^k + t_k(y^k - z^k) - x^{\star}}^2  &= & \norms{z^{k+1} - x^{\star}}^2 + (t_k - 1)^2\norms{z^{k+1} - z^k}^2 + t_k^2\norms{y^k - z^{k+1}}^2 \vspace{1ex}\\
&& + {~} 2(t_k -1)\iprods{z^{k+1} - z^k, z^{k+1} - x^{\star}} + 2t_k\iprods{y^k - z^{k+1}, z^{k+1} - x^{\star}} \vspace{1ex}\\
&& + {~} 2(t_k-1)t_k\iprods{y^k - z^{k+1}, z^{k+1} - z^k}.
\end{array} 
\end{equation*}
Next, using $y^{k+1} - z^{k+1} = \theta_k(z^{k+1} - z^k) + \nu_k(y^k - z^{k+1})$ from \eqref{eq:AEG4NI}, we have
\begin{equation*} 
\arraycolsep=0.2em
\begin{array}{lcl}
 \norms{z^{k+1} + t_{k+1}(y^{k+1} - z^{k+1}) -  x^{\star}}^2  &= & \norms{z^{k+1} - x^{\star}}^2 + t_{k+1}^2\theta_k^2\norms{z^{k+1} - z^k}^2 +  t_{k+1}^2\nu_k^2\norms{y^k  - z^{k+1} }^2 \vspace{1ex}\\
&& + {~} 2t_{k+1}\theta_k\iprods{z^{k+1} - z^k, z^{k+1} - x^{\star} } + 2t_{k+1}\nu_k\iprods{y^k - z^{k+1}, z^{k+1} - x^{\star}} \vspace{1ex}\\
&& + {~} 2t_{k+1}^2\nu_k\theta_k\iprods{y^k - z^{k+1}, z^{k+1} - z^k}.
\end{array} 
\end{equation*}
Combining the last two expressions, we can show that
\begin{equation*} 
\arraycolsep=0.2em
\begin{array}{lcl}
\Tc_{[1]}  &:= &   \norms{z^k + t_k(y^k - z^k) - x^{\star}}^2 - \norms{z^{k+1} + t_{k+1}(y^{k+1} - z^{k+1}) -  x^{\star}}^2 \vspace{1ex}\\
&= & \left[ (t_k-1)^2 - t_{k+1}^2\theta_k^2 \right] \norms{z^{k+1} - z^k }^2 + (t_k^2  - \nu_k^2t_{k+1}^2)\norms{y^k - z^{k+1} }^2 \vspace{1ex}\\
&& + {~} 2(t_k - 1 - t_{k+1}\theta_k)\iprods{z^{k+1} - z^k, z^{k+1} - x^{\star}} \vspace{1ex}\\
&& + {~} 2(t_k - t_{k+1}\nu_k)\iprods{y^k - z^{k+1}, z^{k+1} - x^{\star}} \vspace{1ex}\\
&& + {~} 2 \left[ t_k(t_k-1) - t_{k+1}^2\theta_k\nu_k \right] \iprods{ y^k - z^{k+1}, z^{k+1} - z^k }.
\end{array} 
\end{equation*}
From $\Tc_{[1]}$, \eqref{eq:AEG4NI_potential_func}, and  $z^{k+1} - y^{k+1} = -\theta_k(z^{k+1} - z^k) - \nu_k(y^k - z^{k+1})$ in \eqref{eq:AEG4NI},  we can further derive
\begin{equation}\label{eq:AEG4NI_proof3} 
\hspace{-0.0ex}
\arraycolsep=0.2em
\begin{array}{lcl}
\Pc_k - \Pc_{k+1} &= &  a_k\norms{w^{k-1}}^2 - a_{k+1}\norms{w^k}^2  +  \left[ (t_k-1)^2 - t_{k+1}^2\theta_k^2 \right] \norms{z^{k+1} - z^k }^2 + (t_k^2  - \nu_k^2t_{k+1}^2)\norms{y^k - z^{k+1} }^2 \vspace{1ex}\\
&& + {~} b_k\iprods{w^k - w^{k-1}, z^{k+1} - z^k}  + b_{k+1}\iprods{\nu_kw^k - \theta_kw^{k-1}, y^k - z^{k+1}} \vspace{1ex}\\
&& + {~} \left( b_{k+1}\theta_k - b_k\right) \big[ \iprods{w^k, z^{k+1} - z^k}  + \iprods{w^{k-1}, y^k - z^{k+1}} \big]  \vspace{1ex}\\
&& + {~} 2(t_k - 1 - t_{k+1}\theta_k)\iprods{z^{k+1} - z^k, z^{k+1} - x^{\star}} + 2(t_k - t_{k+1}\nu_k)\iprods{y^k - z^{k+1}, z^{k+1} - x^{\star}} \vspace{1ex}\\
&& + {~} 2 \left[ t_k(t_k-1) - t_{k+1}^2\theta_k\nu_k \right] \iprods{ y^k - z^{k+1}, z^{k+1} - z^k}.
\end{array} 
\hspace{-4.0ex}
\end{equation}
We first choose  $t_k $, $\nu_k$, $\theta_k$, and $b_k$ such that
\begin{equation}\label{eq:AEG4NI_para_cond1}
\arraycolsep=0.2em
\begin{array}{lclclcl}
t_k - t_{k+1}\nu_k & = &  0, \ \ &  & t_k(t_k-1) - \nu_k\theta_kt_{k+1}^2 & = & 0, \vspace{1ex}\\
t_k - 1  - t_{k+1} \theta_k & =&  0, \ &  \text{and} \ &  b_{k+1}\theta_k - b_k & = &  0.
\end{array}
\end{equation}
These conditions lead to $\theta_k = \frac{t_k - 1}{t_{k+1}}$ and $\nu_k = \frac{t_k}{t_{k+1}}$ as in \eqref{eq:AEG4NI_para_choice}, and $b_{k+1} := \frac{b_k}{\theta_k} = \frac{b_kt_{k+1}}{t_k-1}$.

Now, using \eqref{eq:AEG4NI_para_cond1}, \eqref{eq:AEG4NI_proof3} reduces to
\begin{equation}\label{eq:AEG4NI_proof5} 
\arraycolsep=0.2em
\begin{array}{lcl}
\Pc_k - \Pc_{k+1} &= &  a_k\norms{w^{k-1}}^2 - a_{k+1}\norms{w^k}^2 + b_k\iprods{w^k - w^{k-1}, z^{k+1} - z^k} \vspace{1ex}\\
&& + {~} b_{k+1} \iprods{\nu_k w^k - \theta_kw^{k-1}, y^k - z^{k+1} }.
\end{array} 
\end{equation}
By the $\rho$-co-hypomonotonicity of $\Phi$ and $z^{k+1} = x^k - \gamma w^k$ from \eqref{eq:AEG4NI}, we have $\iprods{w^k - w^{k-1}, z^{k+1} - z^k} = \iprods{w^k - w^{k-1}, x^k - x^{k-1}} - \gamma\norms{w^k - w^{k-1}}^2 \geq -(\rho + \gamma)\norms{w^k - w^{k-1}}^2$.
Therefore, we obtain
\begin{equation}\label{eq:AEG4NI_proof7} 
\arraycolsep=0.2em
\begin{array}{lcl}
\iprods{w^k - w^{k-1}, z^{k+1} - z^k} &\geq & - (\gamma + \rho)\left[ \norms{w^k}^2 + \norms{w^{k-1}}^2 - 2\iprods{w^k, w^{k-1}} \right].
\end{array}
\end{equation}
Since $\hat{w}^k := Fy^k + \xi^k$, we have $Fy^k - Fx^k = \hat{w}^k - w^k$.
Using this relation and \eqref{eq:AEG4NI}, we get $\frac{\hat{\eta}_k}{\eta} w^{k-1} + \frac{1}{\eta}(y^k - x^k) - \hat{w}^k  = 0$, leading to $x^k - y^k = \hat{\eta}_k w^{k-1} - \eta \hat{w}^k$.
Combining this expression and the second line of \eqref{eq:AEG4NI}, we have $y^k - z^{k+1} = \gamma w^k + \eta \hat{w}^k -  \hat{\eta}_k w^{k-1}$, leading to
\begin{equation}\label{eq:AEG4NI_proof6} 
\arraycolsep=0.2em
\begin{array}{lcl}
\iprods{\nu_k w^k - \theta_kw^{k-1}, y^k - z^{k+1} } &= & \iprods{\nu_k w^k - \theta_kw^{k-1}, \gamma w^k + \eta \hat{w}^k - \hat{\eta}_kw^{k-1}} \vspace{1ex}\\
&= &  \nu_k\gamma \norms{w^k}^2 +  \theta_k\hat{\eta}_k \norms{w^{k-1}}^2  + \nu_k\eta \iprods{w^k, \hat{w}^k} \vspace{1ex}\\
&& - {~} \eta \theta_k\iprods{w^{k-1}, \hat{w}^k} - (\nu_k\hat{\eta}_k  + \gamma \theta_k) \iprods{w^{k-1}, w^k}.
\end{array}
\end{equation}
Substituting \eqref{eq:AEG4NI_proof7} and \eqref{eq:AEG4NI_proof6} into \eqref{eq:AEG4NI_proof5}, and noting that $b_k = b_{k+1}\theta_k$, we can prove that
\begin{equation}\label{eq:AEG4NI_proof9} 
\arraycolsep=0.2em
\begin{array}{lcl}
\Pc_k - \Pc_{k+1} &\geq &  \left[  a_k + b_k\left(\hat{\eta}_k - \gamma - \rho \right) \right]  \norms{w^{k-1}}^2  +  \left[ b_k\left(\frac{\gamma \nu_k}{\theta_k} - \gamma - \rho\right)  - a_{k+1} \right] \norms{w^k}^2 \vspace{1ex}\\
&& + {~} \eta b_k\left(\frac{\nu_k}{\theta_k} - 1\right)\iprods{w^k, \hat{w}^k}  +  \eta b_k  \iprods{w^k - w^{k-1}, \hat{w}^k }  -  b_k \left( \frac{\hat{\eta}_k\nu_k}{\theta_k}  - \gamma - 2\rho \right)  \iprods{w^{k-1}, w^k}.
\end{array} 
\end{equation}
Now, by the Lipschitz continuity of $F$, we have $\norms{\hat{w}^k - w^k}^2 = \norms{Fy^k - Fx^k}^2 \leq L^2\norms{x^k - y^k}^2 = L^2\norms{\eta\hat{w}^k - \hat{\eta}_k w^{k-1}}^2$.
Expanding this expression and rearranging terms, we can deduce that
\begin{equation*} 
\arraycolsep=0.2em
\begin{array}{lcl}
0 & \geq & \norms{w^k}^2 + (1 - L^2\eta^2)\norms{\hat{w}^k}^2 - 2\left(1 - L^2\eta\hat{\eta}_k \right)\iprods{w^k, \hat{w}^k} -  2L^2\eta\hat{\eta}_k \iprods{w^k - w^{k-1}, \hat{w}^k } - L^2\hat{\eta}_k^2 \norms{w^{k-1}}^2. 
\end{array}
\end{equation*}
Multiplying this expression by $\frac{b_k}{2L^2\hat{\eta}_k}$ and adding the result to \eqref{eq:AEG4NI_proof9}, we get
\begin{equation}\label{eq:AEG4NI_proof10} 
\arraycolsep=0.2em
\begin{array}{lcl}
\Pc_k - \Pc_{k+1} &\geq &  \left[ a_k + \frac{b_k}{2} \left( \hat{\eta}_k - 2\gamma - 2\rho \right) \right] \norms{w^{k-1}}^2 + \frac{b_k(1 - L^2\eta^2)}{2L^2\hat{\eta}_k}  \norms{\hat{w}^k}^2 \vspace{1ex}\\
&& + {~} \left[ b_k\left(\frac{\gamma \nu_k}{\theta_k} + \frac{1}{2L^2\hat{\eta}_k} - \gamma - \rho\right) - a_{k+1} \right] \norms{w^k}^2  \vspace{1ex}\\
&& + {~} b_k \left( \frac{\eta\nu_k}{\theta_k}  -  \frac{1}{L^2\hat{\eta}_k}  \right) \iprods{w^k, \hat{w}^k} -   b_k( \frac{\hat{\eta}_k\nu_k}{\theta_k} - 2\rho - \gamma )  \iprods{w^{k-1}, w^k}.
\end{array} 
\end{equation}
Now we choose $\eta := \gamma + 2\rho$ and $\hat{\eta}_k := \frac{\eta\theta_k}{\nu_k} = \frac{\eta(t_k-1)}{t_k}$ as in \eqref{eq:AEG4NI_para_choice}.
Then, using $\nu_k = \frac{t_k}{t_{k+1}}$, $\theta_k = \frac{t_k-1}{t_{k+1}}$, $b_{k+1} = \frac{b_kt_{k+1}}{t_k-1}$, and $\eta := \gamma + 2\rho$, we can further bound \eqref{eq:AEG4NI_proof10}  as
\begin{equation}\label{eq:AEG4NI_proof10_v1} 
\arraycolsep=0.2em
\begin{array}{lcl}
\Pc_k - \Pc_{k+1} &\geq &  \left( a_k - \frac{b_k(\gamma t_k + \eta)}{2t_k} \right) \norms{w^{k-1}}^2 + \left[ \frac{b_k\left(  \gamma t_k + \gamma + \eta \right) }{2(t_k-1)}  - a_{k+1} \right] \norms{w^k}^2 +  \frac{b_k(1 - L^2\eta^2)t_k}{2L^2\eta(t_k-1)} \norms{w^k - \hat{w}^k}^2. 
\end{array} 
\end{equation}
Next, if we assume that
\begin{equation}\label{eq:AEG4NI_proof10_cond} 
1 - L^2\eta^2  \geq  0, \quad  a_k - \frac{b_k(\gamma t_k + \eta)}{2t_k}  \geq  0, \quad \text{and}\quad  \frac{b_k\left(  \gamma t_k + \gamma + \eta \right) }{2(t_k-1)}  - a_{k+1}  \geq  0,
\end{equation}
then \eqref{eq:AEG4NI_proof10_v1} reduces to $\Pc_k \geq \Pc_{k+1}$ for all $k\geq 0$.

To guarantee \eqref{eq:AEG4NI_proof10_cond}, we note that the first condition of \eqref{eq:AEG4NI_proof10_cond} is equivalent to $L\eta = L(\gamma + 2\rho) \leq 1$.
If $2L\rho < 1$, then we can always choose $0 < \gamma \leq \frac{1}{L} - 2\rho$ such that $1 - L^2\eta^2 \geq 0$.
This is the first condition in Theorem~\ref{th:AEG4NI_convergence}.
If we choose $a_k := \frac{b_k(\gamma t_k + \eta)}{2t_k}$, then the second condition of \eqref{eq:AEG4NI_proof10_cond} automatically holds.
The third condition of \eqref{eq:AEG4NI_proof10_cond} becomes  $a_{k+1} \leq \frac{ b_{k+1}( \gamma t_{k + 1} + \eta ) }{2 t_{k+1}}$ due to $b_{k+1} = \frac{b_kt_{k+1}}{t_k-1}$ and $t_k = k + 2$.
By the choice of $a_k$, this condition holds with equality.
Moreover, for $t_k := k+2$, we get $b_k := \frac{b_0(k+1)(k+2)}{2}$ for any $b_0 > 0$, and thus $a_k =  \frac{b_k(\gamma t_k + \eta)}{2t_k} = \frac{b_0(k+1)(\gamma k + 3\gamma + 2\rho )}{4}$.

Utilizing $z^k = x^{k-1} - \gamma w^{k-1}$ from\eqref{eq:AEG4NI} and $\iprods{w^{k-1}, x^{k-1} - x^{\star}} \geq -\rho\norms{w^{k-1}}^2$ due to the $\rho$-co-hypomonotonicity of $\Phi$, we can derive that
\begin{equation}\label{eq:P_lowerbound} 
\arraycolsep=0.2em
\begin{array}{lcl}
\Pc_k  & =   & \norms{z^k + t_k(y^k - z^k) - x^{\star} - \frac{b_k}{2t_k}w^{k-1}}^2 + \left( a_k - \frac{b_k^2}{4t_k^2} -  \frac{\gamma b_k}{t_k} \right) \norms{w^{k-1}}^2 + \frac{b_k}{t_k}\iprods{w^{k-1}, x^{k-1} - x^{\star}} \vspace{1ex}\\  
& \geq & \norms{z^k + t_k(y^k - z^k) - \frac{b_k}{2t_k}w^{k-1} -  x^{\star} }^2  +  \left( a_k - \frac{b_k^2}{4t_k^2} - \frac{(\gamma + \rho) b_k}{t_k}  \right) \norms{w^{k-1}}^2.
\end{array}
\end{equation}
Finally, since $\Pc_{k+1} \leq \Pc_k$ as shown above, by induction we have $\Pc_k \leq \Pc_0 = a_0\norms{w^{-1}}^2 + b_0\iprods{w^{-1}, z^0 - y^0} + \norms{z^0 + t_0(y^0 - z^0) - x^{\star}}^2 = \norms{y^0 - x^{\star}}^2$ due to $z^0 = y^0$.
Combining this expression and \eqref{eq:P_lowerbound}, and then using the explicit form of $a_k$ and $b_k$, we can deduce that
\begin{equation*} 
\arraycolsep=0.2em
\begin{array}{lcl}
\norms{y^0 - x^{\star}}^2 &\geq & \Pc_k \geq \left( a_k - \frac{b_k^2}{4t_k^2} - \frac{(\gamma + \rho) b_k}{t_k}  \right) \norms{w^{k-1}}^2 = \frac{ b_0(4\gamma - b_0)(k+1)^2}{16}  \norms{w^{k-1}}^2.
\end{array} 
\end{equation*}
If we set $b_0 := 2\gamma$, then this estimate reduces to $ \norms{w^k}^2 \leq \frac{4\norms{y^0 - x^{\star}}^2}{\gamma^2(k+2)^2}$, which is exactly \eqref{eq:AEG4NI_convergence}.
\end{proof}
%%%% End of proof.

%%%%%%%%%%%%%%%%%%%%%%%%%%%%%%%%%%%%%%%%%%%%
%%% 8.2. Nesterov's accelerated past-extragradient method for (NI).
%%%%%%%%%%%%%%%%%%%%%%%%%%%%%%%%%%%%%%%%%%%%
\beforesubsec
\subsection{Nesterov's accelerated past-extragradient method for \eqref{eq:NI}}\label{subsec:APEG}
\aftersubsec
\textbf{The algorithm.}
Alternative to Nesterov's accelerated extragradient method \eqref{eq:AEG4NI}, \cite{tran2022connection,tran2023extragradient}  also develop Nesterov's accelerated past-extragradient methods to solve \eqref{eq:NI}.
We now present this method as follows.
Starting from $x^0\in\dom{\Phi}$, we set $\hat{w}^{-1} := 0$ and $z^0 := y^0$, and at each iteration $k\geq 0$, we update
\begin{equation}\label{eq:APEG4NI}
\arraycolsep=0.2em
\left\{\begin{array}{lcl}
x^k & \in & J_{\eta T}\left(y^k - \eta Fy^k +  \hat{\eta}_k\hat{w}^{k-1}\right), \vspace{1ex}\\
\hat{w}^k & := & \frac{1}{\eta}( y^k - x^k + \hat{\eta}_k \hat{w}^{k-1}), \vspace{1ex}\\
z^{k+1} &:= & x^k - \gamma \hat{w}^k, \vspace{1ex}\\
y^{k+1} &:= & z^{k+1} + \theta_k(z^{k+1} - z^k)  +  \nu_k(y^k - z^{k+1}), 
\end{array}\right.
\tag{APEG}
\end{equation}
where $\theta_k$, $\nu_k$, $\eta$, and $\gamma$ are given parameters, determined later.
Compared to \eqref{eq:AEG4NI}, we have replaced $Fx^k$ by $Fy^{k-1}$ in \eqref{eq:APEG4NI} to save one evaluation of $F$.
Now, if we eliminate $z^k$ from \eqref{eq:APEG4NI}, then we obtain
\begin{equation}\label{eq:APEG4NI_impl}
\arraycolsep=0.2em
\left\{\begin{array}{lcl}
x^k & \in & J_{\eta T}\left(y^k - \eta Fy^k +  \hat{\eta}_k\hat{w}^{k-1}\right), \vspace{1ex}\\
y^{k+1} &:= &  x^k + \theta_k(x^k - x^{k-1}) + \hat{\beta}_k (y^k - x^k) + \tilde{\beta}_k\hat{w}^{k-1}, \vspace{1ex}\\
\hat{w}^k & := & \frac{1}{\eta}( y^k - x^k + \hat{\eta}_k \hat{w}^{k-1}),
\end{array}\right.
\end{equation}
where $\hat{\beta}_k := \nu_k - \frac{\gamma}{\eta}(1 + \theta_k - \nu_k)$, and $\tilde{\beta}_k :=  \gamma\theta_k - \frac{\gamma\hat{\eta}_k}{\eta}(1 + \theta_k - \nu_k)$.
Clearly, if $\hat{\eta}_k  = \hat{\beta}_k = \tilde{\beta}_k = 0$, and $\theta_k = 1$, then we obtain the reflected-forward-backward splitting method \eqref{eq:RFBS4NI} from \cite{cevher2021reflected,malitsky2015projected}.
Hence, we can view \eqref{eq:APEG4NI_impl} as an accelerated variant of \eqref{eq:RFBS4NI}

%%%%%%%%%%%%%%%%%%%%%%%%%%%%%%%%%%%%%%
%%% 8.2.b. Convergence analysis
%%%%%%%%%%%%%%%%%%%%%%%%%%%%%%%%%%%%%%
\vspace{1ex}
\noindent\textbf{Convergence analysis.}
To analyze the convergence of  \eqref{eq:APEG4NI}, we use the following potential function:
\begin{equation}\label{eq:APEG4NI_potential_func}
\arraycolsep=0.2em
\begin{array}{lcl}
\hat{\Pc}_k &:= & a_k\norms{w^{k-1}}^2 + b_k\iprods{w^{k-1}, z^k - y^k} + \norms{z^k + t_k(y^k - z^k) - x^{\star}}^2 +  c_k\norms{w^{k-1} - \hat{w}^{k-1}}^2.
\end{array}
\end{equation}
where $w^k := Fx^k + \xi^k$, $\hat{w}^k := Fy^{k-1} + \xi^k$ for $\xi^k \in Tx^k$,  $a_k > 0$, $b_k > 0$, $c_k > 0$, and $t_k > 0$ are given parameters, determined later.
Now, we can state the convergence of \eqref{eq:APEG4NI} in the following theorem.

%%%% Theorem 4.1.
\begin{theorem}\label{th:APEG4NI_convergence}
Suppose that $\Phi$ in \eqref{eq:NI} is $\rho$-co-hypomonotone, $F$ is $L$-Lipschitz continuous such that $8\sqrt{3}L\rho <  1$, $x^{\star} \in \zer{\Phi}$, $\dom{J_{\eta T}} = \R^p$, and $\ran{J_{\eta T}} \subseteq\dom{F}=\R^p$.
Let $\gamma > 0$ be such that $16L^2 \left[ 3(3\gamma + 2\rho)^2 + \gamma(2\gamma + \rho) \right] \leq 1$, which always exists, and $\sets{(x^k, y^k, z^k)}$ be generated by \eqref{eq:APEG4NI}  using
\begin{equation}\label{eq:APEG4NI_para_choice}
t_k := k + 2, \ \ \eta := 2(3\gamma + 2\rho), \ \ \hat{\eta}_k = \frac{(t_k-1)\eta}{t_k},  \ \ \theta_k := \frac{t_k-1}{t_{k+1}}, \ \text{and}  \ \ \nu_k := \frac{t_k}{t_{k+1}}.
\end{equation}
Then, for $k\geq 0$, the following bound holds:
\begin{equation}\label{eq:APEG4NI_convergence}
\norms{Fx^k + \xi^k}^2 \leq \frac{4\norms{y^0 - x^{\star}}^2}{\gamma^2(k+2)(k + 4)}, \quad\text{where} \quad \xi^k \in Tx^k.
\end{equation}
Consequently, we have the last-iterate convergence rate as $\norms{Fx^k + \xi^k} = \BigO{\frac{1}{k}}$. 
\end{theorem}

%%%% The proof of Lemma 3.1.
\begin{proof}
Similar to the proof of \eqref{eq:AEG4NI_proof5}  from Theorem~\ref{th:AEG4NI_convergence}, but using \eqref{eq:APEG4NI}, \eqref{eq:APEG4NI_potential_func}, and \eqref{eq:APEG4NI_para_choice},  we get
\begin{equation}\label{eq:APEG4NI_proof5} 
\hspace{-2ex}
\arraycolsep=0.2em
\begin{array}{lcl}
\hat{\Pc}_k - \hat{\Pc}_{k+1} &= &   a_k\norms{w^{k-1}}^2 - a_{k+1}\norms{w^k}^2  + c_k\norms{w^{k-1} - \hat{w}^{k-1}}^2 - c_{k+1}\norms{w^k - \hat{w}^k}^2 \vspace{1ex}\\
&& + {~} b_k\iprods{w^k - w^{k-1}, z^{k+1} - z^k}  + b_{k+1} \iprods{\nu_k w^k - \theta_kw^{k-1}, y^k - z^{k+1}}.
\end{array} 
\hspace{-2ex}
\end{equation}
Since $\iprods{w^k - w^{k-1}, x^k - x^{k-1}} \geq -\rho\norms{w^k - w^{k-1}}^2$ due to the $\rho$-co-hypomonotonicity of $\Phi$, $z^{k+1} = x^k - \gamma\hat{w}^k = x^k - \gamma w^k + \gamma(w^k - \hat{w}^k)$ from \eqref{eq:APEG4NI}, and both the Cauchy-Schwarz and Young inequalities, we can derive
\begin{equation}\label{eq:APEG4NI_proof6} 
\hspace{-0.2ex}
\arraycolsep=0.2em
\begin{array}{lcl}
\iprods{w^k - w^{k-1}, z^{k+1} - z^k}  &= & \iprods{w^k - w^{k-1}, x^k - x^{k-1}} -  \gamma \iprods{w^k - w^{k-1},  \hat{w}^k - \hat{w}^{k-1}} \vspace{1ex}\\
& \geq &  \gamma\iprods{w^k - w^{k-1}, (w^k - \hat{w}^k) - (w^{k-1} - \hat{w}^{k-1})} -  (\gamma + \rho ) \norms{w^k - w^{k-1}}^2 \vspace{1ex}\\
& \geq & -\left(2\gamma + \rho \right) \norms{w^k - w^{k-1}}^2 -  \frac{\gamma}{2}\norms{w^k - \hat{w}^k}^2 - \frac{\gamma}{2}\norms{w^{k-1} - \hat{w}^{k-1}}^2.
\end{array}
\hspace{-4ex}
\end{equation}
Now, combining $x^k = y^k - \eta \hat{w}^k +    \hat{\eta}_k \hat{w}^{k-1}$ from \eqref{eq:APEG4NI} and its third line, we obtain $y^k - z^{k+1} = (\gamma  + \eta)\hat{w}^k  - \hat{\eta}_k \hat{w}^{k-1} = \gamma w^k + \eta \hat{w}^k - \hat{\eta}_kw^{k-1} + \gamma(\hat{w}^k - w^k) - \hat{\eta}_k(\hat{w}^{k-1} - w^{k-1})$.
Using this expression, and both the Cauchy-Schwarz and Young inequalities again, for any $\beta > 0$, we can prove that
\begin{equation}\label{eq:APEG4NI_proof7} 
\arraycolsep=0.2em
\begin{array}{lcl}
\Tc_{[3]} &:= & \iprods{\nu_k w^k - \theta_kw^{k-1}, y^k - z^{k+1}} \vspace{1ex}\\
&= & \iprods{\nu_k w^k - \theta_kw^{k-1}, \gamma w^k + \eta \hat{w}^k - \hat{\eta}_kw^{k-1} } + \iprods{\nu_k w^k - \theta_kw^{k-1}, \gamma(\hat{w}^k - w^k) - \hat{\eta}_k(\hat{w}^{k-1} - w^{k-1})} \vspace{1ex}\\
& \geq & \gamma\nu_k \norms{w^k}^2 + \hat{\eta}_k\theta_k\norms{w^{k-1}}^2  - (\hat{\eta}_k\nu_k + \gamma\theta_k)\iprods{w^k, w^{k-1}} - \eta\theta_k\iprods{\hat{w}^k, w^{k-1}} +  \eta\nu_k\iprods{w^k, \hat{w}^k} \vspace{1ex}\\
&& - {~} \frac{\beta }{2\nu_k}\norms{\nu_k w^k - \theta_kw^{k-1}}^2  - \frac{\gamma^2 \nu_k}{\beta} \norms{w^k - \hat{w}^k}^2 - \frac{\hat{\eta}_k^2\nu_k}{\beta }\norms{w^{k-1} - \hat{w}^{k-1}}^2.
\end{array}
\end{equation}
Expanding \eqref{eq:APEG4NI_proof6} and \eqref{eq:APEG4NI_proof7}, and then substituting their results  into \eqref{eq:APEG4NI_proof5} with $b_{k+1}\theta_k = b_k$ from \eqref{eq:AEG4NI_para_cond1}, we can derive that
\begin{equation}\label{eq:APEG4NI_proof8} 
\arraycolsep=0.1em
\begin{array}{lcl}
\hat{\Pc}_k - \hat{\Pc}_{k+1} &\geq &     \left[ c_k - b_k \left( \frac{\gamma}{2} + \frac{\hat{\eta}_k^2\nu_k}{\beta\theta_k} \right) \right] \norms{w^{k-1} - \hat{w}^{k-1}}^2 -  \left[ c_{k+1} + b_k \left( \frac{\gamma}{2} + \frac{\gamma^2\nu_k}{\beta\theta_k}\right)  \right] \norms{w^k - \hat{w}^k}^2 \vspace{1ex}\\
&& + {~} \left[ a_k - b_k\left(2\gamma + \rho + \frac{\beta\theta_k}{2\nu_k} - \hat{\eta}_k\right) \right] \norms{w^{k-1}}^2 +  \left[ b_k\left( \frac{(2\gamma - \beta) \nu_k}{2\theta_k}  - 2\gamma - \rho \right) - a_{k+1} \right] \norms{w^k}^2  \vspace{1ex}\\
&& + {~} b_k\left( 3\gamma + 2\rho + \beta -  \frac{\nu_k\hat{\eta}_k}{\theta_k}  \right) \iprods{w^k, w^{k-1}} + \eta b_k \iprods{\hat{w}^k, w^k - w^{k-1}} +  \eta b_k \left(\frac{\nu_k}{\theta_k} - 1\right) \iprods{w^k, \hat{w}^k}.
\end{array} 
\end{equation}
Next, by the Lipschitz continuity of $F$, we can easily show that $\norms{w^k - \hat{w}^k}^2 = \norms{Fx^k - Fy^k}^2 \leq L^2\norms{x^k - y^k}^2 = L^2\norms{\eta \hat{w}^k -  \hat{\eta}_k \hat{w}^{k-1}}^2$.
Hence, for any $\omega > 0$, by Young's inequality, this expression leads to $0 \geq \omega\norms{w^k - \hat{w}^k}^2 + \norms{w^k - \hat{w}^k}^2 - 2(1+\omega)L^2 \norms{ \eta \hat{w}^k - \hat{\eta}_kw^{k-1} }^2 - 2(1+\omega)L^2\hat{\eta}_k^2\norms{w^{k-1} - \hat{w}^{k-1} }^2$.
If we set $M := 2(1+\omega)L^2$, then by expanding the last inequality, we get
\begin{equation*} 
\arraycolsep=0.2em
\begin{array}{lcl}
0 & \geq & \omega\norms{w^k - \hat{w}^k}^2 + \norms{w^k}^2 + ( 1 - M \eta^2) \norms{\hat{w}^k}^2 - 2(1 - M\eta\hat{\eta}_k) \iprods{w^k, \hat{w}^k}  \vspace{1ex}\\
&&- {~}   2M\eta \hat{\eta}_k \iprods{\hat{w}^k, w^k - w^{k-1}} -  M\hat{\eta}_k^2 \norms{ w^{k-1}}^2 - M\hat{\eta}_k^2\norms{w^{k-1} - \hat{w}^{k-1}}^2.
\end{array}
\end{equation*}
Multiplying this expression by $\frac{b_k}{2M\hat{\eta}_k}$ and adding the result to \eqref{eq:APEG4NI_proof8}, we arrive at
\begin{equation*} 
\arraycolsep=0.2em
\begin{array}{lcl}
\hat{\Pc}_k - \hat{\Pc}_{k+1} &\geq &  \left[ c_k - b_k \left( \frac{\gamma}{2} + \frac{\hat{\eta}_k^2\nu_k}{\beta\theta_k} + \frac{\hat{\eta}_k}{2}\right)  \right] \norms{w^{k-1} - \hat{w}^{k-1}}^2  +   \left[ b_k  \left(\frac{\omega}{2M\hat{\eta}_k} - \frac{\gamma}{2} - \frac{\gamma^2 \nu_k}{\beta \theta_k}\right) -  c_{k+1} \right] \norms{w^k - \hat{w}^k}^2 \vspace{1ex}\\
&& + {~}  \left[ a_k - b_k\left(2\gamma + \rho + \frac{\beta\theta_k}{2\nu_k} - \frac{\hat{\eta}_k}{2}\right) \right] \norms{w^{k-1}}^2 + \frac{(1-M\eta^2)b_k}{2M\hat{\eta}_k}\norms{\hat{w}^k}^2 \vspace{1ex}\\
&& + {~} \left[ b_k\left( \frac{1}{2M\hat{\eta}_k} + \frac{(2\gamma - \beta) \nu_k}{2\theta_k}  - 2\gamma - \rho \right) - a_{k+1} \right] \norms{w^k}^2    - b_k\left( \frac{1}{M\hat{\eta}_k} -  \frac{\eta \nu_k}{\theta_k} \right) \iprods{w^k, \hat{w}^k} \vspace{1ex}\\
&& + {~} b_k\left( 3\gamma +  2\rho  + \beta - \frac{\nu_k\hat{\eta}_k}{\theta_k}  \right) \iprods{w^k, w^{k-1}}.
\end{array} 
\end{equation*}
If we choose $\beta := 3\gamma + 2\rho$, and  $\eta := 2( 3\gamma + 2\rho) = 2\beta$ and  $\hat{\eta}_k := \frac{\eta \theta_k}{\nu_k}$, then  $3\gamma + 2\rho + \beta  - \frac{\nu_k\hat{\eta}_k}{\theta_k} = 0$.
Moreover, we can simplify the last expression as
\begin{equation}\label{eq:APEG4NI_proof8_c} 
\arraycolsep=0.2em
\begin{array}{lcl}
\hat{\Pc}_k - \hat{\Pc}_{k+1} &\geq &   \left[ b_k\left( \frac{(\eta + 4\gamma ) \nu_k}{4\theta_k} - 2\gamma - \rho \right) - a_{k+1} \right] \norms{w^k}^2 +  \left[ a_k - b_k \left(  2\gamma + \rho  -  \frac{\eta\theta_k}{4\nu_k}\right)   \right] \norms{w^{k-1}}^2 \vspace{1ex}\\
&& + {~}  \left[ \frac{b_k}{2\theta_k} \left( \frac{(1-2L^2\eta^2)\nu_k}{2L^2\eta} - \gamma\theta_k - \frac{4\gamma^2\nu_k}{\eta}\right) - c_{k+1} \right] \norms{w^k - \hat{w}^k}^2 \vspace{1ex}\\
&& + {~}  \left[ c_k - \frac{b_k}{2}\left( \gamma +  \frac{5\eta\theta_k}{\nu_k} \right) \right] \norms{w^{k-1} - \hat{w}^{k-1}}^2.
\end{array} 
\end{equation}
Now, if we assume that $1 - M\eta^2 \geq 0$, and
\begin{equation}\label{eq:APEG4NI_proof9_para_cond2}  
\arraycolsep=0.2em
\begin{array}{lcllcl}
a_k -   b_k \left(  2\gamma + \rho   -  \frac{\eta\theta_k}{4\nu_k}\right) & \geq & 0,  \quad  & b_k\left( \frac{(\eta + 4\gamma) \nu_k}{4\theta_k} - 2\gamma - \rho \right) - a_{k+1} & \geq & 0, \vspace{1ex}\\
c_k - \frac{b_k}{2}\left( \gamma + \frac{5\eta\theta_k}{\nu_k} \right) \ & \geq & 0,  \quad\text{and}\quad &  \frac{b_k}{2\theta_k}\left( \frac{(1-2L^2\eta^2)\nu_k}{2L^2\eta}  - \gamma\theta_k - \frac{4\gamma^2\nu_k}{\eta} \right) - c_{k+1}  & \geq & 0,
\end{array}
\end{equation}
then \eqref{eq:APEG4NI_proof8_c} reduces to $\hat{\Pc}_{k+1} \leq \hat{\Pc}_k$ for all $k\geq 0$.

Our next step is to show that the update rules in \eqref{eq:APEG4NI_para_choice} guarantee $1 - M\eta^2 \geq 0$ and \eqref{eq:APEG4NI_proof9_para_cond2}.
First, let us choose $a_k :=  b_k \left(  2\gamma + \rho   -  \frac{\eta \theta_k}{4\nu_k}\right) = \frac{b_k( 2\gamma t_k + \eta) }{4t_k}$.
Then, the first  and second conditions of \eqref{eq:APEG4NI_proof9_para_cond2} automatically hold.
The two last conditions of \eqref{eq:APEG4NI_proof9_para_cond2} hold if $\frac{1-2L^2\eta^2}{2L^2\eta} \geq 5\eta + 2\gamma + \frac{4\gamma^2}{\eta}$.
If we choose $\omega := 5$, then using $M = 2(1+\omega)L^2$, this condition is equivalent to $16L^2 \left[ 3(3\gamma + 2\rho)^2 + \gamma(2\gamma + \rho) \right] \leq 1$.
Clearly, if $8\sqrt{3}L\rho < 1$, then we can always find $\gamma > 0$ such that the last condition is satisfied. 
In addition, the condition $1 - M\eta^2 \geq 0$ is equivalent to $M\eta^2 = 48L^2(3\gamma + 2\rho)^2 \leq 1$, which automatically holds.

Using $z^k = x^{k-1} - \gamma \hat{w}^{k-1} = x^{k-1} - \gamma w^{k-1} + \gamma(w^{k-1} - \hat{w}^{k-1})$ from \eqref{eq:APEG4NI}, $\iprods{w^{k-1}, x^{k-1} - x^{\star}} \geq -\rho\norms{w^{k-1}}^2$, and $-\iprods{w^{k-1}, \hat{w}^{k-1}} = -\iprods{w^{k-1}, \hat{w}^{k-1} - w^{k-1}} - \norms{w^{k-1}}^2 \geq -\frac{3}{2}\norms{w^{k-1}}^2 - \frac{1}{2}\norms{w^{k-1} - \hat{w}^{k-1}}^2$, we have
\begin{equation}\label{eq:Nes_coPEAG_potential_func}
\arraycolsep=0.2em
\begin{array}{lcl}
\hat{\Pc}_k   & =   & \norms{z^k + t_k(y^k - z^k) - x^{\star} - \frac{b_k}{2t_k}w^{k-1}}^2  + c_k\norms{w^{k-1} - \hat{w}^{k-1}}^2 \vspace{1ex}\\
&& + {~} \left( a_k - \frac{b_k^2}{4t_k^2}\right) \norms{w^{k-1}}^2 + \frac{b_k}{t_k}\iprods{w^{k-1}, x^{k-1} - x^{\star}} -  \frac{\gamma b_k}{t_k}\iprods{w^{k-1}, \hat{w}^{k-1}} \vspace{1ex}\\
& \geq &  \left( a_k - \frac{b_k^2}{4t_k^2} - \frac{(3\gamma + 2\rho) b_k}{2t_k} \right) \norms{w^{k-1}}^2 + \left( c_k - \frac{\gamma b_k}{2t_k} \right) \norms{w^{k-1} - \hat{w}^{k-1}}^2.
\end{array}
\end{equation}
Since $b_k = \frac{b_0(k+1)(k+2)}{2}$ and $a_k = \frac{b_k( 2\gamma t_k + \eta) }{4t_k} = \frac{b_0(k+1)( \gamma k+ 5\gamma + 2\rho) }{4}$, by choosing  $c_k := \frac{b_k(\gamma\nu_k + 5\eta\theta_k)}{2\nu_k}= \frac{b_0(k+1)[ (31\gamma + 20\rho)(k+1) + \gamma ] }{4}$,  we can show from the last inequality that
\begin{equation*} 
\arraycolsep=0.2em
\begin{array}{lcl}
\hat{\Pc}_k &\geq &  \left( a_k - \frac{b_k^2}{4t_k^2} - \frac{(3\gamma + 2\rho) b_k}{2t_k} \right) \norms{w^{k-1}}^2 + \left( c_k - \frac{\gamma b_k}{t_k} \right) \norms{w^{k-1} - \hat{w}^{k-1}}^2 \vspace{1ex}\\
&= & \frac{b_0(k+1)[ (4\gamma - b_0) (k+2) + b_0]}{16}  \norms{w^{k-1}}^2 +  \frac{b_0(k+1)[ (31\gamma + 20\rho)(k+1) - \gamma ] }{4} \norms{w^{k-1} - \hat{w}^{k-1}}^2.
\end{array} 
\end{equation*}
Finally, if we choose $b_0 := 2\gamma$, then $\hat{\Pc}_k \geq \frac{\gamma^2(k+1)(k+3)}{4}\norms{w^{k-1}}^2$.
Since $\hat{\Pc}_k \leq \hat{\Pc}_0 = a_0\norms{w^{-1}}^2 + b_0\iprods{w^{-1}, z^0 - y^0} + \norms{z^0 + t_0(y^0 - z^0) - x^{\star}}^2 +  c_0\norms{w^{-1} - \hat{w}^{-1}}^2$, and $y^0 = z^0$ and $w^{-1} = \hat{w}^{-1} = 0$, we get $\hat{\Pc}_k \leq \hat{\Pc}_0 = \norms{y^0 - x^{\star}}^2$.
Putting these steps together, we can conclude that $\norms{w^{k-1}}^2 \leq \frac{4\norms{y^0 - x^{\star}}^2}{\gamma^2(k+1)(k + 3)}$, which is exactly \eqref{eq:APEG4NI_convergence}.
\end{proof}
%%% End of proof.

%%% Remark 1.
\begin{remark}\label{re:APEG4NI_remark1}
The condition $16L^2 \left[ 3(3\gamma + 2\rho)^2 + \gamma(2\gamma + \rho) \right] \leq 1$ in Theroem~\ref{th:APEG4NI_convergence} covers a range of $\gamma$ by solving a quadratic inequation in $\gamma$. 
Compared \eqref{eq:APEG4NI} to the accelerated extragradient methods in \cite{bot2022fast} for \eqref{eq:NE}, we can see that \cite{bot2022fast} uses all variable stepsizes, while  \eqref{eq:APEG4NI} allows $\gamma$ and $\eta$ to be constant.
However,  \eqref{eq:APEG4NI} only achieves Big-O convergence rates instead of small-O convergence rates as in  \cite{bot2022fast}.
\end{remark}

%%%%%%%%%%%%%%%%%%%%%%%%%%%%%%%%
%%%% 9. Conclusion and further remarks
%%%%%%%%%%%%%%%%%%%%%%%%%%%%%%%%
\beforesec
\section{Conclusion and Further Remarks}\label{sec:conclusion}
\aftersec
In this paper, we have provided a survey of classical and recent results on the sublinear convergence rates of the extragradient (EG) method and its variants. 
We presents the full proofs of all the results discussed in the paper, where many proofs are new in certain aspects.
Classical convergence results of EG-type algorithms typically rely on monotonicity assumptions, while recent developments extend EG-type methods to weak-Minty solutions and co-hypomonotone settings. 
In addition, last-iterate convergence rates have been investigated for several EG variants, though this research remains incomplete. 
Various extensions to stochastic and randomized models have also been studied. 
EG-type methods have been widely applied in machine learning, particularly in GANs, online learning, reinforcement learning, and robust optimization.
These algorithms have shown their efficiency in practice, especially for constant and adaptive stepsize variants.

Accelerated variants of EG have also attracted significant attention, including methods relying on Halpern’s fixed-point iteration and Nesterov’s accelerated techniques. 
While several works have focused on theoretical aspects of EGs such as iteration-complexity and last-iterate convergence rates, the practical performance of accelerated EG variants remains limited and requires further investigation. 
It is still unclear whether accelerated variants of EG can outperform their classical counterparts, which opens up a new research question for our future work.
In addition, establishing tighter convergence rates (e.g., small-o rates) as well as convergence of sequences remains largely open for several variants discussed in this paper. 

EGs have been extensively studied for several decades, with numerous researchers making remarkable contributions to the field. 
The theory, algorithms, and applications of EGs have been expanded to various fields, including economics and machine learning. 
However, given the breadth of the literature on EGs, this paper can only survey a small proportion of recent works on sublinear convergence rates for both non-accelerated and accelerated variants in deterministic settings. 
We have no means to fully cover many other works, including classical and recent developments. 
We hope that this paper will provide a useful starting point for us to continue exploring recent literature on minimax problems and their extensions.
We also wish to survey prominent applications of minimax problems and nonlinear inclusions in different fields.

{
\vspace{2ex}
\noindent\textbf{Data availability.}
The author confirms that this paper does not contain any data.
%all data used in this paper is generated synthetically.
%The method for generating data is also described in the paper.
}

\vspace{2ex}
\noindent\textbf{Acknowledgements.}
This work is  partially supported by the National Science Foundation (NSF), grant no. NSF-RTG DMS-2134107 and the Office of Naval Research (ONR), grant No. N00014-20-1-2088.

%%%% Bibliographies 
\bibliographystyle{plain}
%\bibliography{/Users/quoctd/Dropbox/E-Books/tran_bibtex_new}

\end{document}

%% file: my_preamble.tex
\usepackage{amsthm}
\usepackage{amsmath}
\usepackage{amssymb}
\usepackage{graphicx}
\usepackage{color}
\usepackage{ifpdf}
\usepackage{url}
\usepackage{hyperref}
\usepackage{algorithm}
\usepackage[usenames,dvipsnames]{xcolor}
\usepackage{paralist}
\usepackage{bbding}
\usepackage{pifont}
\usepackage{wasysym}

\usepackage{algorithm}
\usepackage{algpseudocode}
\usepackage{algorithmicx}
\usepackage[ruled,vlined, linesnumbered, algo2e]{algorithm2e}
\usepackage[ruled,vlined]{algorithm2e}

\usepackage[T1]{fontenc} % Use 8-bit encoding that has 256 glyphs
\usepackage[hmarginratio=1:1, top=32mm, columnsep=20pt, margin=2.5cm]{geometry} % Document margins
\usepackage{multicol} % Used for the two-column layout of the document
\usepackage[hang, small,labelfont=bf,up,textfont=it,up]{caption} % Custom captions under/above floats in tables or figures
\usepackage{booktabs} % Horizontal rules in tables
\usepackage{float} % Required for tables and figures in the multi-column environment - they need to be placed in specific locations with the [H] (e.g. \begin{table}[H])
\usepackage{hyperref} % For hyperlinks in the PDF

\theoremstyle{plain}
\newtheorem{theorem}{Theorem}[section]

\newtheorem{lemma}{Lemma}[section]

\theoremstyle{definition}

\newtheorem{remark}{Remark}[section]
\usepackage{todonotes}

\newcommand{\Id}{\mathbb{I}}

\newcommand{\R}{\mathbb{R}}

\newcommand{\Rext}{\R\cup\{+\infty\}}

\newcommand{\set}[1]{\left\{#1\right\}}
\newcommand{\sets}[1]{\{#1\}}

\newcommand{\norms}[1]{\Vert#1\Vert}

\newcommand{\prox}{\mathrm{prox}}

\newcommand{\proj}{\mathrm{proj}}

\newcommand{\dom}[1]{\mathrm{dom}(#1)}
\newcommand{\gra}[1]{\mathrm{gra}(#1)}
\newcommand{\ran}[1]{\mathrm{ran}(#1)}

\newcommand{\zero}[1]{{\boldsymbol{0}}}

\newcommand{\zer}[1]{\mathrm{zer}(#1)}

\newcommand{\mbb}[1]{\mathbb{#1}}
\newcommand{\mcal}[1]{\mathcal{#1}}

\newcommand{\Xc}{\mathcal{X}}

\newcommand{\Hc}{\mathcal{H}}
\newcommand{\Lc}{\mathcal{L}}

\newcommand{\Tc}{\mathcal{T}}

\newcommand{\Cc}{\mathcal{C}}

\newcommand{\Nc}{\mathcal{N}}
\newcommand{\Pc}{\mathcal{P}}
\newcommand{\Vc}{\mathcal{V}}

\newcommand{\iprods}[1]{\langle #1\rangle}

\newcommand{\intx}[1]{\mathrm{int}\left(#1\right)}

\newcommand{\BigO}[1]{\mathcal{O}\left(#1\right)}
\newcommand{\BigOs}[1]{\mathcal{O}\big(#1\big)}

\newcommand{\mytb}[1]{\textbf{#1}}
\newcommand{\myti}[1]{\textit{#1}}

\newcommand{\beforesubsec}{\vspace{-1.5ex}}
\newcommand{\aftersubsec}{\vspace{-1ex}}
\newcommand{\beforesec}{\vspace{-1.25ex}}
\newcommand{\aftersec}{\vspace{-1ex}}